\documentclass
{amsproc}
\usepackage{amsmath}
\usepackage{amssymb}
\usepackage{graphicx}
\usepackage{multirow}
\usepackage{hyperref}
\DeclareGraphicsExtensions{.png,.pdf,.eps}
\usepackage[all,2cell]{xy}
\usepackage{tikz}
\usepackage{enumerate}
\usetikzlibrary{cd}
\usepackage{float}

\newtheorem{theorem}{Theorem}[section]
\newtheorem*{theorem*}{Theorem}
\newtheorem{lemma}[theorem]{Lemma}
\newtheorem{corollary}[theorem]{Corollary}
\newtheorem{proposition}[theorem]{Proposition}

\theoremstyle{definition}

\newtheorem{definition}[theorem]{Definition}

\newtheorem{remark}[theorem]{Remark}

\newtheorem{set}[theorem]{Setup}
\newtheorem{example}[theorem]{Example}

\def\C{{\mathbb C}}

\def\G{{\mathbb G}}
\def\LG{{\mathbb L}{\mathbb G}}

\def\P{{\mathbb P}}
\def\Q{{\mathbb Q}}
\def\R{{\mathbb R}}

\def\Z{{\mathbb Z}}

\def\cD{{\mathcal D}}
\def\cE{{\mathcal E}}

\def\cL{L}
\def\cM{{\mathcal M}}
\def\cN{{\mathcal{N}}}
\def\cO{{\mathcal{O}}}

\def\cU{{\mathcal U}}
\def\cV{{\mathcal V}}

\def\cY{{\mathcal Y}}
\def\Q{{\mathbb{Q}}}
\def\G{{\mathbb{G}}}

\def\fg{{\mathfrak g}}
\def\fh{{\mathfrak h}}
\def\fp{{\mathfrak p}}

\def\To{T}

\def\operatorname#1{\mathop{\rm #1}\nolimits}

\def\DA{{\rm A}}
\def\DB{{\rm B}}
\def\DC{{\rm C}}
\def\DD{{\rm D}}
\def\DE{{\rm E}}
\def\DF{{\rm F}}
\def\DG{{\rm G}}

\def\Proj{\operatorname{Proj}}
\def\Aut{\operatorname{Aut}}

\def\Exc{\operatorname{Exc}}

\def\Hom{\operatorname{Hom}}

\def\Pic{\operatorname{Pic}}
\def\Hom{\operatorname{Hom}}

\def\codim{\operatorname{codim}}

\def\id{\operatorname{id}}

\def\rank{\operatorname{rank}}
\def\rk{\operatorname{rk}}

\def\deg{\operatorname{deg}}
\def\Bs{\operatorname{Bs}}
\def\det{\operatorname{det}}

\def\qed{\hspace{\fill}$\rule{2mm}{2mm}$}
\def\NE{{\operatorname{NE}}}

\def\Nu{{\operatorname{N_1}}}
\def\NU{{\operatorname{N^1}}}

\newcommand{\cNE}[1]{\overline{\NE(#1)}}

\def\ad{\operatorname{ad}}
\def\aj{\operatorname{ad}}

\def\SP{\operatorname{Sp}}

\def\Proj{\operatorname{Proj}}
\def\gen{\operatorname{gen}}

\newcommand{\pb}{\ar@{}[dr]|{\text{\pigpenfont J}}}
\def\Mo{\operatorname{\hspace{0cm}M}}

\def\ol{\overline}

\newcommand{\shse}[3]{0 ~\ra ~#1~ \lra ~#2~ \lra ~#3~ \ra~ 0}
\makeatletter
\newcommand{\xleftrightarrow}[2][]{\ext@arrow 3359\leftrightarrowfill@{#1}{#2}}
\makeatother
\newcommand{\xdasharrow}[2][->]{
\tikz[baseline=-\the\dimexpr\fontdimen22\textfont2\relax]{
\node[anchor=south,font=\scriptsize, inner ysep=1.5pt,outer xsep=2.2pt](x){#2};
\draw[shorten <=3.4pt,shorten >=3.4pt,dashed,#1](x.south west)--(x.south east);
}}
\newcommand{\hooklongrightarrow}{\lhook\joinrel\longrightarrow}
\newcommand{\hooklongleftarrow}{\longleftarrow\joinrel\rhook}

\title{Small bandwidth  $\C^*$-actions and birational geometry}

\author[Occhetta]{Gianluca Occhetta} \address{Dipartimento di
  Matematica, Universit\`a di Trento, via Sommarive 14 I-38123 Povo di
  Trento (TN), Italy} \thanks{First and third author supported by PRIN project
  ``Geometria delle variet\`a algebriche'' and by MIUR project FFABR.  First, second and fourth author supported by the Department of Mathematics of the University of
  Trento. First, third and fourth author supported by the Polish National
  Science Center project 2013/08/A/ST1/00804. Second author supported by the Grant HA4383/-1 ``ATAG'' by the German Science Agency (DFG) and Thematic Einstein Semester ``Varieties, Polyhedra, Computation'' by the Berlin Einstein Foundation. Second and fourth author supported by the Polish National Science Center project 2016/23/G/ST1/04282.}
\email{gianluca.occhetta@unitn.it, eduardo.solaconde@unitn.it}

\author[Romano]{Eleonora A. Romano}
\address{Instytut Matematyki UW, Banacha 2, 02-097 Warszawa, Poland}
\email{J.Wisniewski@uw.edu.pl, elrom@mimuw.edu.pl}

\author[Sol\'a Conde]{Luis E. Sol\'a Conde}

\author[Wi\'sniewski]{Jaros\l{}aw A. Wi\'sniewski}

\subjclass[2010]{Primary 14L30; Secondary 14E05, 14M17, 14M25} 

\usepackage[textsize=tiny]{todonotes}

\newcommand{\gotodo}[1]{\todo[color=yellow!30]{go: #1}}

\usepackage{enumitem}

\newcommand\ignore[1]{}

\def\Div{\operatorname{Div}}
\def\div{\operatorname{div}}

\DeclareMathOperator{\Cl}{Cl}
\DeclareMathOperator{\HH}{H}
\DeclareMathOperator{\cone}{Cone}

\newcommand\CC{{\mathbb{C}}}
\newcommand\PP{{\mathbb{P}}}
\newcommand\RR{{\mathbb{R}}}
\newcommand\QQ{{\mathbb{Q}}}
\newcommand\ZZ{{\mathbb{Z}}}

\newcommand\ra{{\ \rightarrow\ }}
\newcommand\lra{\longrightarrow}

\newcommand\iso{{\ \cong\ }}

\begin{document}
\begin{abstract} 
In this paper we study smooth projective varieties and polarized pairs with an action of a one dimensional complex torus. As a main tool, we define birational geometric counterparts of these actions, that, under certain assumptions, encode the information necessary to reconstruct them. In particular, we consider some cases of actions of low complexity --measured in terms of two invariants of the action, called bandwidth and bordism rank-- and discuss how they are determined by well known birational transformations, namely Atiyah flips and Cremona transformations.
\end{abstract}
\maketitle

\tableofcontents


\section{Introduction}

\subsection{Background} Let $X$ be a complex variety with an algebraic action of
a one dimensional complex torus $\CC^*$. It is a classical problem in Mumford's
Geometric Invariant Theory to determine subsets of $X$ on which there exists a good
quotient of the action, see e.g.~\cite{BBS}. Upon the choice of such subsets one
obtains different quotients which are birationally equivalent varieties, leading to the natural question of relating birational geometry of quotients with
the original $\CC^*$-action. This phenomenon was placed  in the context of the
Minimal Model Program in the early 1990's, see the works of Thaddeus and Reid,
\cite{ReidToric,ReidFlip,Thaddeus1994,Thaddeus1996}.
Later, around 2000, the actions of $\CC^*$ have been used to prove Oda
Conjecture, \cite{Morelli}, and weak Hironaka Conjecture, \cite{Wlodarczyk}. In
the latter two papers, Morelli and W{\l}odarczyk introduced the notion of {\em
algebraic cobordism} which enabled to understand birational maps of algebraic
varieties in terms of quotients of varieties with $\CC^*$-actions. On the other hand, the concept of an action of an algebraic torus of  higher rank associated to the class group of a variety is used for Mori Dream Spaces, \cite{HuKeel,CoxRings,CasciniLazic,CortiLazic}. In this approach, different small $\QQ$-factorial modifications of a given Mori Dream Space arise as different GIT quotients of the affine spectrum of its Cox ring.

The present paper goes in the opposite direction. Here the link with birational
geometry is used to understand varieties  with a $\CC^*$-action. We note that apart of the situation when the $\CC^*$-action is obtained by
reducing the action of a bigger group as in the case of toric or
homogeneous varieties, only low dimensional cases are well understood. For
example, see a recent paper \cite{CPS} for a
classification of Fano threefolds with infinite automorphism groups, which
includes the cases of Fano 3-folds with $\CC^*$-actions. 

\subsection{Contents of the paper} Let $X$ be a smooth complex projective variety with a
nontrivial algebraic action of a complex torus $H=\CC^*$ with coordinate
$t\in\CC^*$. We choose an ample
bundle $L$ over $X$ and consider a linearization of the action of $H$ on the
pair $(X,L)$.
 By the {\em source} and the {\em sink} of the action we understand the unique
fixed point components $Y_+,Y_-\subset X^H$ such that for a generic point $x\in X$ we
have $\lim_{t\to 0} t^{\pm 1}\cdot x\in Y_\pm$, respectively. The {\em bandwidth} of
the action of $H$ on $(X,L)$ is a non-negative integer which measures the
difference of the weights of the action of $H$ on fibers of $L$ over $Y_+$ and
$Y_-$. We introduce these and other related concepts in Section \ref{sec:prelim}, establishing the notation that we will use throughout the paper, and discussing some basic features of $\C^*$-actions and linearizations.

In Section \ref{sec:Btype_Bordism} we introduce B-type $\C^*$-actions and bordisms.
We say that an action is of {\em B-type} (Definition \ref{def:bord}) if both
$Y_+$ and $Y_-$ are divisors in $X$; note that if an arbitrary action is {\em
equalized}, that is if it has weights $+1$ and $-1$ at the normal bundle of the
source, respectively sink, then blowing them up we get a B-type action (Lemma \ref{lem:setup_bir}); this result has been recently extended to the non-equalized case in \cite[Lemma~4.4]{BaRo}. The
choice of the sink and source of a general orbit of a B-type action yields
a birational map $\psi$ between $Y_-$ and $Y_+$, see Lemma \ref{lem:birational_map_b}.

We say that a B-type action is a {\em bordism} (Definition \ref{def:simpleBaction}) if every $H$-invariant prime divisor $D$ on $X$ is either $Y_+$, $Y_-$ or both $D\cap Y_-$ and $D\cap Y_+$
are divisors in $Y_-$ and $Y_+$. In case of a
bordism, we prove that $Y_-$ and $Y_+$ are isomorphic in codimension $1$ (Lemma \ref{lem:bordism=iso-in-codim1}). The {\em rank} of the
bordism is, by definition, the corank of $\Pic(X)^H\subset\Pic (X)$ minus one,
where $\Pic(X)^H$ denotes the subgroup of line bundles which are trivial on the closures of
orbits of the $H$-action.

Generally, bordisms of small rank and  small bandwidth $\C^*$-actions yield relatively
simple birational maps between $Y_+$ and $Y_-$. In particular, we deal with bandwidth one
actions in Section \ref{sec:BW1}. The main result in that Section states the following (see Theorem \ref{thm:drums} for a more detailed statement):

\begin{theorem*}
Any action of $\C^*$ of bandwidth one on a smooth polarized pair $(X,L)$ with $\rho_X=1$ and $X\neq \P^n$ is completely 
determined by a smooth projective variety $Y$ with $\rho_Y=2$ admitting two projective bundle structures.
\end{theorem*}

We note here that the same conclusion is true for bordisms
of rank zero (see Lemma \ref{lem:rank0bordism}).

Sections \ref{sect:appendix} and \ref{sec:BW2} deal with the case of bordisms of rank one. In particular we will show the close relation among bordisms of rank one and Atiyah flips (Definition \ref{def:Atiyah.assumptions}), in the spirit of the work of Reid, Thaddeus, Morelli and W{\l}odarczyk:

\begin{theorem*}
Let $(X,L)$ be a polarized pair with $X$ smooth, admitting an equalized $\C^*$-action of bandwidth two that is a bordism of rank one. Then the variety $X$ and the action 
are completely determined by the birational map $\psi$ induced between the sink and the source, which is an Atiyah flip.  
\end{theorem*}

Our arguments (see Theorem \ref{flip=>bordism} and Corollary \ref{cor:flip=>bordism} for the precise statements) show how to construct such an action out of an Atiyah flip: 
$$\begin{tikzcd}[cramped] Y_- \arrow[leftrightarrow,
dashed]{rr}\arrow{dr}&&Y_+\arrow{dl}\\ &Y'& \end{tikzcd}$$ 
More concretely, we will show that the variety $X$ in question is obtained as the Atiyah flip of
$\PP^1$-bundles over $Y_-$ and $Y_+$, and it fits into the following diagram
$$\begin{tikzcd}Y_-\arrow[hook]{r}\arrow{dr}&X\arrow{d}&Y_+\arrow[hook']{l}\arrow{dl} \\
&Y'&\end{tikzcd}$$
The fixed point components of the $\CC^*$-action on $X$ which are different from $Y_-$ and $Y_+$ can be identified with the center of the flip in $Y'$.

Subsequently, we show that the action is completely determined by the flip in Theorem \ref{bordism=>unique}. Throughout the proofs we use a number of toric calculations (see Section \ref{sec:BW2}), that are needed for proving the claimed results about bordisms. 
Our results are then completed with the geometric description of some examples of bandwidth two actions on rational homogeneous manifolds associated to short gradings on simple Lie algebras. In the case in which the Picard number of the manifold is one, blowing up the source and sink we get rank one bordisms which fulfill the conditions described above. Thus they are completely determined by the source and sink of the $\CC^*$-action and their normal bundles, see Corollary \ref{cor:adjointBW2}.

Finally, in Section \ref{sec:BW3} we deal with bandwidth three equalized actions for
which the sink and the source are isolated points. After blowing them up, we get B-type actions with sink and source being projective spaces with normal bundle
$\cO(-1)$, and Cremona transformations among these projective spaces. The results contained in the Section can be summarized as follows: 

\begin{theorem*}
Let $(X,L)$ be a polarized pair with $X$ smooth, $\rho_X=1$, admitting an equalized $\C^*$-action of bandwidth three with isolated extremal fixed points. Then this action is completely 
determined by the Cremona transformation $\psi$ induced by it. Moreover, $\psi$ and its inverse are special quadratic Cremona transformations, and the complete list of the varieties $X$ admitting such an action is:
$$\LG(2,5)\simeq\DC_3(3),\,\,\, \G(2,5)\simeq \DA_5(3),\,\,\, S^{15}\simeq \DD_6(6),\,\,\, \DE_7(7).$$
\end{theorem*}

Our arguments show how to construct an equalized bandwidth three action out of a special Cremona transformation. The precise statements that we will prove are Theorems \ref{thm:bandwidth3->Cremona}, \ref{thm:BW3end} (see also Proposition \ref{prop:bw3class}).

In particular, this result completes the classification
of bandwidth three varieties with an equalized action for which sink and source are isolated  points, which was partially established in \cite{RW} 
(see Theorem \ref{thm:bw3}). This theorem has been used in the paper \cite{WORS2} to extend results in \cite{BWW,RW} about the classification of complex contact manifolds.

\subsection{Notation, conventions} We work over $\mathbb{C}$. Let $X$ be a
normal projective variety of arbitrary dimension. Let $\NU(X)$
(respectively $\Nu(X)$) be the real vector space of Cartier divisors
(respectively, 1-cycles on $X$), modulo numerical equivalence.  We denote by
$\rho_{X}:= \dim{\NU(X)}=\dim{\Nu(X)}$ the \textit{Picard number} of $X$, and by $[\cdot]$ the numerical equivalence classes in $\Nu(X)$ and $\NU(X)$. We denote by $\cNE{X}\subset \Nu(X)$ the closure of the convex cone generated by classes of effective curves. Given a line bundle $\cL\in\Pic(X)$, we denote by $\Bs|\cL|$ the base locus of the complete linear system of $\cL$.

A \textit{contraction} $\varphi\colon X\to Y$ is a surjective morphism with connected fibers onto a normal projective variety. It is \textit{elementary} if $\rho_X-\rho_Y=1$.
If $\dim{X}>\dim{Y}$ then $\varphi$ is of \textit{fiber type}, otherwise it is \textit{birational}. Let $\Exc{(\varphi)}$ be the exceptional locus of $\varphi$, i.e., the locus where $\varphi$ is not an isomorphism; if $\varphi$ is birational and $\codim{\Exc{(\varphi)}}\geq 2$ we say that $\varphi$ is {\em small}, otherwise it is {\em divisorial}. The push-forward of 1-cycles defined by $\varphi$ is a surjective linear map $\varphi_{*}\colon \Nu(X)\to \Nu(Y)$. We denote by $\Nu(X/Y)$ its kernel.

Throughout this paper $H:=\C^*$ is a complex torus of dimension one. If $H$ acts on a variety $X$ we will denote by $X^H$ the fixed locus of the action. All the actions considered in the paper are assumed to be nontrivial. 

\medskip

\noindent{\bf Acknowledgements:} The authors would like to thank two anonymous referees for their valuable comments, which helped improving the exposition of this paper. 


\section{Preliminaries}\label{sec:prelim}

In this section we introduce some background material on $H$-actions that we will use later on, as well as some terminology and notation regarding rational homogeneous spaces that we will need to present examples, and to state some results contained in the successive sections.

\subsection{$\C^*$-actions}\label{ssec:Cstar}

Let $X$ be a smooth complex projective variety of dimension $n$, 
and let $H= \C^*$ be an algebraic torus which acts nontrivially on $X$; the fixed locus of the action decomposes in connected components as
$$X^H=\bigsqcup_{i\in I} Y_i,$$ 
and we denote by   $\cY=\{Y_i\}_{i\in I}$ the set of the connected components of $X^H$. By \cite[Main Theorem]{IVERSEN}  the $Y_i$'s are smooth, therefore irreducible.
  
 For every $Y\in \cY$ the $H$-action on $TX_{\mid Y}$ gives a  decomposition
\begin{equation} \label{eq:dectang}
TX_{\mid Y} =T^{+}\oplus T^{0}\oplus T^{-},
\end{equation}
where $T^{+}$, $T^{0}$, $T^{-}$ are  the
subbundles of $TX_{\mid Y}$ where the torus acts with positive,
zero or negative weights, respectively. Then, by local linearization, $T^0=TY$ and
\begin{equation} \label{eq:decnorm}
T^{+}\oplus T^{-}=N_{Y/X}=N^+(Y)\oplus N^-(Y)
\end{equation}
is the decomposition of the normal bundle $N_{Y/X}$ into summands on
which $H$ acts with positive, respectively, negative weights. 
We set
\begin{equation} \label{eq:nudef}
\nu^{+}(Y):=\rank{N^+(Y)}, \qquad  \nu^{-}(Y):=\rank{N^-(Y)}.
\end{equation}

In this paper we will pay special attention to the case in which the action of $H$ on the normal bundles $N_{Y/X}$ is simple, in the following sense (see \cite[Definition~1.6]{RW}):

\begin{definition}
\label{def-equalized}
We say that the $H$-action on $X$ is {\em equalized at a fixed component} $Y\in \cY$ if the torus acts
on $N^+(Y)$ with all weights $+1$ and on $N^-(Y)$ with all weights $-1$. The action is {\em equalized} if it is equalized at each fixed component.
\end{definition}

Note that the closure $C$ of a $1$-dimensional orbit of an $H$-action is a rational curve. The following definition concerns the infinitesimal behaviour of the induced action of $H$ on the  fixed points of $C$.
 
\begin{definition}\label{def:delta}
Let $C=\overline{H\cdot x}$ be an $H$-invariant rational curve in $X$. 
The extension of the $H$-action to the normalization $f:\P^1\to C\subset X$  has exactly two fixed points, that we denote by $z_+,z_-\in \P^1$. We choose the signs so that, if $z\in \P^1$ is a general point,
$$z_{\pm}=\lim_{t^{\pm 1}\to 0}tz.$$ Their images $x_+:=f(z_+),x_-:=f(z_-)\in C$ are called, respectively, the {\it source} and the {\it sink} of the orbit of $x$.
Note that we have a linear action of $H$ on $T_{\P^1,z_+}$, whose weight is denoted by $\delta(z_+)$. Clearly the weight $\delta(z_-)$ of the lifted action on  $T_{\P^1,z_-}$ is equal to $-\delta(z_+)$. We set:
$$
\delta(C):=\delta(z_+).
$$ 
\end{definition}  
  
Since the $H$-invariant curves of the form $\overline{H\cdot x}$ are rational, the uniruledness of $X$ follows. In fact we may state:

\begin{lemma}\label{lem:uniruled}
A smooth projective variety $X$ with an $H$-action is uniruled. In particular, if $\rho_X=1$ then $X$ is a Fano manifold, hence rationally connected.
\end{lemma}

\begin{proof}
It is a known fact (see for instance \cite[Remark~4.2 (4)]{De}) that the existence of a rational curve through a general point of a variety $X$ (defined over an algebraically closed uncountable field) implies its uniruledness. 
If moreover $\rho_X=1$, then $X$ is Fano, hence rationally connected (see \cite[Proposition~5.16]{De}). 
\end{proof}


A fundamental result in the theory of torus actions is the so-called Bia{\l}ynicki-Birula decomposition ({\em BB-decomposition}, for short), that we shall introduce now as presented in \cite[Theorem 4.2, Theorem 4.4]{CARRELL} (see \cite{BB} for the original exposition). For every fixed point component $Y\in \cY$ we define the {\em BB-cells} associated to $Y$ as follows:
$$X^+(Y):=\{x\in X: \lim_{t\rightarrow 0} t x\in Y\}{\rm \ \ and \ \
}X^-(Y):=\{x\in X: \lim_{t\to 0}t^{-1} x\in Y\}.$$
Then we have:

\begin{theorem}
\label{thm:BB_decomposition}
In the situation described above the following hold:
 \begin{itemize}

\item [(1)] $X^{\pm}(Y)$ are locally closed subsets and there are two decompositions
 $$X=\bigsqcup_{Y\in \cY}X^+(Y)=\bigsqcup_{Y\in \cY}X^-(Y).$$
\item [(2)] For every $Y\in\cY$ there are $H$-equivariant isomorphisms $X^+(Y)\simeq N^{+}(Y)$ and $X^-(Y)\simeq N^{-}(Y)$ lifting the natural maps
 $X^\pm(Y)\rightarrow Y$. Moreover, the map $X^\pm(Y)\rightarrow Y$ is algebraic and is a $\CC^{\nu^\pm(Y)}$-fibration. 
\item [(3)] There are decompositions in homology:
 $$H_m(X,\ZZ)\simeq\bigoplus_{Y\in \cY}H_{m-2\nu^+(Y)}(Y,\ZZ)\simeq \bigoplus_{Y\in \cY}H_{m-2\nu^-(Y)}(Y,\ZZ).$$ \end{itemize}
\end{theorem} 

We note that, in general, the morphism $X^\pm(Y)\rightarrow Y$ is 
merely an affine bundle, and not a vector bundle. 

\begin{remark}\label{rem:sinksource}
As a consequence of Theorem \ref{thm:BB_decomposition} (1), there exists a unique component $Y_+\in \cY$ (resp. $Y_-\in \cY$) such that $X^+(Y_+)$ (resp. $X^-(Y_-)$) is a dense open set in $X$. We will call $Y_+$ and $Y_-$ the \textit{source} and the \textit{sink} of the action, and refer to them as the \textit{extremal fixed components} of the action. We will call  \textit{inner components}  the fixed point components which are not extremal.
\end{remark}

Another consequence of Theorem \ref{thm:BB_decomposition} is that the rational connectedness of a projective manifold $X$ is inherited by the extremal fixed point components of an $H$-action defined on it:


\begin{lemma} \label{lem:extremal_rational_conn} Let $X$ be a complex projective manifold  with an $H$-action.
If $X$ is rationally connected, then also the source and the sink are rationally connected. \end{lemma}

\begin{proof} 
Let us take two points $y,y'\in Y_-$. The rational connectedness of $X$ implies that we may find an irreducible rational curve $C$ passing through $y,y'$. We then consider the limit cycle $C'$ of $tC$ when $t^{-1}$ goes to $0$; its irreducible components are rational curves. Since  $C\cap X^-(Y_-)$ is a dense open subset of $C$, hence connected, it follows that $C'\cap Y_-$ contains a connected cycle passing through $y,y'$. 
This shows that $Y_-$ is rationally chain connected, hence, since it is smooth, rationally connected (cf. \cite[Corollary 4.28]{De}). The same holds for $Y_+$.
\end{proof}

\begin{example}\label{rem:innerRC}
An analogous statement for  inner fixed components does not hold. Consider for instance the natural action of $H$ on $\P^1$, and its lift to the product $X':=\P^1\times\P^2$, whose extremal fixed point components are $\{0\}\times\P^2$, $\{\infty\}\times\P^2$. We consider a smooth non rational curve $C\subset\{0\}\times\P^2$ and define $X$ as the blow up of $X'$ along $C$. Since $C\subset X'$ consists of $H$-fixed points, the action of $H$ on $X'$ extends to an action on $X$ and, by construction, $X^H$ contains an inner fixed point component isomorphic to $C$.
\end{example}

%
%
%


The following Lemma was partially observed in \cite[Lemma~3.4]{BWW}, and it is a consequence of Theorem \ref{thm:BB_decomposition} (3). A similar statement holds replacing $Y_-$ with $Y_+$ and $\nu^+$ with $\nu^-$.

\begin{lemma} \label{lem:Pic_extremal}Let $X$ be a complex projective manifold with an $H$-action. Assume that $\Pic(X)=\ZZ \cL$, and let $Y_-$ be its sink. Then:
\begin{enumerate}
\item if $\dim Y_->0$, then  $Y_-$ is a Fano manifold and $\Pic(Y_-)=\Z{\cL_{|Y_-}}$; moreover $\nu^+(Y) \ge 2$ for every fixed point component $Y$ different from $Y_-$.
\item if $\dim Y_-=0$, $\dim X\geq 2$, then there exists a unique inner fixed component $Y$ such that $\nu^+(Y) =1$ and $\nu^+(Y') \ge 2$, for every fixed point component $Y'$ different from $Y_-$ and $Y$.
\end{enumerate}
\end{lemma}

\begin{proof}
Suppose that $\dim Y_->0$ so that, in particular, $\rank H_2(Y_-,\ZZ)\geq 1$. Since $X$ is rationally connected by Lemma \ref{lem:uniruled}, $$\Pic(X) \simeq \HH^2(X, \ZZ) \simeq \Hom_\Z(\HH_2(X, \ZZ),\ZZ),$$ and from  Theorem \ref{thm:BB_decomposition} (3) we get $\HH_2(X, \ZZ) \simeq \HH_2(Y_-,\ZZ)$ and $\nu^+(Y)\ge 2$ for any fixed point component different from the sink. Note that, as explained in \cite[Section~IIb]{CaSo}, this isomorphism is induced by the inclusion $Y_-\subset X$. Since  $Y_-$ is rationally connected by Lemma \ref{lem:extremal_rational_conn}, 
the restriction of the Picard groups $\Pic(X)\to \Pic(Y_-)$ is an isomorphism, as well. Finally, since $Y_-$ is rationally connected of Picard number one, then it is a Fano manifold. 

In the case in which  $Y_-$ consists of an isolated point, using the decomposition provided by Theorem \ref{thm:BB_decomposition} (3), we obtain:
$$
H_2(X,\Z)\simeq \bigoplus_{Y\neq Y_-} H_{2-2\nu^+(Y)}(Y,\Z)\oplus H_{2-2\nu^+(Y_-)}(Y_-,\Z)\simeq\bigoplus_{Y\neq Y_-} H_{2-2\nu^+(Y)}(Y,\Z).
$$
This equality follows from the fact that $\nu^+(Y_-)=0$. 
Since $\nu^+(Y)\geq 1$ for all $Y$, the summand provided by  $Y\in \cY\setminus \{Y_-\}$ can only be different from zero if $\nu^+(Y)=1$, and in this case the corresponding group would be $H_0(Y,\Z)\simeq\Z$. We conclude that this happens for a unique component $Y\in \cY\setminus \{Y_-\}$. This finishes the proof.
\end{proof}

We observe now that Theorem \ref{thm:BB_decomposition} (2) provides some geometric insight on the concept of equalization (Definition \ref{def-equalized}). In fact, given a fixed point component $Y$, the isomorphisms $X^{\pm}(Y)\simeq N^{\pm}(Y)$ send closures of $1$-dimensional orbits in $X^{\pm} (Y)$ to curves in the fibers of $N^{\pm}(Y)$ over $Y$. This allows us to identify the weights of the action on $N^{\pm}(Y)$ with values $\delta(C)$ of $H$-invariant curves $C$ in $X$ (see Definition \ref{def:delta}). Moreover, if all the weights of the action of $N^{\pm}(Y)$ are equal to $\pm 1$ then the images of these curves into $N^{\pm}(Y)$ are all lines, 
since in this case the action on every vector space $N^{\pm}(Y)_y$, $y\in Y$, is homothetical. Summing up, we may state the following result:
 

\begin{corollary}\label{cor:allequalized}
Let $X$ be a smooth projective variety supporting an $H$-action.
Then:
\begin{enumerate}
\item[(i)] If the action is equalized at a component $Y$, then any $H$-invariant curve having an extremal fixed point $x'$ in $Y$ is smooth at $x'$.
\item[(ii)] The action is equalized if and only if $\delta(C)=1$ for every $H$-invariant rational curve $C$ on $X$.
\item[(iii)] If the action is equalized at the source $Y_+$ or at the sink $Y_-$, then the isotropy group at the orbit of a general point is trivial. In particular the action is faithful.
\item[(iv)] If the action is equalized at $Y_+$ and $Y_-$, and every nontrivial orbit has source in $Y_+$ or sink in $Y_-$ (or both), then the action is equalized.
\end{enumerate}
\end{corollary}

\subsection{$\C^*$-actions on polarized pairs}

For an arbitrary line bundle $\cL \in \Pic (X)$ we denote by $\mu_\cL\colon H\times \cL\rightarrow \cL$ (or simply by $\mu$) a \textit{linearization} of the action of $H$ on the line bundle $\cL$; linearizations always exists by \cite[Proposition 2.4]{KKLV}. By abuse, we continue to denote by
$\mu_\cL\colon \cY \to M(H):=\Hom(H,\C^*) \simeq \Z$ the associated map on the set of fixed point components. 
Note also that given a linearization on $\cL$ we may construct other linearizations by adding characters of $H$ to it; in other words, we may choose the linearization so that $\mu_{\cL}$ has a prescribed value on a given fixed point component. 

\begin{definition} Given a smooth complex projective variety $X$ an ample line bundle $\cL$ on $X$ we will call  $(X,\cL)$ a \textit{polarized pair}.
\end{definition}

\begin{definition}\label{def:bandwidth} Let $(X,\cL)$ be a polarized pair with an $H$-action admitting a linearization $\mu$ on $L$. The \textit{bandwidth} of the action on the pair $(X,\cL)$ is defined as $|\mu|=\mu_{\max}-\mu_{\min}$, where $\mu_{\max}$ and $\mu_{\min}$ denote the maximal and minimal value of the function $\mu_\cL$.
 \end{definition}
 
 \begin{remark}\label{rem:bandwidth}
 Note that the values $\mu=\mu_{\max},\mu_{\min}$ are attained at the source and the sink, respectively. In fact, the claim is obvious when the pair is $(\P^n,\cO(1))$, as one can check easily by choosing a system of coordinates of $\P^n$ on which the $H$-action diagonalizes. In general, given an $H$-action on $(X,\cL)$, since $\mu_{m\cL}=m\mu_{\cL}$, we may assume that $\cL$ is very ample, and consider the extension of the $H$-action via the inclusion $X\hookrightarrow\P(V)$, where $V=\HH^0(X,\cL)$. Let $V=V_{a_0}\oplus\dots \oplus V_{a_r}$ be the weight decomposition of the action, with $a_0<\dots<a_r$. Since $X\subset\P(V)$ is nondegenerate, the general point $x\in X$ can be written as the homothety class of $\sum_{i=0}^rv_i$, where $v_i\in V_{a_i}^\vee$ is a nonzero vector for all $i$. In particular, $\lim_{t\to 0}t^{-1}x\in\P(V_{a_0})$, $\lim_{t\to 0}tx\in\P(V_{a_r})$, which shows that $X\cap \P(V_{a_0}), X\cap \P(V_{a_r})\neq\emptyset$, and that $\mu_{\min}=a_0$, $\mu_{\max}=a_r$ are attained respectively at the sink and at the source of the $H$-action on $X$.
 \end{remark}
 

The AM vs FM equality, which has been introduced in \cite[$\S$2.A]{RW}, relates the amplitude of a line bundle on $\PP^1$ with the difference of weights of the action on the fibers of the line bundle over the fixed points and the weight $\delta(y_+)$ of the action on the tangent space $T_{y_+}\PP^1$ of $\P^1$ at the source $y_+$. 
We will discuss now some consequences of this equality, and  generalize it to vector bundles (cf. Lemma \ref{lem:generalizedAMvsFM}).  

\begin{lemma}\label{lem:AMvsFM} \cite[Lemma 2.2]{RW}
Let $H\times\PP^1\rightarrow\PP^1$ be an
action with source and sink $y_+$ and $y_-$. Consider a line bundle $\cL$
over $\PP^1$ with linearization $\mu_\cL$. Then
\begin{equation}\tag{AM vs FM}
\mu_\cL(y_+)-\mu_\cL(y_-)=\delta(y_+)\cdot \deg\cL.
\end{equation}
\end{lemma}

%
%

\begin{corollary}\label{cor:sum_of_orbits}
Let $X$ be a smooth projective variety with an $H$-action. Denote by $Y_+$ and $Y_-$, respectively, the source and the sink, and by $C_{\gen}$ the closure of a general orbit.
Let $C_1,\dots,C_m$ be 
$H$-invariant irreducible curves, with source and sink $y_+^{i}$, $y_-^{i}$, such that $C_1\cap Y_+\ne\emptyset\ne C_m\cap Y_-$ and $y_-^i$  is contained in the fixed
point component containing  $y_+^{i+1}$.
Then, in $\Nu(X)$, we have $$\sum_i
\delta(C_i)[C_i]=\delta(C_{\gen})[C_{\gen}].$$
In particular, if the action is equalized, then  
$$\sum_i [C_i]=[C_{\gen}].$$
\end{corollary}

\begin{proof}  Let  $\cL$ be a line bundle on $X$, with linearization $\mu_{\cL}$. By Lemma \ref{lem:AMvsFM},
 $$\delta(C_{\gen})\cdot \deg f^*{\cL}=\mu_\cL(y_+)-\mu_\cL(y_-),$$ 
where $y_+$ and $y_-$ are respectively the source and the sink of $C_{\gen}$.
Again by Lemma \ref{lem:AMvsFM} we get 
$$\delta(C_i)\cdot \deg f_{i}^*\cL=\mu_\cL(y_+^{i})-\mu_\cL(y_-^{i}).$$
 Our assumptions imply that $\mu_\cL(y_-^{i})=\mu_\cL(y_+^{i+1})$ for $i=1,\dots,m-1$, $\mu_{\cL}(y_-^{m})=\mu_{\cL}(y_-)$, and $\mu_\cL(y_+^{1})=\mu_\cL(y_+)$. Therefore, using the above equalities, we can write 
 $$\delta(C_{\gen})(\cL\cdot C_{\gen})=\sum_{i=1,\dots,m} (\mu_\cL(y_+^{i})-\mu_\cL(y_-^{i}))=\sum_{j=1,\dots,m} \delta(C_{i}) (\cL\cdot C_i),$$
and, since we can repeat this procedure for every line bundle $\cL$ on $X$, we get the statement. The last claim follows from Corollary \ref{cor:allequalized} (ii).
\end{proof}

Another consequence of Lemma \ref{lem:AMvsFM} is that there exists a chain of $H$-invariant curves linking any point with the sink and the source:
\begin{corollary}\label{cor:chain}
Let $(X,\cL)$ be a polarized pair with an $H$-action. 
Given a point $x\in X$ there exist a sequence of $H$-invariant irreducible curves $C_1,\dots,C_m$, such that $C_i\cap C_{i+1}$ is an $H$-fixed point for every $i$, $C_1\cap Y_-\neq \emptyset\neq C_m\cap Y_+$, and such that $x\in \bigcup_i C_i$.  
\end{corollary}

\begin{proof}
Let $x\in X$ be any point, and let $x_{\pm}:=\lim_{t\to 0}t^{\pm 1}x$. We claim that, either $x_-\in Y_-$, or there exists an $H$-invariant irreducible curve $C$ with source $x_-$ and whose sink $x_-'$ satisfies $\mu_{\cL}(x_-')<\mu_{\cL}(x_-)$. This can be used recursively to show that there exists a connected chain of $H$-invariant irreducible curves linking $x$ with $Y_-$. A similar argument provides a chain linking $x$ with $Y_+$. 

To prove the claim, we note that if $x_-\not\in Y_-$, denoting by $Y$ the fixed point component containing $x_-$, we have that $\nu^+(Y)>0$. Then by Theorem \ref{thm:BB_decomposition}(2), there exists a $1$-dimensional orbit converging to $x_-$ when $t$ goes to $0$; let $C$ be the closure of this curve, and $x_-'$ be its sink. By Lemma \ref{lem:AMvsFM}, we get $\mu_{\cL}(x_-')<\mu_{\cL}(x_-)$.
\end{proof}

Lemma \ref{lem:AMvsFM} can be generalized to vector bundles in the following way:

\begin{lemma}\label{lem:generalizedAMvsFM}
Let $\cE$ be a vector bundle over $\PP^1$ with splitting $\cE\iso\bigoplus_i\cO(d_i)$.
We take an  $H$-action on $\PP^1=\CC^*\cup\{y_+\}\cup\{y_-\}$ with source  $y_+$ and sink  $y_-$, and  consider a linearization of this action on $\cE$. We denote by $a_i$ the weights at $\cE_{y_+}$ and by $b_i$ the weights at $\cE_{y_-}$.
Then, after possibly renumbering the $a_i$'s and the $b_i$'s, we have 
$$a_i-b_i=d_i$$
\end{lemma}

 \begin{proof}
 First we note that, by additivity of characters, the formula behaves well
 with respect to the twist of $\cE$ by any line bundle with any linearization.
 That is, if $\cL=\cO(d)$ has a linearization such that the weights at $\cL_0$
and $\cL_\infty$ are $a$ and $b$, respectively, then, by Lemma \ref{lem:AMvsFM}, one has 
 $a-b=d$ and $\cE(d)$ has a product linearization with weights $a_i+a$
 at $\cE(d)_{y_+}$ and $b_i+b$ at $\cE(d)_{y_-}$.

Next we note that splitting $\cE=\bigoplus_j V_j\otimes\cO(d_j)$, with $V_j$
vector spaces and $d_j$ pairwise different integers, is preserved under the
 $H$-action. Therefore it is enough to prove our formula for vector
 bundles of type $V\otimes\cO(d)$ and, eventually, by the remark above, for
 trivial vector bundles.

If $\cE=V\otimes\cO$ then any linearization of $\cE$ yields a linearization of
$V=H^0(\PP^1,\cE)$. The evaluation of sections at $y_+$ and $y_-$ yields equivariant morphisms $V\rightarrow\cE_{y_+}$ and $V\rightarrow\cE_{y_-}$.
Therefore the weights of the action  at $\cE_{y_+}$ and $\cE_{y_-}$ are the same which concludes the proof of the lemma.
\end{proof}

Lemma \ref{lem:generalizedAMvsFM} can be used to determine the splitting type of the tangent bundle on
orbits of the action. In particular we have the following:

\begin{corollary}\label{cor:splittingTangent}
Let $X$ be a projective variety admitting an equalized action of $H=\CC^*$ such that the
source is a point. Then the restriction of $TX$ to the closure of any
orbit of the action joining the source with a component $Y\in \cY$ is
$$\cO(2)^{\nu^-(Y)}\oplus\cO(1)^{\dim Y}\oplus\cO^{\nu^+(Y)}.$$
\end{corollary}

\begin{proof}
Since the action is equalized, by Corollary \ref{cor:allequalized} (i) the closures of all orbits are smooth rational curves. Moreover the weights at the source are equal to $1$, while the weights
at $Y$ are $(-1^{\nu^-(Y)}, 0^{\dim Y}, 1^{\nu^+(Y)})$.
\end{proof}

We finish this section by introducing $H$-trivial line bundles on $X$ and discussing some properties of their base loci.

\begin{definition} \label{def:H-invariant}
A line bundle $\cL\in\Pic (X)$ is called $H$-{\em trivial} 
if it is trivial on the closure of every orbit of $H$. Note that $H$-trivial line bundles are $H$-invariant, but the converse is not true in general; there are many examples of $H$-actions in which the sink $Y_-$ is a divisor, which is $H$-invariant but not $H$-trivial (see Section \ref{sec:Btype_Bordism} below). The subgroup of
$H$-trivial line bundles will be denoted by $\Pic(X)^H<\Pic (X)$. We can define $\NU(X)^H$ in a similar way. In other words, $\NU(X)^H$ is the space orthogonal, with respect to the pairing of $\NU(X)$ and $\Nu(X)$ given by the intersection, to the subspace $\Nu(X)^H$ spanned by the classes of closures of orbits of $H$.
\end{definition}

A very important property of $H$-trivial line bundles is that their spannedness on the sink or the source implies their spannedness on the whole variety: 

\begin{lemma}\label{lem:base-locus}
Let $X$ be a variety with an $H$-action, with source $Y_+$ and sink $Y_-$ and let  $\cL$ be a $H$-trivial line bundle. If either
$\Bs|\cL|\cap Y_-=\emptyset$ or $\Bs|\cL|\cap Y_+=\emptyset$, then
$\Bs|\cL|=\emptyset$.
\end{lemma}

\begin{proof}
Let $x\in \Bs|\cL|$ be a base point, and let $C_1\cup\dots\cup C_m$ be a connected chain of $H$-invariant irreducible curves linking $x$ with the sink and with the source, which exists by Corollary \ref{cor:chain}. By hypothesis $C_i\cdot \cL=0$ for every $i$, and $\bigcup_iC_i\cap \Bs|\cL|\neq \emptyset$. If a component $C_j$ intersects $\Bs|\cL|$, then it meets every effective divisor $D\in|\cL|$, and the equality $C_j\cdot \cL=0$ implies that $C_j$ is contained in each $D\in|\cL|$; it follows that $C_j\subset \Bs|\cL|$. Since $\bigcup_iC_i$ is connected and meets $\Bs|\cL|$, we may conclude that it is contained in $\Bs|\cL|$ and, in particular, $\Bs|\cL|$ intersects both $Y_-$ and $Y_+$.
\end{proof}


\subsection{Rational homogeneous spaces: notation}\label{ssec:notRH}

We will finish this section of preliminaries by introducing the notation we will use when dealing with rational homogeneous varieties.

It is well known that two isogenous semisimple groups have the same projective quotients, hence a rational homogeneous variety $G/P$ is completely determined by the Lie algebra $\fg$ of $G$ and the parabolic subalgebra $\fp\subset \fg$ of $P$ which, up to adjoint action, are determined by the Dynkin diagram $\cD$ of $\fg$ and the choice of a particular set of its nodes.

More concretely, given a semisimple group $G$ with Lie algebra $\fg$, we fix a Borel subgroup $B\subset G$ and a maximal torus $T\subset B$. This choice determines a root system $\Phi$ contained in the character lattice $\Mo(T):=\Hom(T,\C^*)$ of $T$, together with a base of positive simple roots $\Delta$ of $\Phi$, that determines a partition of the root system into subsets of positive and negative roots, $\Phi=\Phi^+\cup \Phi^-$. The finite group $W={\rm N}(T)/T$ which can be shown to be independent of the choice of the maximal torus $T$ of $G$ is called the \textit{Weyl group} of $G$. Any element $\alpha_i\in\Delta$ defines an involution of $\Mo(T)$ sending $\alpha_i$ to $-\alpha_i$ and leaving $\Phi$ invariant, that we denote by $s_i$. The Dynkin diagram $\cD$ of $G$ will be a graph (admitting multiple, oriented edges) whose set of vertices $D$ is in one to one correspondence with $\Delta$, and whose edges encode the behaviour of the involutions $s_i$ on $\Phi$. It follows that $\Phi$ is completely determined by $\cD$, and one can show that the Lie algebra $\fg$ is completely determined by $\Phi$.

For every subset of nodes $I \subset D$ one can consider the subgroup $W(D \setminus I)\subset W$ generated by the elements $s_i$, $i \notin I$. The subgroup $P(D\setminus I):=BW(D \setminus I)B$ is a parabolic subgroup of $G$, and quotient $G/P(D\setminus I)$ is a projective variety, that we call the rational homogeneous variety associated to $G$ and to the set of nodes $I$. Note that the Lie algebra of $P(D\setminus I)$ is equal to
$$
\fp(D\setminus I)=\fh\oplus\bigoplus_{\alpha\in\Phi^+}\fg_\alpha\oplus\bigoplus_{\alpha\in\Phi^+(D\setminus I)}\fg_{-\alpha},
$$
where $\Phi^+(D\setminus I)$ denotes the subset of $\Phi^+$ generated by the positive simple roots $\alpha_j$, $j\in D\setminus I$.

As we have already noted, the variety obtained upon $I$ for another semisimple group $G'$ isogenous to $G$ is the same, and it makes sense to use the following notation:
$$\cD(I):=G/P(D \setminus I).$$

Note that, with this notation, the choice of an inclusion $I\subset J\subset D$ gives rise to a contraction:
$$
\cD(J)\to\cD(I).
$$

For instance, $\cD(D)=G/P(\emptyset)=G/B$ is the complete flag variety associated to $G$, and  $\cD(\emptyset)=G/G$ is always a point. In the cases in which $\cD$ is disconnected, a rational homogeneous variety $\cD(I)$ is a product, whose factors correspond to the connected components $\cD_i$ of $\cD$, marked on the nodes of $I$ contained in $\cD_i$.
The rational homogeneous varieties associated to Dynkin diagrams of classical type ($\DA_n$, $\DB_n$, $\DC_n$, $\DD_n$) can be easily described in classical language, that we will use in the cases in which a more detailed geometric description of certain rational homogeneous varieties is convenient. For instance:
 $$\begin{array}{l}\DA_n(1)\simeq \P^n,\,\, \DA_n(n)\simeq\P^{n\vee},\,\, 
 \DA_n(i)\simeq\G(i-1,n)=\{\mbox{subspaces }\P^{i-1} \subset\P^n\},\\[2pt]
 \DA_n(1,n)\simeq\P(T_{\P^n}),\\[2pt]
 \DB_n(1)\simeq\Q^{2n-1}\subset\P^{2n}, \,\,\DD_n(1)\simeq\Q^{2n-2}\subset\P^{2n-1} \mbox{(smooth quadrics)},\\[2pt]
 \DC_3(3)\simeq \LG(2,5)\mbox{ (Lagrangian Grassmannian parametrizing Lagrangian $\P^2\subset\P^5$}\\[1pt]
 \mbox{w.r.t. a contact form in $\P^5$)},\\[2pt]  
 \DD_6(6)\simeq S^{15} \mbox{ (Spinor variety, parametrizing one of the two connected families of }\\[1pt]
 \mbox{subspaces $\P^5\subset \P^{11}$ that are contained in a smooth quadric $\Q^{10}\subset \P^{11}$).}
 \end{array}
 $$




\section{B-type torus actions and bordisms}\label{sec:Btype_Bordism}

In this section we discuss the concepts of B-type $\C^*$-action and bordism. In a nutshell, the conditions defining these two types of actions will allow us to define a birational morphism between the  sink and the source, regular in codimension one, which encodes many properties of the action, and that, in certain situations, will determine it.

\subsection{B-type torus actions}

\begin{definition}\label{def:bord}
Let  $X$ be a complex projective manifold 
with an action of $H=\CC^*$. We say that the action is of {\em B-type} if its extremal fixed components $Y_-$ and $Y_+$ are codimension one subvarieties.
\end{definition}

By Theorem \ref{thm:BB_decomposition} (2), an $H$-invariant open neighborhood of $Y_\pm$ in $X$ (the BB-cell $X^{\pm}(Y_{\pm})$) is $H$-isomorphic to the line bundle $N_{Y_{\pm}/X}$. Two important consequences of this fact are:


\begin{remark}\label{rem:Btypeequalized}
If the action of $H$ is faithful, then the action on a general fiber of $\pi_{\pm}\colon N_{Y_{\pm}/X}\to Y_\pm$ is the homothety action or its inverse.  Since the weights of the action on the fibers of $\pi_{\pm}$ does not depend on the chosen point in $Y_\pm$,  a faithful B-type action is equalized at the extremal fixed components. 
The converse is true for any $H$-action, as observed in Corollary \ref{cor:allequalized} (iii).
\end{remark}

\begin{remark}\label{rem:Btypeintnumber}
In particular, if a B-type action of $H$ is faithful, the $H$-isomorphism $X^{\pm}(Y_{\pm})\simeq N_{Y_{\pm}/X}$ given by Theorem \ref{thm:BB_decomposition}(2) implies that the closure  $C$ of an orbit meeting $Y_-$ (resp. $Y_+$) meets $Y_-$ transversally; in particular $Y_- \cdot C=1$ (resp. $Y_+ \cdot C=1$).
\end{remark}


Given an $H$-action on $X$ with source $Y_+$ and sink $Y_-$, and  inner fixed components $Y_i$, $i\in I$, we consider the subvarieties  $Z_-^i:=\ol{X^+(Y_{i})}\cap Y_{-}\subset Y_{-}$, $Z_+^i:=\ol{X^-(Y_{i})}\cap Y_{+}\subset Y_{+}$ and we set: $$Z_-:=\bigcup_{i\in I}Z_-^i,\quad  Z_+:=\bigcup_{i\in I}Z_+^i.$$

\begin{lemma} \label{lem:birational_map_b} In the above notation, let $X$ be a complex projective manifold  with a B-type $H$-action. 
Then there exists an isomorphism:
$$
\psi:Y_{-}\setminus Z_{-}\lra Y_{+}\setminus Z_+,
$$
assigning to every point $x$ of $Y_{-}\setminus Z_{-}$ the limit for $t \to 0$ of the unique orbit having limit for $t^{-1} \to 0$ equal to $x$.
\end{lemma}

\begin{proof}
The complementary set of the zero section $N_{Y_\pm/X}^0\subset N_{Y_\pm/X}$ is a principal $\C^*$-bundle, from which $Y_\pm$ can be defined as the geometric quotient by the action of $H$; it is an algebraic map, mapping every $x$ to $\lim_{t\to 0} t^{\pm 1}\cdot x\in Y_\pm$.

We may now consider the open subsets $U_\pm=N_{Y_\pm/X}^0\setminus \pi_\pm^{-1}(Z_\pm)$; these two subsets are isomorphic since they can be identified in $X$ with the union of the set of orbits of the action having limiting points at $Y_{-}$ and $Y_{+}$. It then follows that their quotients by the action of $H$, $Y_-\setminus Z_-$ and $Y_+\setminus Z_+$,  are isomorphic.
\end{proof}

We will now show that, under certain mild conditions, the map $\psi$ constructed in Lemma \ref{lem:birational_map_b}, understood as a rational map from $Y_-$ to $Y_+$, is an isomorphism in codimension one.  Let us start by observing the following:

\begin{lemma} \label{eq:codimZ0}Assume that  the $H$-action on $X$ is of B-type, and let $Y_i$ be an inner fixed point component as above. If $Z_-^i \not = \emptyset$ then $\codim(Z_-^i,Y_-)=\nu^-(Y_{i})$. If $Z_+^i \not = \emptyset$ then  $\codim(Z_+^i,Y_+)=\nu^+(Y_{i})$. 
\end{lemma}
\begin{proof}
Assume that $Z_-^i \not = \emptyset$. Then,
by  BB-decomposition:
$$
\dim(Y_{i})+\nu^+(Y_{i})=\dim {X}^+(Y_{i})=\dim (Z_-^i)+1,
$$
so we have that
$\dim (Z_-^i)=\dim (Y_{i})+\nu^+(Y_{i})-1$, and, since by assumption $Y_-$ has codimension one in $X$, we get:
\begin{center}
$\codim(Z_-^i,Y_{-})=\dim X-\dim (Y_{i})-\nu^+(Y_{i})=\nu^-(Y_{i})$.
\end{center}
A similar argument provides $\codim(Z_+^i,Y_{+})=\nu^+(Y_{i})$.
\end{proof}

The following statement has been proven in \cite[Theorem~3]{CaSo}:
\begin{lemma}\label{lem:CaSo}
Let $X$ be a smooth projective variety admitting an $H$-action, and let $Y_-$ be its sink. There exists a short exact sequence:
$$\shse{\sum_{\substack{Y\in\cY\\\nu^+(Y)=1}}\Z\left[\overline{X^-(Y)}\right]}{\Pic(X)}{\Pic(Y_-)}.$$
\end{lemma}

A consequence of this is the following equivalence:

\begin{corollary}\label{cor:codim1} Let $X$ be a smooth projective variety with a B-type $H$-action.
 Then the following conditions are equivalent:
\begin{itemize}
\item[$(\star)$] the restriction map $\Pic(X)\to \Pic(Y_-)$ fits into a short exact sequence
$$\shse{\Z [Y_+]}{\Pic(X)}{\Pic(Y_-)};$$
\item[$(\star\star)$] $\nu^+(Y_{j})
\geq 2$, for every inner fixed point component $Y_j$.
\end{itemize}
Moreover, if these conditions hold, then $\codim{(Z_+, Y_+)}\geq 2$.
\end{corollary}

\begin{proof}
Note that $X^-(Y_+)=Y_+$, and that for a B-type $H$-action we have $\nu^+(Y_+)=1$, then $(\star\star)\Rightarrow (\star)$ follows directly from Lemma \ref{lem:CaSo}. For the converse we note that $(\star)$ together with Lemma \ref{lem:CaSo} implies that for any inner component $Y$ with $\nu^+(Y)=1$, the divisor $\overline{X^-(Y)}$ is numerically proportional to $Y_+$. But this is not possible since $\overline{X^-(Y)}$ is not numerically trivial and, denoting by $C_{\gen}$ the closure of a general orbit in $X$, we have that $Y_+\cdot C_{\gen}=1$, $\overline{X^-(Y)}\cdot C_{\gen}=0$. 
\end{proof}



\begin{definition}\label{def:simpleBaction}%
If a B-type action on $X$ satisfies the equivalent conditions of Corollary
\ref{cor:codim1} for both the source $Y_+$ and the sink $Y_-$ of the action, 
then the triple $Y_-\hooklongrightarrow X \hooklongleftarrow Y_+$
together with the action is called a {\it bordism}.
\end{definition}

In particular, for a bordism, the last statement of Corollary \ref{cor:codim1} implies:

\begin{lemma}\label{lem:bordism=iso-in-codim1}
If $Y_-\hooklongrightarrow X \hooklongleftarrow Y_+$ is a bordism then the birational map
$\psi: Y_-\dashrightarrow Y_+$ defined in Lemma \ref{lem:birational_map_b} is an isomorphism in codimension one. In particular, the push-forward
maps of divisors $\psi_*: \Div (Y_-)\rightarrow \Div (Y_+)$ and linear equivalence classes
$\psi_*: \Cl (Y_-)=\Pic (Y_-)\rightarrow \Cl (Y_+)=\Pic (Y_+)$ are isomorphisms.
\end{lemma}

The following statement shows that the technical conditions of Corollary \ref{cor:codim1} are  fulfilled by the B-type $H$-actions obtained by blowing-up along the source and the sink a variety with an equalized $H$-action, whenever these  components are positive dimensional (note this condition is necessary by Lemma \ref{lem:Pic_extremal} (2)). This construction will come in handy in Sections \ref{sec:BW1} and \ref{sec:BW22},  to study varieties of bandwidth $1$ and $2$.

\begin{lemma}\label{lem:setup_bir}
Let $X$ be a smooth projective variety with an $H$-action which is equalized at the source $Y_{+}$ and at the sink $Y_-$. Let $X_\flat$ be the blowup of $X$ along $Y_+$ and $Y_-$, with exceptional divisors $Y_+^\flat,Y_-^\flat$.
Then the action extends to a  
 B-type action on $X_\flat$, equalized at the source $Y_+^\flat$ and at the sink $Y_-^\flat$.
If moreover $\rho_X=1$  and both $Y_+$ and $Y_-$ are positive dimensional, then $Y_-^\flat\hooklongrightarrow X_\flat \hooklongleftarrow Y_+^\flat$ is a bordism.
\end{lemma}

\begin{proof}
The sink and the source are fixed point components, therefore there is an induced action on their normal bundle, hence on the exceptional divisors $Y_\pm^\flat:=\P(N^\vee_{Y_\pm/X})$. The fact that the action is equalized at $Y_\pm$ yields that $H$ acts homothetically on the fibers of the normal bundles $N_{Y_\pm/X}$, hence  the induced action on $X_\flat$ is trivial along the fibers of the blowup.
In particular the exceptional divisors $Y_+^\flat,Y_-^\flat$ are fixed point components and are the source and the sink of the induced action. Since the action on $X$ is faithful (see Corollary \ref{cor:allequalized}), it is also faithful on $X_\flat$, and so it is equalized at its sink and source by Remark \ref{rem:Btypeequalized}. 

For the second part, note first that, by Lemma \ref{lem:uniruled}, $X$ is rationally connected.
Moreover, by Lemma \ref{lem:Pic_extremal} (1), it follows that $\nu^\pm(Y_j)\geq 2$ for every inner fixed component $Y_j$. Since blowing up the sink and the source does not change these values, we may claim that $X_\flat$ satisfies condition ($\star\star$) of Corollary \ref{cor:codim1}. 
\end{proof}

Let $Y_-\hooklongrightarrow X \hooklongleftarrow Y_+$ be a bordism;
let us denote by $\iota_\pm: Y_\pm \rightarrow X$  the embeddings, by
$\iota_\pm^*: \Pic (X)\rightarrow \Pic (Y_\pm)$, $\iota_\pm^*:\NU(X)\rightarrow \NU(Y_\pm)$  the restriction
maps, and by $\iota^\pm_*:\Nu(Y_\pm)\rightarrow \Nu(X)$ the maps of the space of 1-cycles induced by push-forward.

We will now define right inverses for $\iota_\pm^*$ exploiting the $H$-action on $X$.

\begin{definition} \label{def:extend_div}
Given a prime divisor $D\in \Div (Y_-)$ we define $\iota^-_\#(D)\in\Div (X)$ to be the closure in $X$ of the pullback of $D$  via the projection $X^-(Y_-)\rightarrow Y_-$. 
\end{definition}

We note that, by construction, 
$\iota^-_\#(D)\cap Y_+$ is a divisor in $Y_+$ equal to $\psi_*(D)$.
The map $\iota^-_\#$ induces  homomorphisms $\iota^\pm_{\#}:\Pic (Y_\pm)\rightarrow \Pic (X)$
and $\iota^\pm_{\#}:\NU(Y_\pm)\rightarrow \NU(X)$.
In fact, if $f\in \C(Y_-)$ is a rational function and $f'\in \C(X)$ is its pull-back via
$X^-(Y_-)\rightarrow Y_-$, then the valuation of $f'$ at $Y_+$ is zero, and thus
$\iota^-_\#(\div_{Y_-}(f))=\div_X(f')$. Therefore we have the following:

\begin{lemma}\label{lem:psistar}
The birational transformation $\psi:Y_- \dasharrow Y_+$
described in Lemma \ref{lem:birational_map_b} induces isomorphisms
$\psi_*=\iota^*_+\circ\iota_{\#}^-:\NU(Y_-)\rightarrow \NU(Y_+)$ and
$\psi_*^{-1}=\iota^*_-\circ\iota_{\#}^+:\NU(Y_+)\rightarrow
\NU(Y_-)$.
\end{lemma}

We also note that the compositions  $\iota_-^*\circ\iota^-_\#$ and  $\iota_+^*\circ\iota^+_\#$ are the
identity, hence  the exact sequences of Corollary \ref{cor:codim1} split as:
$$\Pic (X) =\iota^-_\#(\Pic (Y_-))\oplus \ZZ [Y_+] = \iota^+_\#(\Pic (Y_+))\oplus \ZZ [Y_-].$$
Moreover $\iota^*_+[Y_+] = N_{Y_+/X}$ and $\iota^*_-[Y_-] = N_{Y_-/X}$ (cf. \cite[Example~2.5.5]{Fult}).

\begin{lemma}\label{lem:Cgenort}
  Let $[C_{\gen}]$ be the class of the closure of a general $H$-orbit
  in $X$.  Then the image of $\iota^{\pm}_{\#}$ is orthogonal, in terms of
  intersection of $\NU(X)$ and $\Nu(X)$, to the class
  $[C_{\gen}]$. That is,
  $$\iota_\#^-(\NU(Y_-))=\iota_{\#}^+(\NU(Y_+))=[C_{\gen}]^\perp$$
\end{lemma}

\begin{proof}
Let $D\subset Y_-$ be a prime divisor. Choose a general $y\in Y_-\setminus D$. Then the orbit with the sink in $y$ does not meet $\iota^-_\#(D)$ hence $D\cdot C_{\gen}=0$. 
\end{proof}

We can summarize the situation with the following diagram:

\begin{equation}\label{diag:bordism}
\xymatrix@R=18pt{&&&\NU(Y_+)\\\NU(Y_-)\ar@/^/[rrru]^{\psi_*}\ar@/_/[rrrd]_{\id}\ar[rr]^{i_\#^-}&&[C_{\gen}]^\perp \subset \NU(X)\ar[ru]_{i_+^*}\ar[rd]^{i_-^*}&\\&&&\NU(Y_-)}
\end{equation}

\subsection{Low rank bordisms}

We will now introduce an invariant measuring the complexity of a bordism, called {\em rank}, and study in detail the cases of rank $0$ and $1$.

\begin{definition}\label{def:rank}
The {\it rank} of a bordism $Y_-\hooklongrightarrow X \hooklongleftarrow Y_+$ is defined as the number $\dim \Nu(X)^H-1$ (cf. Definition \ref{def:H-invariant}).
\end{definition}

\begin{lemma}\label{lem:rank0bordism}
The rank of a bordism
$Y_-\hooklongrightarrow X \hooklongleftarrow Y_+$
is nonnegative, and it is zero if and only if $Y_-\simeq
  Y_+$ and $X$ is an $H$-equivariant $\PP^1$-bundle over $Y_-\simeq Y_+$.
\end{lemma}

\begin{proof} Since the action is non trivial, in $\Nu(X)^H$ we have at least the class $[C_{\gen}]$ of the closure of a general orbit, hence, by the above definition, the rank of the bordism is nonnegative. Assume that the rank is zero; again by definition this happens if and only if the closure of every orbit is numerically proportional to $[C_{\gen}]$. By Remark \ref{rem:Btypeintnumber}, $Y_-\cdot
 C_{\gen}=Y_+\cdot C_{\gen}=1$, thus the intersection number of the closure of every
 orbit with the extremal fixed point components is positive, and thus
 there is no other fixed point component (since the closure of an orbit cannot meet three fixed point components).
 Then, by Lemma \ref{lem:birational_map_b}, since the subvarieties $Z_{\pm}\subset Y_{\pm}$ are empty, we obtain that $Y_-\simeq Y_+$ and, by the proof of the same lemma, it follows that $X$ is an $H$-equivariant $\PP^1$-bundle over $Y_-\simeq Y_+$.
\end{proof}

\begin{lemma}\label{lem:Cpm}
Let $Y_-\hooklongrightarrow X \hooklongleftarrow Y_+$ be a bordism of rank $1$. Let $Y_0\subset X^H$ be the union of all the inner fixed point components. Then every nontrivial orbit with sink in $Y_0$ has source in $Y_+$,
and every nontrivial orbit with source in $Y_0$ has 
sink in $Y_-$. Moreover there exists a unique numerical equivalence class of closures of orbits having sink (resp. source) on $Y_0$.
\end{lemma}

\begin{proof}
First, note that $[Y_-]$ and $[Y_+]$ are not proportional as forms
on $\Nu(X)^H$. Indeed, using Lemma \ref{lem:rank0bordism} we know that $X$ is not a $\PP^1$-bundle; then there
exists an orbit of $H$ whose closure meets, say, $Y_-$ and does not
meet $Y_+$. 
Since $\dim \Nu(X)^H =2$, then  $[Y_-]^\perp\cap[Y_+]^\perp\cap N_1(X)^H=\{0\}$, so the
closure of every nontrivial orbit of $H$ meets either $Y_-$ or
$Y_+$. By Corollary \ref{cor:allequalized} (iv) the action is equalized.

To prove the last statement, let us complete a basis of $\NU(X)^H$ to a basis of $\NU(X)$, by adding $[Y_+]$ and $[Y_-]$. 
Given two $H$-invariant curves $C_+$ and $C_+'$  linking $Y_0$ with $Y_+$, by Remark \ref{rem:Btypeintnumber}, $Y_+ \cdot C_+ =Y_+ \cdot C_+'=1$. On the other hand $C_+$ and $C_+'$ have intersection zero with $Y_-$ and with every $H$-trivial divisor, hence $[C_+] = [C_+']$. The same argument works in the case of orbits linking $Y_0$ with $Y_-$.
\end{proof}

\begin{definition}\label{def:Cpm}
Let $Y_-\hooklongrightarrow X \hooklongleftarrow Y_+$ be a bordism of rank $1$. We will denote by $[C_+]$ (resp. $[C_-]$) the numerical class of the closure of an orbit linking $Y_+$ (resp. $Y_-$) with an inner component, whose existence and uniqueness is guaranteed by the previous Lemma. We represent this situation in the following graph.
\begin{equation}\label{eq:gr1}
\xygraph{
!{<0cm,0cm>;<2cm,0cm>:<0cm,1cm>::}
!{(0,0) }*+{\bullet}="0"
!{(0,-0.35) }*+{{Y_-}}="3"
!{(1.5,0) }*+{\bullet}="1"
!{(1.5,-0.35) }*+{{Y_0}}="4"
!{(3,0) }*+{\bullet}="2"
!{(3,-0.35) }*+{{Y_+}}="5"
"0"-@/^/@[green]"1" _{C_-}
"0"-@/^0.6cm/@[blue]"2" ^{C_{\gen}}
"1"-@/^/@[red]"2" _{C_+}
} 
\end{equation}
\end{definition}

\begin{lemma}\label{rank1bordism-curves}
Let $Y_-\hooklongrightarrow X \hooklongleftarrow Y_+$ be a bordism of rank $1$, and denote by $C_{\gen}$ the closure of a general orbit. Then $\Nu(X)^H$ is generated by  $[C_\pm]$,  in $\Nu(X)$ we have the equality $[C_{\gen}]=[C_+]+[C_-]$, and
$$\iota_{*}^+\Nu(Y_+)\cap \Nu(X)^H=\RR [C_+],\quad \ \
\iota_*^-\Nu(Y_-)\cap \Nu(X)^H=\RR [C_-].$$
\end{lemma}

\begin{proof}
The first assertion follows by the assumption on the rank of the bordism, and the fact that $[C_-],[C_+]$ are not proportional. The second follows by Corollary \ref{cor:sum_of_orbits}, and the fact that $C_{\gen}\cdot Y_\pm=C_\pm\cdot Y_{\pm}=1$,  $C_\pm\cdot Y_{\mp}=0$, where the first two equalities follow from Remark \ref{rem:Btypeintnumber}.

For the proof of the last part we note first that, without loss of generality, we may assume that the action is faithful. Since it is of B-type, then (by Remark \ref{rem:Btypeequalized}) it is equalized at $Y_\pm$;  by Corollary \ref{cor:allequalized} (iv) we may conclude that the action is equalized.  

We now consider a point $y$ in an inner fixed point component $Y\subset X$. Since, by Corollary \ref{cor:codim1}, $\rk N^-(Y)=\nu^-(Y)>1$, we may consider $\Lambda\subset X^-(y)$, defined as the image of a two dimensional vector subspace of $ N^-(Y)_y$. Since action is equalized,  the geometric quotient of $\Lambda\setminus\{y\}$ by the action of $H$ is isomorphic to $\P^1$. Denoting by $\Lambda_+:=\overline{\Lambda}\cap Y_+$ we may embed $\Lambda\setminus\{y\}$ into $X^+(Y_+)$, as the complement of $\Lambda_+$ in its inverse image by the projection $X^+(Y_+)\to Y_+$. Since $X^+(Y_+)$ is a line bundle over $Y_+$, it follows  that $\Lambda_+\subset Y_+$ is isomorphic to the quotient of $\Lambda\setminus\{y\}$ by the action of $H$, hence to $\P^1$. In particular, the set $\overline{\Lambda}$, which is a projective variety by construction, is smooth, since it is  the union of the smooth varieties $\Lambda$ and $X^+(\Lambda_+)$. Since every point of $\overline{\Lambda}$ is connected to $y$ by a unique curve of type $C_+$, isomorphic to $\P^1$, we conclude that there exists a $\P^1$-bundle $S=\P(\cO_{\Lambda_+}\oplus\cO_{\Lambda_+}(e))$ over $\Lambda_+\simeq \P^1$ and a birational map $S\to \overline{\Lambda}$ sending a minimal section of $S$ over $\Lambda_+$ to the point $y$. Since $\overline{\Lambda}$ is smooth, it follows that $e=1$ and $\overline{\Lambda} \simeq \P^2$. In particular, the line $\Lambda_+$ will be numerically equivalent to the $H$-equivariant curve $C_+$. 

By Lemma \ref{lem:Cgenort}, the class $[C_{\gen}]$ is not contained in  $\iota_{*}^+\Nu(Y_+)$, therefore the subspace $\iota_{*}^+\Nu(Y_+)\cap \Nu(X)^H$ is proper, hence equal to $\RR[C_+]$. A similar argument with $C_-$ provides the last equality.
\end{proof}
 
\begin{corollary}\label{cor:sinknef}
Let $Y_-\hooklongrightarrow X \hooklongleftarrow Y_+$ be a bordism of rank $1$. Then the  divisors $Y_\pm \subset X$ have nonnegative intersection with every curve whose class is contained in $N_1(X)^H$.
\end{corollary}

\begin{proof} Assume by contradiction that there exists an effective curve $C$ such that  $[C] \in N_1(X)^H$
and $Y_- \cdot C <0$. Then $C \subset Y_-$ and, by Lemma \ref{rank1bordism-curves}, $[C] \in \R_+[C_-]$. This contradicts that $Y_- \cdot C_- >0$.
\end{proof}

\begin{corollary}\label{rank1bordism-divisors}
Let $Y_-\hooklongrightarrow X \hooklongleftarrow Y_+$ be a bordism of rank $1$, and let $[C_\pm]$ be as in Definition \ref{def:Cpm}. Then the birational
transformation  $\psi:Y_- \dasharrow{~~~~~~~} Y_+$ flips the signs of the intersections, i.e., for every $D\in \NU(Y_-)$
$$\psi_*(D)\cdot [C_+]=-D\cdot [C_-]$$
Moreover
$N_{Y_-/X}\cdot[C_-]=N_{Y_+/X}\cdot[C_+]=1$ and $\psi_*(N_{Y_-/X})=-N_{Y_+/X}$.
\end{corollary}

\begin{proof} By Lemma
\ref{lem:Cgenort} we know that $\iota_\#^-N^1(Y_-)=\iota_{\#}^+N^1(Y_+)$ is the space orthogonal to $[C_{\gen}]$.
Therefore, recalling that $[C_{\gen}] =[C_-]+[C_+]$ by  Lemma \ref{rank1bordism-curves}, $[C_+]=-[C_-]$ as linear forms on $\iota_\#^{\pm}N^1(Y_\pm)$.
Since, by Lemma \ref{lem:psistar}, one has $\psi_*=\iota^*_+\circ\iota_\#^-$
the conclusion follows.

The statement about the intersection with the normal bundles follows from Remark \ref{rem:Btypeintnumber} and the fact that  the restriction of $Y_\pm$ to itself is $N_{Y_\pm/X}$. Again by Remark \ref{rem:Btypeintnumber}, $(Y_--Y_+)\cdot C_{\gen} =0$. Since $i_-^*(Y_--Y_+)=N_{Y_-/X}$ and $i_+^*(Y_--Y_+)=-N_{Y_+/X}$, from Diagram (\ref{diag:bordism}) we get $\psi_*(N_{Y_-/X})=-N_{Y_+/X}$.
\end{proof}

\begin{lemma}\label{lem:smoothZ}
Let $Y_-\hooklongrightarrow X \hooklongleftarrow Y_+$ be a bordism of rank $1$. For every inner fixed point component $Y_i$ we set $Z_-^i:=\overline{X^+(Y_i)} \cap Y_-$. Then $Z_-^i \simeq\P(N^+(Y_i)^\vee)$, and, via this isomorphism, the lines in the fibers of $Z_-^i \to Y_j$ are curves whose numerical class is $[C_-]$. Moreover, the exceptional locus of the birational map $\psi$ is the disjoint union of the varieties $Z_-^i$ and, in particular, its reduced structure is smooth.
\end{lemma}

\begin{proof}
Without loss of generality, we may assume that the action is faithful, hence, by Remark \ref{rem:Btypeequalized} and Corollary \ref{cor:allequalized} (iv), we may assume that it is equalized. 

Consider  the set $X^+(Y_i)\setminus Y_i \subset X,$
which is isomorphic, by Theorem \ref{thm:BB_decomposition} (2), to the complementary set $$N^+(Y_i)_0\subset N^+(Y_i)$$ of the zero section of this vector bundle over $Y_i$. Since the rank of the bordism is $1$, by Lemma \ref{rank1bordism-curves}, every orbit which has the source in $Y_i$ has the
sink in $Y_-$. Hence, applying Theorem \ref{thm:BB_decomposition} (2) to $X^-(Y_-)$, we get that $X^+(Y_i)\setminus Y_i$ is isomorphic to the open set 
$$(N^-(Y_-)_{|Z_-^i})_0\subset N^-(Y_-)_{|Z_-^i},$$ 
complementary to the zero section over $Z_-^i$. This is a $\C^*$-principal bundle over $Z_-^i$, so  its projection to $Z_-^i$ is a geometric quotient, whose fibers are the orbits of the action of $H=\C^*$.  Since the $H$-action is  assumed to be equalized, it follows that $H$ acts homothetically on $N^+(Y_i)$, therefore we get $Z_-^i\simeq\P(N^+(Y_i)^\vee)$. The fact that the lines in the fibers of $Z_-^i \to Y_i$ are mapped to the curves whose numerical class is $[C_-]$ follows from the proof of Lemma \ref{rank1bordism-curves}.

The last assertion follows from the fact that, set-theoretically, the exceptional locus of $\psi$ is the disjoint union of the varieties $Z_-^i$; in fact, given a point $y\in \bigcup_i Z^i_-$, there exists a unique orbit $\C^*x$ whose sink is $y$ and whose source lies in an inner fixed point component. Since the inner fixed point components are disjoint, the claim follows.
\end{proof}



\section{Bandwidth one actions}\label{sec:BW1}

In this Section we will characterize polarized pairs $(X,L)$ admitting an action of $H=\C^*$ of bandwidth one. We start by introducing a projective construction. 

\begin{definition}\label{def:drum} Let $Y$ be a normal projective variety with $\rho_Y=2$ and two elementary 
contractions: 
$$
\xymatrix{&Y\ar[dl]_{\pi_-}\ar[dr]^{\pi_+}\\Y_-&&Y_+}
$$
Let $L_\pm\in\Pic(Y)$ be the pullbacks via $\pi_\pm$ of ample line bundles in $Y_\pm$. Then the vector bundle $\cE:=L_-\oplus L_+$ is semiample and there is a contraction $\varphi$ of $\P(\cE)$, with supporting line bundle $\cO_{\P(\cE)}(1)$. The image $\varphi(\P(\cE))$ will be called the {\em drum} associated to the triple $(Y,L_-,L_+)$.
\end{definition}

In order to describe the contraction $\varphi$, let us denote by  $s_\pm:Y\to\P(\cE)$ the sections corresponding to the quotients $\cE\to L_\pm$ and set $E_\pm=s_\pm(Y)$. Moreover let us denote by $\Gamma_\pm$ the image in $s_\pm(Y)$ of a curve of minimal degree in a fiber of $\pi_\pm$ and by $f$ a fiber of the projection $p:\P(\cE) \to Y$.
$$
\xymatrix{&\P(\cE)\ar[d]_p\ar[rr]^\varphi&&X\\&Y\ar[dl]_{\pi_-}\ar[dr]^{\pi_+}
\ar@/^1pc/[u]^{s_-}\ar@/_1pc/[u]_{s_+}
&&\\Y_-&&Y_+&}
$$
Then $\NE(\P(\cE))=\langle [\Gamma_-], [\Gamma_+], [f] \rangle$; in fact this cone is the intersection of the positive halfspaces determined by the nef line bundles $\cO_{\P(\cE)}(1),p^*L_-,p^*L_+$. In particular the exceptional locus of $\varphi$ consists of the two disjoint divisors $E_\pm$ which are contracted to varieties isomorphic to $Y_\pm\subset X$.

\begin{remark} \label{rem:action_on_drum}
Observe that there exists a natural $H$-action on $\P(\cE)$, whose fixed point locus is $s_-(Y)\sqcup s_+(Y)$. It may be defined locally: we first choose a trivialization $\{\cU_i\times \P^1,\,\,i\in I\}$ of the bundle such that the images of the sections $s_-,s_+$ are given by $\cU_i\times\{0\}$, $\cU_i\times\{\infty\}$, and define the action by $t(x,\lambda)=(x,t\lambda)$, for $(x,\lambda)\in \cU_i\times \P^1$, $t\in\C^*$. This action descends via $\varphi$ to an $H$-action on $X$, whose sink and source  are $Y_-$ and $Y_+$, respectively. \end{remark}

\begin{example}\label{ex:toricdrum}
In the context of toric geometry, a construction analogous to drums appears in the literature as the counterpart of the {\em Cayley sum} of polytopes of Convex Geometry (see for instance \cite{Ito}). We start with two birational contractions $\pi_{\pm}:Y\to Y_{\pm}$ of projective toric varieties, and with two very ample line bundles $L_\pm$ on $Y_\pm$, so that the varieties $Y_\pm$ are determined by  two polytopes $P_\pm$ in the space $M\otimes_\Z\R$ (where $M=\Mo(T)$ is the lattice of characters of the torus $T$ acting on $Y$ and $Y_\pm$ and defining them as toric varieties). The Cayley sum $P_-\star P_+$ of the two polytopes, defined as the convex hull of $(P_-\times\{0\})\cup(P_+\times\{1\})\subset (M\oplus \Z)\otimes_\Z \R$,  determines a projective toric variety $X$, which is the drum associated to $(Y,L_-,L_+)$. Note that even if $Y$ and $Y_\pm$ are smooth, the toric variety $X$ is not smooth, unless the two contractions $\pi_\pm$ are the same (in this case $X$ will be a $\P^1$-bundle over $Y_\pm$). 
\end{example}

In general, a drum will not be a smooth variety. The following Lemma gives a characterization of  {\em smooth drums}, in terms of the contractions $\pi_\pm$.

\begin{lemma}
Let $Y$ be a smooth variety of Picard number two, assume that the nef cone of $Y$ is generated by two semiample divisors $L_-$, $L_+$, and let $X$ be the drum associated to $(Y,L_-,L_+)$. Then $X$ is smooth if and only if the contractions $\pi_-:Y\to Y_-$, $\pi_+:Y\to Y_+$  associated to $L_-,L_+$ are projective bundles, and denoting by $F_-$ and $F_+$  the corresponding fibers, $\deg(L_-|_{F_+})=\deg(L_+|_{F_-})=1$. 
\end{lemma}

\begin{proof} 
Assume that the contractions $\pi_-$ and $\pi_+$ satisfy the above assumptions; 
using the canonical bundle formula for projectivizations we can write
$$
 -K_{\P(\cE)}\equiv 2D-p^*(K_Y+L_-+L_+),
$$
where $D$ is a divisor associated to the line bundle $\cO_{\P(\cE)}(1)$. From the above formula we can compute the lengths of the extremal rays contracted by $\varphi$: if $\ell_\pm$ is  the image in $s_\pm(Y)$ of a line in a fiber of $\pi_\pm$, then 
$$
 -K_{\P(\cE)} \cdot \ell_\pm = -p^*K_Y \cdot \ell_\pm -1 = \dim Y - \dim Y_\pm.
$$
It follows that  $\varphi$ is a smooth blowup by \cite[Theorem 5.1]{AO2}.
On the other hand, if $X$ is smooth, $\varphi$ is a smooth blowup by \cite[Theorem 1.1]{ESB}.  Its exceptional divisor consists of the two sections $s_{\pm}(Y)\simeq Y$, and the restrictions $\varphi_{|s_{\pm}(Y)}:s_{\pm}(Y)\to Y_{\pm}$ are the two projective bundles satisfying the required hypotheses.
\end{proof}


\begin{remark}\label{rem:equalizeddrum}
Note that, given a smooth drum $X$ as above, by construction, the action of $H$ that we have defined  satisfies that its restriction to the image into $X$ of a fiber of $\P(\cE)\to Y$ is the natural one. In particular it is equalized.
\end{remark}

For a Fano manifold $X$, its {\em index} is defined as the maximum positive integer $k$ such that $(1/k)K_X$ is a Cartier divisor. In the following Proposition we show how to write the index of a smooth drum in terms of its defining data.

\begin{proposition}\label{prop:index}
Let $X$ be a smooth drum, constructed upon a manifold $Y$ of Picard number two supporting structures of $\P^{k^-}$\!-bundle and $\P^{k^+}$\!-bundle. Then  the Fano index of $X$ is equal to $k^-\!+k^+\!+2$.
\end{proposition}

\begin{proof}

Let $D$ be a divisor associated to the line bundle $\cO_{\P(\cE)}(1)$; we can compute that: 
$$E_\pm\equiv D- p^*L_\pm.$$
Using the canonical bundle formulas for projectivizations and blowups we can write
$$
-2D+p^*(K_Y+L_-+L_+)\equiv K_{\P(\cE)}\equiv \varphi^*K_X + k^-E_-+k^+E_+.
$$
Using this formula and recalling that $K_Y\equiv-(k^-\!+1)L_- -(k^+\!+1)L_+$ we get
$$-\varphi^*K_X\equiv (k^-\!+k^+\!+2)D.$$
Since $D\cdot f=1$ the proof is finished.
\end{proof}

Note that,  
 as a direct application of Lemma \ref{lem:AMvsFM} and Corollary \ref{cor:allequalized} (ii), a bandwidth one $H$-action on a polarized pair $(X,L)$ is always equalized. 

\begin{example}\label{ex:projdrum}
The simplest example of action of bandwidth $1$ is the action on $X=\P^n$ given by $t(x_0:\dots:x_n)=(x_0:\dots:x_{m-1}:tx_{m}:\dots:tx_n)$, whose source and sink are two disjoint projective subspaces $Y_+\simeq \P^{m-1}$, $Y_-\simeq \P^{n-m}$. Blowing up the extremal fixed point components we obtain two exceptional divisors isomorphic to $Y=\P^{m-1}\times \P^{n-m}$. When $m\neq 1$ and $n\neq m$, $Y$ has Picard number two, and $\P^n$ is the drum associated to the triple $(Y,L_-,L_+)$, where $L_\pm$ are the tautological line bundles on $Y_\pm$. Note that it has been shown in \cite[Theorem~3.1]{RW} that the only example of an $H$-action of bandwidth one that has an isolated extremal fixed point is the action described above, with $m=1$. 
\end{example}

The following statement relates smooth drums with polarized pairs with an action of bandwidth one. 

\begin{theorem}\label{thm:drums} 
 Let $X$ be a smooth projective variety with $\rho_X=1$ different from the projective space, and $L$ be an ample line bundle on $X$.  Then $X$ is a (smooth) drum if and only if there exists a $\C^*$-action on $(X,L)$ of bandwidth one. 
\end{theorem}

\begin{proof}
Assume that $X$ is a smooth drum and keep the notation as in Definition \ref{def:drum}. 
We consider the $H$-action on $X$ defined in Remark \ref{rem:action_on_drum}, that is equalized (Remark \ref{rem:equalizeddrum}) and, in particular, faithful. Then, denoting by $L$ the line bundle on $X$ whose pullback to $\P(\cE)$ is $\cO_{\P(\cE)}(1)$ we have, by construction, a linearization of the $H$-action on $L$. By Lemma \ref{lem:AMvsFM}, since the action is equalized, the bandwidth $|\mu|$ of the action on $(X,L)$ is equal to the degree of $L$ on the closure $C$ of a general orbit of the action, which is the image of a fiber of $\P(\cE)\to Y$. Then $|\mu|=L\cdot C=1$. 

Conversely, let $(X,L)$  be a polarized pair supporting an (equalized) $H$-action of bandwidth one, and consider the blowup $\varphi:X_{\flat}\rightarrow X$ at the source $Y_-$ and the sink $Y_+$ of the action, with exceptional divisors $Y_-^\flat$, and $Y_+^\flat$. By Lemma \ref{lem:setup_bir} the action extends to a B-type action on $X_\flat$. Using Lemma \ref{lem:birational_map_b}, since the action has no inner fixed point components, we see that $Y_-^\flat \simeq Y_+^\flat =:Y$. Moreover, since $X$ is not a projective space, we may assume that $\dim Y_\pm>0$ (see Example \ref{ex:projdrum}), and then, by Lemma \ref{lem:Pic_extremal} (1) we have $\rho_{Y_-}=\rho_{Y_+}=1$, and $\rho_Y=2$.
 
We claim that the projective bundle structures $\pi_-\colon Y\to Y_-$ and $\pi_+\colon Y\to Y_+$ are different; if this were not the case, then the orbits joining corresponding points of two fibers of the contraction would be mapped by $\varphi$ to a positive dimensional family of curves of degree one with respect to the ample line bundle $L$, passing by two different points, against the Bend and Break Lemma (see, for instance \cite[Proposition~3.2]{De}).
\end{proof}

\begin{remark}\label{rem:pasquierkan}
One may construct many examples of smooth drums upon rational homogeneous varieties $Y$ of Picard number two with two projective bundle structures. However, these are not the only ones. In fact, there exist examples of non homogeneous varieties of Picard number two with two projective bundle structures (see \cite[Section 2]{Kan}). Moreover, we remark that in the case in which $Y$ is a homogeneous variety of Picard number two with two projective bundle structures, the drum obtained upon it can be either rational homogeneous or non homogeneous horospherical; we refer the reader to \cite{pas} for details. 
\end{remark}



\section{Atiyah flips and bordisms: the case of toric varieties}\label{sect:appendix}
In this Section we explore the relation among Atiyah flips and bordisms in the case of toric varieties. Our main motivation is to state results that can be used locally in Section \ref{sec:BW2} to extend that relation to the framework of smooth complex projective varieties. We use the language of toric geometry as, for example, in \cite{CLS}. In particular, our notation is consistent with that book. In the first subsection we briefly recall what has been done in a greater generality by Reid, \cite{ReidToric}, see also \cite[Lect.~3]{JW-Toric}. Subsection \ref{ssect:cobordism}  explains ideas of cobordism by Morelli and W{\l}odarczyk, \cite{Morelli,Wlodarczyk}. Both subsections serve as preparation to \ref{ssect:toricbordism} where we construct the toric bordism.

\subsection{Toric Atiyah flips}\label{ssec:toricAtiyah}

Let us fix a triple of integers $(r,s,n)$ with $r, s\geq 1$ and
$q:=n-1-r-s\geq 0$.  Take a lattice $N$ of rank $n+1$ generated by $e_0,\dots,e_r,
f_0,\dots,f_s, h_1,\dots,h_q$, that is 
$$N=\bigoplus_{i=0}^r \ZZ\cdot e_i\oplus\bigoplus_{j=0}^s
\ZZ\cdot f_j\oplus\bigoplus_{k=1}^q \ZZ\cdot h_k.$$
By $\Delta\subset N_\RR:=N\otimes_{\Z}\R$ we denote the convex cone generated by the
basis, that is $\Delta= \cone(e_i,f_j,h_k)$, with $i=0,\dots,r$, $j=0,\dots,s$, $k=1,\dots,q$. We also define:
$$\delta_-:=\cone(f_j),\ \ \delta_+:=\cone(e_i),\ \ \delta_0:=\cone(h_k),$$
with $i,j,k$ as above, so that
$\Delta=\delta_-+\delta_++\delta_0$.  We define, for
$i=0,\dots, r$ and $j=0,\dots,s$, the distinguished facets
of $\Delta$: $$\delta^i_-:=\cone(e_m: m\ne i)+\delta_-+\delta_0,\  {\rm
 and} \ \ \delta_+^j:=\cone(f_m: m\ne j)+\delta_++\delta_0.$$ 

We take $v:=\sum_i e_i-\sum_j f_j$ and consider the quotient lattice
$N':=N/\ZZ v$ of rank $n$. We denote by $\Delta'$ the image of $\Delta$ under the
projection $N\rightarrow N'$. By abuse, we denote by
the same names the classes of $e_i$'s, $f_j$'s and $h_k$'s in $N'$, as
well as the cones $\delta$'s defined above.  

The cone $\Delta'$ has two regular triangulations
$\Sigma_-$ and $\Sigma_+$ in which the cones of maximal dimensions
are, respectively:
$$\Sigma_+(n)=\{\delta_+^j: j=0,\dots,s \},\  {\rm and}\ \
\Sigma_-(n)=\{\delta_-^i: i=0,\dots,r \}.$$ The resulting two toric
varieties of dimension $n$ are smooth and we denote them by
$X_{\Sigma_+}$ and $X_{\Sigma_-}$, respectively. Below, whenever the
statement concerns both fans or the associated toric varieties, we
will write $\Sigma_\pm$ and $X_{\Sigma_\pm}$, respectively.

\begin{remark}\label{rem:extend_triang}
We note that the two subdivisions
$$\Delta'=\bigcup_{j=0}^s \delta^j_+ = \bigcup_{i=0}^r \delta^i_-$$
are trivial on the boundary of $\Delta'$. Therefore,
if a fan $\Sigma'_-$ in $N'$ contains a subfan with maximal cones
$\delta_-^i$ and support $\Delta'$ then it can be modified to a fan
$\Sigma'_+$ which contains subdivision of $\Delta'$ by cones
$\delta^j_+$.
\end{remark}

If $\Sigma(\Delta')$ is the fan of faces of $\Delta'$, then we have
birational contractions which come from these subdivisions
$$\begin{tikzcd}X_{\Sigma_-}\ar[r]& X_{\Sigma(\Delta')}&
  X_{\Sigma_+}\ar[l]\end{tikzcd}.$$ Their exceptional sets are $V(\delta_-)\simeq\PP^r\times\CC^q$ and
 $V(\delta_+)\simeq\PP^s\times\CC^q$, respectively, where we use the standard notation introduced in \cite{CLS} for $V(\delta_{\pm})$ to denote the closure of the orbit associated to $\delta_{\pm}$, whose fan is the star of $\delta_{\pm}$.  Moreover, the above contractions restricted to $V(\delta_{\pm})$
are the projections to $\CC^q=V(\delta_-+\delta_+)$. Indeed, the star 
of the cone $\delta_+\in\Sigma_+$ consists of maximal cones of type
$\delta_+^j\times\delta_0$ which yields the fan of $\PP^s\times\CC^q$.
A similar statement follows for $\delta_-$.

On the other hand, the relation $\sum_i e_i - \sum f_j=0$ in $N'$
shows that on $X_{\Sigma_+}$ the line bundle associated to any divisor
$V(\RR_{\geq 0}\cdot e_i)$ is $\cO(-1)$ on fibers of the projection
$\PP^s\times\CC^q=V(\delta_+)\rightarrow \CC^q$ while any divisor
$V(\RR_{\geq0}\cdot f_j)$ is associated to $\cO(1)$. In fact,
$V((\RR_{\geq0}\cdot f_j)+\delta_0)$ is a hyperplane in
$V(\delta_++\delta_0)\simeq\PP^s$ and $V(\RR_{\geq0}\cdot
e_i)+V(\RR_{\geq0}\cdot f_j)$ is a principal divisor on
$X_{\Sigma_+}$. Moreover, $V(\delta_+)$ is the complete intersection
of divisors $V(\RR_{\geq0}\cdot e_i)$ hence the normal to
$V(\delta_+)$ in $X_{\Sigma_+}$, restricted to $V(\delta_++\delta_0)\simeq\PP^s$ is
\begin{equation}(N_{V(\delta_+)/X_{\Sigma_+}})_{|\PP^s}\simeq\cO_{\P^s}(-1)^{\oplus(r+1)}.\label{eq:normalAtiyah}
\end{equation}  Similarly, on $X_{\Sigma_-}$ the
line bundle associated to the divisor $V(\RR_{\geq 0}\cdot e_i)$ is
$\cO(1)$ on fibers of the projection
$\PP^r\times\CC^q=V(\delta_-)\rightarrow \CC^q$, whereas
$V(\RR_{\geq0}\cdot f_j)$ is $\cO(-1)$. We will denote the respective
line bundles on varieties $X_{\Sigma_\pm}$ by $\cO_{\Sigma_{\pm}}(1)$
and $\cO_{\Sigma_{\pm}}(-1)$. We note that $\Pic
(X_{\Sigma_-})=\ZZ\cdot[\cO_{\Sigma_-}(1)]$, and $\Pic
(X_{\Sigma_+})=\ZZ\cdot[\cO_{\Sigma_+}(1)]$.

\begin{definition}\label{def:Atiyahtoric}
The birational transformation defined above
$$\begin{tikzcd}
X_{\Sigma_-}\arrow[leftrightarrow,dashed]{r}&X_{\Sigma_+}
\end{tikzcd}$$ 
will be called {\em toric Atiyah flip of type $(r,s,n)$}, or
{\em Atiyah flop} if $r=s$. The morphism $X_{\Sigma_-} \to X_{\Sigma(\Delta')}$ will be
called a {\em toric small contraction of Atiyah type $(r,s,n)$}. The flip is an isomorphism
outside the exceptional loci of these contractions, that is
$V(\delta_-)$ and $V(\delta_+)$, respectively; the cones $\delta_\pm$
and the varieties $V(\delta_\pm)$ will be called the {\em centers} of the
flip.
\end{definition}

Since the flip is an isomorphism in codimension $1$, the corresponding strict transform  defines an isomorphism between the divisors class groups of $X_{\Sigma_\pm}$, 
under which $\cO_{\Sigma_-}(-1)$ is sent to $\cO_{\Sigma_+}(1)$.

The Atiyah flip can be resolved by a single blowup. Namely, in the lattice
$N'$ we take a vector $u=\sum e_i =\sum f_j$ and consider the
following fan $\Sigma_\#$ in $N'_\RR$ with support $|\Sigma_\#|=\Delta'$ and
such that:
$$\Sigma_\#(n)=\{\RR_{\geq 0}\cdot  u+(\delta_-^i\cap\delta_+^j)+\delta_0:
i=0,\dots,r,\ j=0,\dots,s \}.$$
Then the resulting birational morphisms
$$\begin{tikzcd}
X_{\Sigma_-}&X_{\Sigma_\#}\arrow{l}\arrow{r}&X_{\Sigma_+}
\end{tikzcd}$$ are blowdowns, with exceptional divisors
$V(\RR_{\geq 0}\cdot u)\subset X_{\Sigma_\#}$, $V(\RR_{\geq
  0}\cdot u)\simeq \PP^r\times\PP^s\times\CC^q$, which are mapped to
$V(\delta_-)\subset X_{\Sigma_-}$ and $V(\delta_+)\subset
X_{\Sigma_+}$, respectively.

\subsection{Morelli--W{\l}odarczyk cobordism}\label{ssect:cobordism}
We use the notation of the previous subsection. The Atiyah flip can be described in terms of the action of the
1-parameter group $\lambda^v$ with $v=\sum_{i} e_i-\sum_{j} f_j$ on the toric variety
$X_{\Sigma(\Delta)} \simeq \CC^{n+1}$. Then we have a commuting
diagram associated to the projection $N\rightarrow N'$: 
$$\begin{tikzcd}
  X_{\widetilde{\Sigma}_-}\arrow[hook]{r}\arrow{d}&
  X_{\Sigma(\Delta)}\arrow[dashed]{dl}\arrow{dd}\arrow[dashed]{dr}&
  X_{\widetilde{\Sigma}_+}\arrow[hook']{l}\arrow{d}\\
  X_{\Sigma_-}\arrow{dr}\arrow[leftrightarrow,dashed]{rr}&&X_{\Sigma_+}\arrow{dl}\\
  &X_{\Sigma(\Delta')}& \end{tikzcd}$$ Here the solid arrow $X_{\Sigma(\Delta)}\rightarrow
X_{\Sigma(\Delta')}$ is the categorical quotient of affine
varieties, while the dashed arrows from
$X_{\Sigma(\Delta)}$ are good quotients on open subsets $X_{\widetilde{\Sigma}_\pm}$ of this variety.
In fact, the fan $\Sigma_-$ in $N'_\RR$ which is a division
of the cone $\Delta'$ determines a subfan $\widetilde{\Sigma}_-$ of
$\Sigma(\Delta)$ so that the projection $N\rightarrow N'$ yields the
map of fans $\widetilde{\Sigma}_-\rightarrow\Sigma_-$, which is
bijective on cones. Note that $\widetilde{\Sigma}_-$ has cones of
dimension $\leq n$ and $\widetilde{\Sigma}_-(n)=\{\delta_-^i:
i=0,\dots,r \}$ where $\delta_-^i$ are cones in $N$. The associated
morphism of varieties $X_{\widetilde{\Sigma}_-}\rightarrow
X_{\Sigma_-}$ is a $\CC^*$-bundle so that the map is a geometric
quotient. The same holds for the fan $\Sigma_+$ for which we take the
respective fan $\widetilde{\Sigma}_+$ in $N$, and we get
$X_{\widetilde{\Sigma}_+}\rightarrow X_{\Sigma_+}$.

The two quotient morphisms in the upper part of the
above diagram 
$$\begin{tikzcd}
X_{\Sigma_-}&X_{\widetilde{\Sigma}_-}\arrow{l}\arrow[hook]{r}&
X_{\Sigma(\Delta)}&X_{\widetilde{\Sigma}_+}\arrow[hook']{l}\arrow{r}&X_{\Sigma_+}
\end{tikzcd}$$ will
be called the \textit{Morelli--W{\l}odarczyk cobordism associated to the Atiyah
flip}. A schematic description of the map of fans of the cobordism
associated to the classical 3-dimensional flop is presented in the
following diagram, with the central tetrahedron representing a section
of the 4-dimensional simplicial cone $\cone(e_0,e_1,f_0,f_1)$:
\par\medskip
\centerline{
\begin{tikzpicture}[scale=0.6]
\draw [green] (-7.5,2)--(-9,1)--(-8.5,-2)--(-7,-1)--(-7.5,2);
\draw [blue] (-9,1)--(-7,-1);
\draw [green] (-5.2,1)--(-3.1,2)--(-4.8,-1)--(-3.2,-2);
\draw [green, dashed] (-5.2,1)--(-3.2,-2);
\draw [green] (-4.35,-0.26)--(-3.2,-2);
\draw [blue] (-5.2,1)--(-4.8,-1);
\draw [green] (8.5,2)--(7,1)--(7.5,-2)--(9,-1)--(8.5,2);
\draw [red] (8.5,2)--(7.5,-2);
\draw [green] (3.2,1)--(5.1,2)--(3.2,-1)--(4.8,-2);
\draw [green, dashed] (3.2,1)--(4.8,-2);
\draw [green] (3.8,-0.14)--(3.2,1);
\draw [red] (5.1,2)--(4.8,-2);
\draw [green] (-1.2,1)--(1.1,2)--(-0.8,-1)--(0.8,-2);
\draw [green, dashed] (-1.2,1)--(0.8,-2);
\draw [red] (1.1,2)--(0.8,-2);
\draw [blue] (-1.2,1)--(-0.8,-1);
\node at (1.5,2) {$e_0$}; \node at (1.3,-2) {$e_1$};
\node at (-1.6,1) {$f_0$}; \node at (-1.2,-1) {$f_1$};
\draw [->] (-3,0)--(-2,0);
\draw [->] (3,0)--(2,0);
\draw [->] (-5.5,0)--(-6.5,0);
\draw [->] (5.5,0)--(6.5,0);
\end{tikzpicture}
}
\par

\begin{remark}\label{piecewise-linear} If we decompose
  $N_\RR$ as  $N_\RR=N'_\RR\oplus\RR\cdot v$, where the splitting is given by taking
  $e_1,\dots,e_r,f_0,\dots,f_s$ as the basis of $N'$, then
  the supports $|\widetilde{\Sigma}_\pm|$ are graphs of functions $\epsilon^0_\pm:
  |\Delta'| \rightarrow \RR$ defined by the projection on $\RR\cdot v$,  which are linear on cones in the
  respective fan $\Sigma_\pm$ and such that $\epsilon_\pm^0(e_0)=1$
  and $\epsilon_\pm^0$ on all $f_j$'s and the other $e_i$'s is zero.
  In the language of \cite{CLS}, the function $\epsilon^0_\pm$ is the
  {\em supporting function} of the divisor $-V(\RR_{\geq0}\cdot e_0)$
  on the respective fan $\Sigma_\pm$, see \cite[Sect. 4.2]{CLS}.
  Since the divisor $-V(\RR_{\geq0}\cdot e_0)$ is associated to
  $\cO_{\Sigma_-}(1)$ on $X_{\Sigma_-}$, and $\cO_{\Sigma_+}(-1)$ on
  $X_{\Sigma_+}$, the respective morphism $X_{\widetilde{\Sigma}_\pm}\rightarrow
  X_{\Sigma_\pm}$ is a $\CC^*$-bundle associated to each of these line
  bundles.
\end{remark}

Finally, we note that Morelli, \cite{Morelli}, in the toric case and
W{\l}odarczyk, in general, introduced the notion of birational
cobordism, see \cite[Def.~3]{Wlodarczyk}, which encompasses the
construction explained above.

\subsection{Toric bordism}\label{ssect:toricbordism}
The $\CC^*$-bundle $X_{\widetilde{\Sigma}_-}\rightarrow X_{\Sigma_-}$
can be extended to a line bundle over $X_{\Sigma_-}$ by adding a zero
section. This can be done in two ways, depending on the choice of the
$\CC^*$-action which determines the ``zero'' limits of the action. In
toric terms this is described by two fans $\widehat{\Sigma}_-^+$ and
$\widehat{\Sigma}_-^-$ in $N_\RR$, which are obtained by adding to
$\widetilde{\Sigma}_-$ a ray generated by $-v$ or $v$, respectively, so
that their sets of cones of maximal dimension are, respectively:
$$\widehat{\Sigma}_-^+(n+1)=\{\delta_-^i+\RR_{\geq 0}\cdot (-v)\},\ {\rm and}
\ \ \widehat{\Sigma}_-^-(n+1)=\{\delta_-^i+\RR_{\geq 0}\cdot v\},$$
where $i=0,\dots,r$. Similarly we obtain two fans $\widehat{\Sigma}_+^+$,
$\widehat{\Sigma}_+^-$ in $N_\RR$, by
adding to $\widetilde{\Sigma}_+$ a ray generated by $v$ or $-v$,
respectively:
$$\widehat{\Sigma}_+^+(n+1)=\{\delta_+^j+\RR_{\geq 0}\cdot v\},\  {\rm and}
\ \ \widehat{\Sigma}_+^-(n+1)=\{\delta_+^j+\RR_{\geq 0}\cdot(-v)\},$$
where $j=0,\dots,s$.
\begin{lemma}\label{lem:toriclinebundle}
  $X_{\widehat{\Sigma}_\pm^+}\rightarrow X_{\Sigma_\pm}$ is the total
  space of the line bundle $\cO_{\Sigma_\pm}(1)$, and
  $X_{\widehat{\Sigma}_\pm^-}\rightarrow X_{\Sigma_\pm}$ is the total
  space of $\cO_{\Sigma_\pm}(-1)$.
\end{lemma}
\begin{proof}
  The claim follows by the above Remark \ref{piecewise-linear} and
  the toric description of a line bundle associated to the respective
  $\CC^*$-bundle.
\end{proof}

Since $\widehat{\Sigma}_\pm^-\cap \widehat{\Sigma}_\pm^+=\widetilde{\Sigma}_\pm$, that correspond to the $\C^*$-bundles $X_{\widetilde{\Sigma}_\pm}\to X_{{\Sigma}_\pm}$, we may observe the following:
\begin{corollary}\label{lem:toricprojbundle}
  The fan $\widehat{\Sigma}_\pm= \widehat{\Sigma}_\pm^- \cup
  \widehat{\Sigma}_\pm^+$ defines a $\PP^1$-bundle over
  $X_{\Sigma_\pm}$, isomorphic to 
  $\pi_{\Sigma_\pm}:
  \PP(\cO\oplus\cO_{\Sigma_\pm}(1)) \rightarrow X_{\Sigma_\pm}$.
\end{corollary}

The $\PP^1$-bundle $\pi_{\Sigma_\pm}$ has two sections
associated to the splitting of
the bundle $\cO\oplus\cO_{\Sigma_\pm}(1)$, whose associated divisors, that we denote by $D_{\Sigma_\pm}^0,
D_{\Sigma_\pm}^1 \subset X_{\widehat{\Sigma}_\pm}$, have normal bundles $\cO_{\Sigma_\pm}(1)$ and
$\cO_{\Sigma_\pm}(-1)$, respectively. In toric terms we may describe them as: 
$$\begin{array}{l}D_{\Sigma_-}^0=V(\RR_{\geq
  0}\cdot(-v)),\,\, D_{\Sigma_-}^1=V(\RR_{\geq 0}\cdot(v)) \subset
X_{\widehat{\Sigma}_-},\\[2pt]
D_{\Sigma_+}^0=V(\RR_{\geq 0}\cdot(v)),\,\,
D_{\Sigma_+}^1=V(\RR_{\geq 0}\cdot(-v)) \subset X_{\widehat{\Sigma}_+},\end{array}$$ and these
divisors are the components of the fixed point locus of the action of
the 1-parameter group $\lambda^v$.

\par\medskip

Now let us deal with the fan $\widehat{\Sigma}_-$ and the $\PP^1$-bundle
$\pi_{\Sigma_-}: X_{\widehat{\Sigma}_-}\rightarrow X_{\Sigma_-}$.  By
construction $\delta_-+\RR_{\geq 0}\cdot
v=\cone(f_0,\dots,f_s,v)\in\widehat{\Sigma}_-$, and the star of the cone
$\delta_-+\RR_{\geq 0}\cdot v$ in $\Sigma_-$ contains the cones
$\delta_-^i+\RR_{\geq 0}\cdot v$ for $i=0,\dots, r$.  Thus
$V(\cone(f_0,\dots,f_s,v))\simeq \PP^{r}\times\CC^q\subset D_{\Sigma_-}^1$. Since in
the lattice $N$ we have the relation $f_0+\cdots+f_s+v=e_0+\cdots+e_r$, we
are in the situation of Section \ref{ssec:toricAtiyah} and we may conclude the
following:

\begin{lemma}\label{lem:fandivison-flipbordism}
  The cone $\Delta+\RR_{\geq 0}\cdot v$ admits the following two
  regular triangulations in $N_\RR$ which are trivial on the boundary
  of this cone:
  $$\Delta+\RR_{\geq 0}\cdot v =
  \bigcup_{i=0}^r (\delta_-^i+\RR_{\geq 0}\cdot v) =
  \bigcup_{j=0}^s (\delta_+^j+\RR_{\geq 0}\cdot v) \cup \Delta.$$
\end{lemma}

Let us now set:
$$\widehat{\Sigma}:=\widehat{\Sigma}_-^+\cup\Sigma(\Delta)
\cup\widehat{\Sigma}_+^+,$$ which is a fan in $N_\R$, and consider the corresponding toric variety:
$$X_{\widehat{\Sigma}}=X_{\widehat{\Sigma}_-^+}\cup X_{\Sigma(\Delta)} \cup
X_{\widehat{\Sigma}_+^+}.$$

\begin{proposition}\label{prop:toric_bordism-construction}
The variety
$X_{\widehat{\Sigma}}$ admits the action of a 1-parameter group $\lambda^v$, with
sink and source  $X_{\Sigma_-}$ and $X_{\Sigma_+}$;  the only inner fixed
point component  is $V(\delta_-+\delta_+)\simeq\CC^q$. Moreover, the variety
$X_{\widehat{\Sigma}}$ admits two $\lambda^v$-equivariant Atiyah flips
$$\begin{tikzcd}
X_{\widehat{\Sigma}_-}\arrow[leftrightarrow,dashed]{rr}\arrow{dr}&&
X_{\widehat{\Sigma}}\arrow[leftrightarrow,dashed]{rr}\arrow{dl}\arrow{dr}&&
X_{\widehat{\Sigma}_+}\arrow{dl}\\
&X_{\widehat{\Sigma}'_-}&&X_{\widehat{\Sigma}'_+} \end{tikzcd}$$ where
$\widehat{\Sigma}'_-=\widehat{\Sigma}_-^+\cup\Sigma(\Delta+\RR_{\geq 0}\cdot v)$ and 
$\widehat{\Sigma}'_+=\widehat{\Sigma}_+^+\cup\Sigma(\Delta+\RR_{\geq 0}\cdot (-v))$.
\end{proposition}

\begin{proof} The two triangulations of the cone $\Delta+\RR_{\geq 0}\cdot v$ provide a toric Atiyah flip  of type $(r,s+1,n+1)$ of the variety  $X_{\widehat{\Sigma}_-^-}$ centered at $V(\delta_-+\RR_{\geq 0}\cdot v)$.
By Remark \ref{rem:extend_triang}, since $\widehat{\Sigma}_-=\widehat{\Sigma}_-^+\cup \widehat{\Sigma}_-^-$, we may extend it to a toric Atiyah flip of $X_{\widehat{\Sigma}_-}$ and the
first part of the statement follows.

The fixed point locus of the action of
$\lambda^v$ can be computed by looking at its restriction to each of the three torus invariant covering sets. The varieties $X_{\widehat{\Sigma}_-^+}$ and $X_{\widehat{\Sigma}_+^+}$ are total spaces of line bundles on which $\lambda^v$ acts by
homotheties. Therefore their fixed point components, which will be the sink and the source of the action on $X_{\widehat{\Sigma}}$, are the zero sections of the bundles, which are isomorphic to $X_{\Sigma_-}$ and $X_{\Sigma_+}$, respectively.  The
third covering set is the affine space $V(\Delta)$ on which the statement about
the only inner fixed point component can be verified easily. In view of Corollary 
\ref{lem:toricprojbundle}, the last statement is the content of  Lemma
\ref{lem:fandivison-flipbordism} applied for $\widehat{\Sigma}_-$ (and its
counterpart for $\widehat{\Sigma}_+$) which is the following:
$$\Delta+\RR_{\geq 0}\cdot (-v) = \bigcup_{j=0}^s (\delta_+^j+\RR_{\geq
0}\cdot (-v)) = \bigcup_{i=0}^r (\delta_-^i+\RR_{\geq 0}\cdot (-v)) \cup
\Delta.$$
\end{proof}

Thus we have the following
diagram, equivariant with respect to the action of $\lambda^v$, which is
built upon the Atiyah flip $\begin{tikzcd}
  X_{\Sigma_-}\arrow[leftrightarrow,dashed]{r}&X_{\Sigma_+}:
\end{tikzcd}$
$$\begin{tikzcd}
X_{\widetilde{\Sigma}_-}\arrow[hook]{r}\arrow{d}&
  X_{\Sigma(\Delta)}\arrow[dashed]{dl}\arrow{dd}\arrow[dashed]{dr}&
  X_{\widetilde{\Sigma}_+}\arrow[hook']{l}\arrow{d}\\
  X_{\Sigma_-}\arrow[hook]{dr}\arrow[hook]{d}\arrow[leftrightarrow,dashed]{rr}
  &&X_{\Sigma_+}\arrow[hook']{dl}\arrow[hook']{d}\\
  X_{\widehat{\Sigma}_-^+}\arrow[hook]{r}\arrow[bend left=30]{u}&X_{\widehat{\Sigma}}&
  X_{\widehat{\Sigma}_+^+}\arrow[hook']{l}\arrow[bend right=30]{u}
\end{tikzcd}$$

The central dashed two-end arrow is the 
Atiyah flip, and the rational maps in the upper part come from the
Morelli--W{\l}odarczyk cobordism induced by the action of
$\lambda^v$. The hooked arrows are embeddings, and the only upward
arrows are projections of line bundles which, on the intersection
$X_{\widehat{\Sigma}_\pm^+}\cap X_{\Sigma(\Delta)} =
X_{\widetilde{\Sigma}_\pm}$, are the quotients in the upper part of the
diagram.
\begin{definition}\label{def:toric_bordism}
The above  $\lambda^v$-equivariant embeddings
 $$\begin{tikzcd}
   X_{\Sigma_-}\arrow[hook]{r}&X_{\widehat{\Sigma}}&
   X_{\Sigma_+}\arrow[hook']{l}
 \end{tikzcd}$$ will be called the {\em toric bordism} associated to the toric Atiyah flip.
\end{definition}
 
 We now define a line bundle $L$ over $X_{\widehat{\Sigma}}$ which is
 associated to the sum of the $\lambda^v$-invariant divisors
 $V(\RR_{\geq 0}\cdot v)+V(\RR_{\geq 0}\cdot(-v))$. We note that the variety $X_{\widehat{\Sigma}}$ is not complete, and so 
Definition \ref{def:bandwidth} does not apply.  Nevertheless we may still claim that the pair $(X,L)$, together with the action of $\lambda^v$, has bandwidth
$2$, in the following sense:

\begin{lemma}
  In the above situation, $L_{|X_{\Sigma_\pm}}=\cO_{\Sigma_{\pm}}(1)$
  and the natural linearization of the action of $\lambda^v$ on $L$
  assigns to source and sink the values $+1$ and $-1$, and for
  $V(\delta_-+\delta_+)$ value 0.
\end{lemma}

Finally, we recall that the results contained in this Section are meant to be used locally in Section \ref{sec:BW2}.  For the reader's convenience we rephrase the main results of this Section using the (non-toric) notation that we will use in Section \ref{sec:BW2}.  We set: $$Y'=X_{\Sigma(\Delta')},\,\,Y_\pm=X_{\Sigma_\pm},\,\,X_\pm=X_{\widehat{\Sigma}_\pm},\,\,X'_\pm=X_{\widehat{\Sigma}'_\pm},\mbox{ and }X=X_{\widehat{\Sigma}}.$$

\begin{corollary}\label{cor:toricbordism}
Given a toric Atiyah flip
$$\begin{tikzcd}
Y_-\arrow[leftrightarrow,dashed]{rr}\arrow{dr}&&Y_+\arrow{dl}\\
&Y'&\end{tikzcd}$$
there exists a toric bordism
$\begin{tikzcd}
  Y_-\arrow[hook]{r}&X&Y_+\arrow[hook']{l}
\end{tikzcd}$
with two Atiyah flips
$$\begin{tikzcd}
X_-\arrow[leftrightarrow,dashed]{rr}\arrow{dr}&&
X\arrow[leftrightarrow,dashed]{rr}\arrow{dl}\arrow{dr}&&
X_+\arrow{dl}\\
&X'_-&&X'_+ \end{tikzcd}$$
such that $X_\pm$  are $\PP^1$-bundles over $Y_\pm$.
\end{corollary}



\section{Atiyah flips and bordisms: the case of projective varieties}\label{sec:BW2}

In this section we study the close relation among bordisms of rank one and Atiyah flips in the framework of smooth complex projective varieties, based on the toric constructions discussed in Section \ref{sect:appendix}; the  main results of the section are stated in  Theorems \ref{flip=>bordism}  and \ref{bordism=>unique}.




\begin{definition}\label{def:Atiyah.assumptions}
An {\em Atiyah flip} 
is a birational transformation fitting in a commutative diagram
of 
projective varieties
$$\begin{tikzcd}
Y_-\arrow[rightarrow,dashed,"\psi"]{rr}\arrow["\varphi_-"']{dr}
&&Y_+\arrow["\varphi_+"]{dl}\\
&Y'&
\end{tikzcd}$$
satisfying the following conditions:
\begin{enumerate}[leftmargin=*]
\item $Y_-, Y_+$ are smooth, $Y'$ is normal, $\varphi_-, \varphi_+$
are birational, proper and surjective morphisms.
 
\item The morphisms $\varphi_-$, $\varphi_+$
can be locally (in analytic or etal\'e topology)
identified with  toric small contractions of Atiyah type (see Definition \ref{def:Atiyahtoric}).
 In particular, the exceptional loci
 $Z_\pm:=\Exc(\varphi_\pm)$ will be smooth varieties, 
possibly disconnected: their irreducible components are in one to one correspondence, and we denote them by $Z_-^j$, $Z_+^j$, $j\in J$, respectively; for each $j\in J$, the image $\varphi_{\pm}(Z_{\pm})$ is an irreducible component $Y^j_0$ of the exceptional locus $Y_0:=\Exc(\varphi_\pm^{-1})$, and the restrictions $\varphi_-: Z_-^j\rightarrow  Y_0^j$, $\varphi_+: Z_+^j\rightarrow Y_0^j$ are projective 
bundles.
\item There exist  $\varphi_{\pm}$-ample line bundles $\cN_\pm$ on
$Y_\pm$ such that
\begin{enumerate}
\item $\Pic (Y_\pm)=\varphi_\pm^*\Pic (Y')\oplus\ZZ\cdot \cN_\pm$;
\item The restriction of $\cN_\pm$ to every fiber of $\varphi_\pm: Z_\pm^j\rightarrow Y_0^j$ is $\cO(1)$;
\item $\psi_*(\cN_-)=-\cN_+$.
\end{enumerate}
\end{enumerate}
In order to encompass the data which is essential for our definition
we will write $$\begin{tikzcd}[cramped]
  (Y_\pm,\cN_\pm)\arrow["\varphi_\pm"]{r}&Y'\supset Y_0
\end{tikzcd}$$
for the Atiyah flip $\psi$ 
defined above. We will say that $\varphi_\pm$ 
are {\em small contractions of Atiyah type}. 
\end{definition}

\begin{remark}\label{rem:Atiyah.assumptions}
We note that the first two conditions are local. In particular, for every component $Z^j_{-}\subset Z_-$, the restrictions $\varphi_-:Z_-^j\rightarrow Y_0^j$, $\varphi_+:Z_+^j\rightarrow Y_0^j$ are $\P^{r_j}$-bundles and $\P^{s_j}$-bundles, respectively, with $r_j=\codim(Z^j_{+},Y_+)-1$, $s_j=\codim(Z^j_{-},Y_-)-1$. However, the values $r_j,s_j$ may depend on $j$. In this sense a global Atiyah type, as introduced for the toric Atiyah flips in Section \ref{ssec:toricAtiyah}, is not defined. However, note that condition (2) implies  (see Equation (\ref{eq:normalAtiyah}) in Section \ref{ssec:toricAtiyah}, noticing that locally $Z^j_- \subset Y_- $ coincides with $V(\delta_-) \subset X_{\Sigma_-}$) that the normal of a fiber of $\varphi_-:Z_-^j\rightarrow Y_0^j$
is isomorphic to $$\cO_{\P^{r_j}}(-1)^{\oplus(s_j+1)}\oplus\cO_{\P^{r_j}}^{\oplus(\dim Y_--r_j-s_j-1)}.$$ In particular, by the adjunction formula, $(K_{Y_-})_{|\P^{r_j}}=\cO_{\P^s}(s_j-r_j)$. But the line bundle  $\cN_-$ of condition (3) is defined globally, and so we may write
 $$K_{Y_-}+(r_j-s_j)\cN_-\in\varphi_-^*\Pic (Y'),$$ and using condition (3b) we conclude that the integer 
 $d:=s_j-r_j$ does not depend on $j$. 
 In particular, if $s_j-r_j\geq 0$ then $K_{Y_-}$ is $\varphi_-$-nef
and $-K_{Y_+}$ is $\varphi_+$-nef.
\end{remark}

The following result shows that Atiyah flips can be realized geometrically by means of bordisms of rank one.

\begin{theorem}\label{flip=>bordism}
Suppose that 
$\begin{tikzcd}[cramped](Y_\pm,\cN_\pm)\arrow["\varphi_\pm", shift right = 1]{r}&Y'\supset Y_0 \end{tikzcd}$
 is an Atiyah flip, and let $n=\dim Y_\pm $. 
 Then there exists
a smooth variety $X$ of dimension $n+1$ with a faithful action of
$H=\CC^*$, such that
\begin{enumerate}[leftmargin=*]
\item  $Y_-\hooklongrightarrow X \hooklongleftarrow Y_+$
  is a bordism of rank $1$,
\item the inner fixed point components of the action are isomorphic to the irreducible components $Y_0^j$ of $Y_0$,
\item there is an isomorphism $N_{Y_\pm/X}\simeq \cN_\pm$,
\item the birational transformation $\psi:Y_-\dashrightarrow Y_+$ 
induced by the bordism is the Atiyah flip $\begin{tikzcd}[cramped](Y_\pm,\cN_\pm)\arrow["\varphi_\pm", shift right = 1]{r}&Y'\supset Y_0 \end{tikzcd}$.
\end{enumerate}

\end{theorem}
\begin{proof}
This proof is a global version of the construction performed in the toric setting in Section \ref{ssect:toricbordism} (see Corollary \ref{cor:toricbordism}). We take the projective bundle
$$\begin{tikzcd}\pi_-: X_-:=\PP_{Y_-}(\cO\oplus\cN_-)\arrow{r}&
Y_-\end{tikzcd}$$
with tautological bundle $\cO(1)$, that is $(\pi_-)_*\cO(1)=\cO\oplus\cN_-$,
and consider its two sections $Y_-^0,\ Y_-^1$ associated to the splitting; then
$j^*_0\cO(1)=\cO$, $j_1^*\cO(1)=\cN_-$, where $j_0$ and $j_1$ denote the inclusions. Note that the normal bundles of  $Y_-^0,\ Y_-^1$ in $X_-$ are $-\cN_{-}$ and $\cN_{-}$ respectively. 

First we show that there exists a small contraction $\phi_{-}^-\colon X_{-}\to X'_{-}$ with exceptional locus $j_0(Z_{-})$, where $Z_- \subset Y_-$ is the exceptional locus of $\varphi_-$.
Note that, by construction, $\cO(Y_-^1) =\cO(1)$ and 
\begin{equation}\label{eq:531}
Y_-^1-Y_-^0=\pi_-^*\cN_- \quad \mbox{in } \Pic (X_-).
\end{equation}
Let us consider now a sufficiently ample line bundle $\cM'$ on $Y'$. Since, by the definition of Atiyah flip, $\cN_-$ is $\varphi_-$-ample, then we may assume that $\varphi_-^*\cM' + \cN_-$ is ample on $Y_-$. Set
 $$\cM_-:=\pi_-^*(\varphi_-^*\cM')+Y^1_-.$$ 
By definition it is  nef on $X_-$ (since its pushforward is $\varphi^*_{-}\cM' \oplus (\varphi^*_{-}\cM'+\cN_-)$). Let us prove that, for a sufficiently ample $\cM'$,  $\cM_-$ is semiample.

To this end, we note that we can assume that $\cM'$ is
basepoint free on $Y'$, so for a positive integer $a$  the base points of $a\cM_-$
are contained in the divisor $Y_-^1$. We will conclude by showing that the base points of $a\cM_-$
are also contained in the divisor $Y_-^0$.
By equation (\ref{eq:531}) we can write:
$$a\cM_-=a\pi_-^*(\varphi_-^*\cM'+\cN_-)+aY^0_-,$$  
so it is enough to observe that, since $\varphi_-^*\cM'+\cN_-$ is ample, $a\pi_-^*(\varphi_-^*\cM'+\cN_-)$ is globally generated for $a \gg 0$.

Thus we have shown that $\cM_-$ is the supporting divisor of a contraction $\phi_-^-:X_-\rightarrow X_-'$. 
Since $j_1^*Y_-^1=\cN_-$, then $\cM_-$ is ample on
$Y_-^1$,  so the restriction of $\phi_-^-$ to $Y_-^1$ is an embedding; moreover $j_0^*Y_-^1$ is trivial, hence the restriction of $\cM_-$ to $Y_-^0$ is $\varphi_-^*\cM'$, and  the restriction of $\phi_-^-$ to $Y_-^0$ is $\varphi_{-}$. Furthermore, since $\cM_-$ is $\pi_-$-ample, then $\Exc(\phi_-^-) =j_0(\Exc(\varphi_-))$, and the composition $\varphi_-\circ\pi_-$ factors via $\phi_-^-$, 
so that we have a morphism $\Pi_-: X_-'\rightarrow Y'$ fitting in the commutative diagram:
$$
\xymatrix{X_-\ar[r]^{\phi_-^-}\ar[d]_{\pi_-}&X'_-\ar[d]^{\Pi_-}\\
Y_-\ar[r]_{\varphi_-}&Y'}
$$

We will show now that the small contraction $\phi_-^-$ admits a flip $\psi_-:X_-\dashrightarrow X$.  
Note that if such a flip exists, the variety $X$ can be described as $X:=\Proj_{X'_-} (\bigoplus_{m\geq 0} (\phi_-^-)_*\cO(-mY_-^0))$. In particular, its existence can be proved locally analytically around every point of $X'_-$, and this holds because the variety $X_-$   
coincides locally with the toric variety $X_{\widehat\Sigma_-}$ from Corollary \ref{lem:toricprojbundle}, and the restriction of $\phi_-^-$ to $X_{\widehat\Sigma_-}$ coincides locally with the small contraction of $X_{\widehat\Sigma_-}$ whose flip is the variety $X_{\widehat{\Sigma}}$  (see also Corollary \ref{cor:toricbordism}).

Let us denote the corresponding small contraction of $X$ by $\phi_-^+:X\to X'_-$, so that we have a commutative diagram:
$$\begin{tikzcd}
X_-\arrow[rightarrow,dashed,"\psi_-"]{rr}\arrow["\phi_-^-"']{dr}
&&X\arrow["\phi_-^+"]{dl}\\
&X'_-&
\end{tikzcd}$$
The strict transforms in $X$ of $Y_-^1$ and $Y_-^0$ are isomorphic, respectively, to $Y_-$ and $Y_+$, so, abusing notation, we will denote them by $Y_-,Y_+\subset X$.


Let us  show that $X$ supports a bordism of rank one with the properties listed in the statement. We take the $\CC^*$-action on $\cO\oplus\cN_-$ with weights $0$ and $1$, which descends to a faithful action on 
$\PP_{Y_-}(\cO\oplus\cN_-)$, with sink $Y^1_-$ and source $Y^0_-$. This is a bordism of rank zero by Lemma \ref{lem:rank0bordism}. Since the exceptional locus of $\phi_-^-$ is contained in a fixed point component, the $H$-action descends  to an action on $X'_-$ and, since $Y_-^0$ is $H$-invariant, it extends to the $\cO_{X'_-}$-algebra $\bigoplus_{m\geq 0} (\phi_-^-)_*\cO(-mY_-^0)$ and, subsequently, to its relative projectivization $X$. By construction, the map $\psi_-:X_-\dashrightarrow X$ is $H$-equivariant (on the open set where it is defined), and the sink and source of the $H$-action on $X$ are $Y_-$, and $Y_+$. 


Since $X$ was constructed by glueing toric bordisms, the inner fixed point components of the $H$-action in $X$ are obtained by glueing analytic sets biholomorphic to the sets $V(\delta_-+\delta_+)$ of Proposition \ref{prop:toric_bordism-construction}. In particular, we have that the only inner fixed point components of the action are isomorphic to the connected components $Y_0^j$, which shows (2). 

The same applies to the sets $\overline{X^\pm(Y_0^j)}\cap Y_{\pm}$, so that we may conclude that $\bigcup_j \overline{X^\pm(Y_0^j)}\cap Y_{\pm}=Z_{\pm}$. Moreover, the irreducible components of the indeterminacy locus of $\psi_-$ are $Z^j_-\subset Y_-\subset X_-$, with $\codim(Z^j_-, X_-)=\codim(Z^j_-,Y_-)+1>1$; on the other hand, the irreducible components of the indeterminacy locus of $\psi^{-1}_-$ are  $\overline{X^\pm(Y_0^j)}\subset X$, and so  $\codim(\overline{X^\pm(Y_0^j)}, X)=\codim(Z^j_+,Y_+)>1$, and we conclude that the action of $H$ on $X$ is a bordism, by definition. Moreover, since this tells us that the exceptional locus of the birational transformation $\psi: Y_-\dashrightarrow Y_+$ induced by the bordism coincides with the exceptional locus of the original Atiyah flip, item (4) of the statement follows. 

Every orbit of the action is contracted by $\Pi_- \circ \phi_-^+$, hence the pullback of $\Pic (Y')$, which has codimension two in $\Pic (X)$, is $H$-trivial, 

 so the rank of the bordism is at most one; it is in fact one since there exist inner fixed point components (see Lemma \ref{lem:rank0bordism}). This concludes (1).

We finish by proving (3). Since $\phi_-^-$ is an isomorphism in a neighborhood of $Y_-^1$, it follows that we have isomorphisms $N_{Y_-/X} \simeq N_{Y_-^1/X_-} \simeq \cN_-$. On the other hand the restriction of $\psi_-$ to $Y_-^0$ is the birational transformation of the original flip, hence, by property $(3\text{c})$ in Definition \ref{def:Atiyah.assumptions} we have $N_{Y_+/X} \simeq \psi_{-*}(N_{Y_-^0/X_-}) \simeq \psi_{-*}(-\cN_-) \simeq \cN_+$.
\end{proof}

\begin{corollary}
In the setup of Theorem \ref{flip=>bordism}, there exists an Atiyah flip: 
$$\begin{tikzcd}
X_-=\PP_{Y_-}(\cO\oplus\cN_-)\arrow[rightarrow,dashed,"\psi_-"]{rr}\arrow["\phi_-^-"',pos=0.40]{dr}
&[-25pt]&X\arrow["\phi_-^+"]{dl}\\
&X'_-&
\end{tikzcd}$$
\end{corollary}

\begin{proof}
We have already shown in the proof of Theorem \ref{flip=>bordism} that $\psi_-$ is locally a toric Atiyah flip, 
so we only need to check the condition (3) of the definition.

By Equation (\ref{eq:531}) we have 
$$Y_-^0|_{j_0(Z_-)}\simeq -\pi^*_-\cN_-|_{j_0(Z_-)} \simeq -\cN_-|_{Z_-},$$
hence $Y^0_-$ restricts to $\cO(-1)$ on the fibers of $\phi_-^-$ and, setting $\cN_-^-:=-Y_-^0 \in \Pic (X_-)$ we can write
$$\Pic (X_-)=\phi_-^{-*}\Pic (X_-')\oplus\ZZ[\cN_-^-].$$
On the other hand, we set $\cN_-^+:=-\psi_{-*}(\cN_-^-)$. As we have seen, the restriction of $\cN_-^-$ to $Y_-^0\subset X_-$ is $\cN_-$, hence its  restriction to each fiber of the contraction $\phi_-^-$ is isomorphic to $\cO(1)$; on the other hand the restriction of $\cN_-^+$ to $Y_+\subset X$ equals $-\psi_*((\cN_-^-)_{|Y_-^0})=-\psi_*(\cN_-)=\cN_+$ and so its restriction to each fiber of $\phi_-^+$ is also isomorphic to  $\cO(1)$. 
\end{proof}

\begin{remark}\label{rem:atiyaflip_twoways}
In the proof of the Theorem \ref{flip=>bordism} we have constructed the variety $X$ as the Atiyah flip of a $\P^1$-bundle $\PP_{Y_-}(\cO\oplus\cN_-)$. If we start from $Y_+$, we will obtain the same variety $X$ and, in particular we have a diagram of Atiyah flips:
$$\begin{tikzcd}
\PP_{Y_-}(\cO\oplus\cN_-)\arrow[rightarrow,dashed,"\psi_-"]{rr}\arrow["\phi_-^-"',pos=0.40]{dr}
&[-25pt]&X\arrow["\phi_-^+"]{dl}\arrow["\phi_+^-"']{dr}\arrow[leftarrow,dashed,"\psi_+",pos=0.63]{rr}&&[-25 pt]
\PP_{Y_+}(\cO\oplus\cN_+)\arrow["\phi_+^+",pos=0.40]{dl}\\
&X'_-\arrow["\Pi_-"']{dr}&&X'_+\arrow["\Pi_+"]{dl}&\\
&& Y'&&
\end{tikzcd}$$ 
with $H$-equivariant arrows. 
In order to see this, we note that locally the construction of $X$ can be done in toric terms, and  
then the result follows from Proposition \ref{prop:toric_bordism-construction}.
\end{remark}

\begin{corollary}\label{cor:flip=>bordism}
In the setup of Theorem \ref{flip=>bordism}, there exists an ample line bundle $L$ on $X$ such that the bandwidth of the $H$-action on the  pair $(X,L)$ is two.
\end{corollary}

\begin{proof}
With the notation of the proof of Theorem \ref{flip=>bordism}, the line bundle: 
$$L:=Y_-+Y_++(\Pi_-\circ\phi_-^+)^*\cL',$$ 
with $\cL'$ a sufficiently ample line bundle on $Y'$, satisfies the required conditions. 
In fact, by Remark \ref{rem:Btypeintnumber}, $L \cdot C_{\gen} =2$, and so $|\mu|=2$ by Lemma \ref{lem:AMvsFM}, recalling that the action on $X$ is faithful, hence $\delta(C_{\gen})=1$.
\end{proof}

The following statement tells us that the rank $1$ bordism $Y_-\hooklongrightarrow X \hooklongleftarrow Y_+$ obtained in Theorem \ref{flip=>bordism} upon an Atiyah flip $\begin{tikzcd}[cramped](Y_\pm, (Y_\pm)_{|Y_\pm})\arrow["\varphi_\pm", shift right = 1.2]{r}&Y'\supset Y_0 \end{tikzcd}$ is unique.

\begin{theorem}\label{bordism=>unique}
Let $\begin{tikzcd}[cramped](Y_\pm, (Y_\pm)_{|Y_\pm})\arrow["\varphi_\pm", shift right = 1.2]{r}&Y'\supset Y_0 \end{tikzcd}$ be an Atiyah flip. Then there exists a unique bordism $Y_-\hooklongrightarrow X \hooklongleftarrow Y_+$ of rank $1$ whose induced birational transformation $\psi:Y_-\dashrightarrow Y_+$ is the Atiyah flip.
 
\end{theorem}

\begin{proof}
Let $Y_-\hooklongrightarrow X \hooklongleftarrow Y_+$ be a bordism of rank $1$ whose induced birational map $\psi:Y_-\dashrightarrow Y_+$ is the Atiyah flip
$\begin{tikzcd}[cramped](Y_\pm, (Y_\pm)_{|Y_\pm})\arrow["\varphi_\pm", shift right = 1.2]{r}&Y'\supset Y_0 \end{tikzcd}$.
In order to prove the unicity, we will show that $X$ is the Atiyah flip of a $\P^1$-bundle $\P_{Y_-}(\cO\oplus (Y_\pm)_{|Y_\pm})$. We will start by constructing a small contraction $\eta_-:X\to X'_-$, for which we will prove the existence of an Atiyah flip. 

Let $[C_\pm]$ be the classes of invariant curves with respect to $H=\C^*$ in $X$ linking $Y_{\pm}$ with an inner fixed point component (see Definition \ref{def:Cpm}), and 
let $\cM'$ be a very ample line bundle on $Y'$; we pull it back to $Y_-$ and extend it to $X$ (see Definition \ref{def:extend_div}), setting $\cM:=\iota^-_\#\varphi_-^*\cM'$. By construction, $\cM\cdot [C_-]=0$. On the other hand, we may write $\cM=\iota^+_\#\varphi_+^*\cM'$, so we also have $\cM\cdot [C_+]=0$ and, by Lemma \ref{rank1bordism-curves}, we conclude that $\cM\in\Pic(X)^H$.
By Lemma \ref{lem:base-locus} the linear system $\iota^-_\#\varphi_-^*|\cM'|\subset |\cM|$ is base point free on $X$, hence $\cM$ defines a contraction $\eta$ of $X$. Since $\cM$ is $H$-trivial, then $\eta$ contracts the closure of every orbit; moreover the restriction of $\eta$ to $Y_\pm$ is equal to $\varphi_\pm$, so the image of $\eta$ is $Y'$, and the following diagram commutes: 
$$\begin{tikzcd} Y_-\arrow[hook]{r}\arrow["\varphi_-"']{dr}&X\arrow["\eta"]{d}&Y_+\arrow[hook']{l}\arrow["\varphi_+"]{dl}\\ &Y'&\end{tikzcd}$$ 
Since the contractions $\varphi_\pm$  are elementary  and, by the definition of bordism, $\rho_X=\rho_{Y_{\pm}}+1$, we have $\rho_X-\rho_{Y'}=2$, hence
$\Nu(X/Y')=\Nu(X)^H$. Since $Y_- \cdot \, C_+=0$, $Y_- \cdot C_-=1$ and, by Corollary \ref{cor:sinknef}, $Y_-$ is nef on $\Nu(X)^H \cap \cNE{X}$, it follows that $\R_+[C_+]$ is an extremal ray of $\cNE{X}$. In a similar way we show that also  $\R_+[C_-]$ is an extremal ray of $\cNE{X}$. 
Therefore there exist elementary contractions  $\eta_\pm:X \to X'_\pm$, factoring $\eta$, which
contract the extremal rays $\R_+[C_\pm]$,  and whose supporting divisors are  $Y_\pm+k\cM$ for $k \gg 0$. Thus we have a
commutative diagram
$$\begin{tikzcd}Y_-\arrow[hook]{r}
&X\arrow["\ \ \eta_-"']{dl}\arrow["\eta_+"]{dr}\arrow["\eta"]{d}
&Y_+\arrow[hook']{l} \\
 X'_-\arrow{r}&Y'&X'_+\arrow{l}\end{tikzcd}$$ such that
$(\eta_\pm)_{|Y_\pm}= \varphi_{\pm}:Y_{\pm}\rightarrow Y'$.  

We will prove now that $\eta_-$ is a small contraction of Atiyah type, 
and construct the corresponding flip. 
Let us denote by $\bigcup_jZ^j_-$ the exceptional locus of $\varphi_-$ and, as in Remark \ref{rem:Atiyah.assumptions}, $r_j=\codim(Z^j_{+},Y_+)-1$, $s_j=\codim(Z^j_{-},Y_-)-1$, for every $j$.
We claim first that $\Exc(\eta_-)$ is equal to $\bigcup_j\overline{X^+(Y_j)}$, and that $\overline{X^+(Y_j)}$ is a $\P^{r_j+1}$-bundle over $Y_j$, for every inner fixed point component $Y_j$ of the bordism. 

Let $F$ be an irreducible component of a fiber of $\eta_-$. Since it intersects $Y_-$ along a fiber of $\varphi_-$, which is isomorphic to $\P^{r_j}$ for some $j$, it follows that $\dim F\leq r_j+1$. 
On the other hand, by Lemma \ref{lem:smoothZ}, every fiber of $\varphi_-$ is contained in a unique variety isomorphic to $\P(N^+(Y_j)^\vee_y)$, for some $y$ contained in an inner fixed point component $Y_j$. Since the lines in $\P(N^+(Y_j)^\vee_y)$ correspond to curves in the class $[C_-]$, we conclude that the fibers of $\eta_-$ are all isomorphic to $\P(N^+(Y_j)^\vee_y)\simeq \P^{r_j+1}$. This concludes the claim.

%

We construct now the Atiyah flip of $\eta_-$ by glueing toric Atiyah flips, so that, in particular, condition (2) of Definition \ref{def:Atiyah.assumptions} is fulfilled. For every point $y$ of an inner fixed point component $Y_j$, we may take an analytic neighborhood $U$ of $y\in Y_j$, biholomorphic to $\C^q$, and a $H$-invariant analytic open set in $X$ which is $H$-biholomorphic to $(N_{Y_j/X})_{|U}\simeq\C^{\dim X}$ (see \cite[Theorem~2.5]{BB}), that we identify with the toric variety $X_{\Sigma(\Delta)}$ introduced in Section \ref{ssec:toricAtiyah}. Following Section \ref{ssect:cobordism}, the $H$-action inherited by $X_{\Sigma(\Delta)}$ defines two good $H$-quotients, $X_{\Sigma_\pm}$, of two open subsets in $ X_{\Sigma(\Delta)}$, and a birational transformation among them, that we denote by $\psi_{\circ}$. By construction, this map is the restriction of the birational map $\psi:Y_-\dashrightarrow Y_+$, which is an Atiyah flip, by assumption, so we conclude that $\psi_\circ$ is a toric Atiyah flip. We then consider the toric variety $X_{\widehat{\Sigma}}$ constructed in Proposition \ref{prop:toric_bordism-construction} which is isomorphic to the analytic closure of $X_{\Sigma(\Delta)}$ into $X$, consisting of the open set $X_{\Sigma(\Delta)}$, together with the limiting points of all the orbits of the $H$-action. 
In particular, the second part of Proposition \ref{prop:toric_bordism-construction} allows us to claim that the subset $X_{\widehat{\Sigma}}\subset X$ admits a toric Atiyah flip to a $\P^1$-bundle $X_{\widehat{\Sigma}_-}$ whose associated Atiyah type contraction is, by construction, the restriction of $\eta_-:X\to X'_-$. Glueing these toric Atiyah flips we obtain the flip  $X\dashrightarrow \PP_{Y_-}(\cO\oplus (Y_-)_{|Y_-})$ of the contraction $\eta_-$.

Finally, we note that the line bundle $Y_-$ satisfies the condition (3) of Definition \ref{def:Atiyah.assumptions}, so we may conclude that $X\dashrightarrow \PP_{Y_-}(\cO\oplus(Y_-)_{|Y_-})$ is an Atiyah flip.  
\end{proof}

\section{Bandwidth two actions and Atiyah flips}\label{sec:BW22}

We will now show how to apply the machinery we have just developed to the case of polarized pairs $(X,L)$ that support bandwidth two $H$-actions, showing that, under certain conditions,
they are determined by the sink and the source together with their normal bundles (see Proposition \ref{prop:Atiyah_adjoint}). We conclude the Section by presenting some examples  of bandwidth two actions supported on rational homogeneous varieties (Section \ref{ssec:HomoBW2}), and discuss how they are determined by their sink and source (see Corollary \ref{cor:adjointBW2}).

We will choose, for bandwidth two pairs, the linearization $\mu_L$ of $L$ whose weights at the source and the sink are, respectively, $1$ and  $-1$,
and we will denote the source and the sink, respectively, by $Y_+$ and $Y_-$.
Let us start with the following general  observation.

\begin{lemma} Let $(X,L)$ be a polarized pair supporting an  $H$-action of bandwidth two, equalized at the source $Y_+$ and the sink $Y_-$. Then the action is equalized. Moreover,  if there exists an inner fixed point component of the action $Y_0^j$,  the degree of $L$ on the $H$-invariant curves with sink or source in $Y_{0}^j$ is equal to one.
\label{lem:equalizedSS}
\end{lemma}

 \begin{proof}
The action is equalized by Corollary \ref{cor:allequalized} (iv). In particular, if $C_\pm$ is the closure  of an orbit linking an inner fixed point component $Y_0^j$ with $Y_\pm$, then the degree of $L$ on $C_\pm$ is equal to one by Lemma \ref{lem:AMvsFM} and Corollary \ref{cor:allequalized} (ii). 
\end{proof}

We will now show how to construct a bordism of rank one out of a bandwidth two pair $(X,L)$.

 \begin{lemma}\label{lem:BW2bordism}
 Let $(X,L)$ be a polarized pair supporting an  $H$-action of bandwidth two, equalized at the sink $Y_-$ and source $Y_+$. Assume moreover that $\rho_X=1$, that the sink and the source have positive dimension, and that $X$ has at least one inner fixed point component.  Let $\alpha:X_\flat \to X$ be the blowup of $X$ along $Y_-\cup Y_+$, and denote by $Y_-^\flat$ and $Y_+^\flat$ the exceptional divisors.  Then $L^\flat:=\alpha^*L-Y_-^\flat-Y_+^\flat$ is a nonzero $H$-trivial divisor and the extended $H$-action on the blowup is a bordism of rank one.
 \end{lemma}
  
 \begin{proof}
 The fact that the action on $X_\flat$ is a bordism follows by Lemma \ref{lem:setup_bir}. 
Since the closure of every $1$-dimensional orbit meets the sink or the source, by Lemma \ref{lem:AMvsFM} and Corollary \ref{cor:allequalized} (ii) the line bundle $L^\flat$ has intersection number zero with the closure of all the orbits of the action. Moreover, taking a curve $F$ contracted by $\alpha$, we get $L^\flat\cdot F>0$, therefore $L^\flat$ is nontrivial. In particular $L^\flat$ is an $H$-trivial divisor, so the rank of the bordism is at most $\rho_{X_\flat}-1-1=1$. If it were zero, then $X_\flat$ would be a $\P^1$-bundle by Lemma \ref{lem:rank0bordism}, contradicting the assumption on the existence of an inner component. 
 \end{proof}

\begin{remark} In the above lemma, the assumption on the positive dimension of the sink and the source is necessary. In fact, by Lemma
 \ref{lem:Pic_extremal}, if one of the extremal components is a point, then there exists an inner component $Y_j$ for which either $\nu^+(Y_j)$ or $\nu^-(Y_j)$ is equal to one, against the definition of bordism.
\end{remark}

 We will now show that the birational transformation induced by such a bordism is an Atiyah flip, under some positivity conditions on the conormal bundles of the sink and the source. 
 
\begin{proposition}\label{prop:Atiyah_adjoint}
Let $(X,L)$ be as in Lemma \ref{lem:BW2bordism}. Assume moreover that the vector bundles $N_{Y_{\pm}/X}^\vee \otimes L$ are semiample. 
 Let $Y_-^\flat\hooklongrightarrow X_\flat \hooklongleftarrow Y_+^\flat$
be the bordism induced on the blowup $X_\flat$ of $X$ along sink and source, and let $\psi:  Y_-^\flat\dashrightarrow Y_+^\flat$ be the induced birational transformation.
Then there exist projective varieties $Y_0 \subset Y'$ and contractions $\varphi_\pm: Y_\pm^\flat \to Y'$ such that $\psi$ is the birational transformation of an Atiyah flip
$\begin{tikzcd}[cramped]
  (Y_\pm^\flat,(Y_\pm^\flat)_{|Y_\pm^\flat})\arrow["\varphi_\pm",shift right = 1]{r}&Y'\supset Y_0
\end{tikzcd}$.
\end{proposition}

\begin{proof}
By Lemma \ref{lem:BW2bordism} the line bundle $L^\flat:=\alpha^*L-Y_-^\flat - Y_+^\flat$ and its multiples are $H$-trivial. Then,  by Lemma \ref{lem:base-locus}, $mL^\flat$ is globally generated for $m\gg0$; in fact the restriction of $L^\flat$ to $Y_-^\flat$ is  the tautological bundle of $Y_-^\flat:=\P(N^\vee_{Y_-/X}\otimes L)$, which is semiample by assumption. 
Let us denote by $L^\flat_\pm$ the restrictions $i^*_\pm L^\flat \in \Pic(Y_\pm^\flat)$, and by $\varphi_\pm:Y_\pm^\flat \to Y'_\pm$ the corresponding contractions of $Y_\pm^\flat$. By Lemma \ref{lem:Pic_extremal} the Picard number of $Y_{\pm}$ is one, hence the Picard number of $Y_\pm^\flat$ is two. Since, by Lemma \ref{lem:BW2bordism}, $L_-^\flat \cdot C_-= L_+^\flat \cdot C_+=0$ and $L^\flat$ is not trivial on $Y_\pm^\flat$, the contractions $\varphi_\pm$ are the contractions of the rays in $\NE(Y_\pm^\flat)$ generated by the classes $[C_\pm]$.  

Thus Corollary \ref{cor:sum_of_orbits} gives that $[C_{\gen}]\cdot L^\flat=0$ and, by Lemma \ref{lem:Cgenort}, $L^\flat \in i_\#^\pm(\Pic(Y_\pm^\flat))$, hence $L^\flat=i_\#^\pm L^\flat_\pm$. By Lemma \ref{lem:psistar} we have $\psi_*(L^\flat_-)=L^\flat_+$; it follows that $Y'_-=Y'_+:=Y'$ and the following diagram commutes:
$$\begin{tikzcd}
Y_-^\flat\arrow[rightarrow,dashed,"\psi"]{rr}\arrow["\varphi_-"']{dr}
&&Y_+^\flat\arrow["\varphi_+"]{dl}\\
&Y'&
\end{tikzcd}$$ 
In particular, by Lemma \ref{lem:smoothZ}, the exceptional locus of $\varphi_-$ is the disjoint union of smooth varieties $Z_-^j=\P(N^+(Y_0^j)^\vee)$, where $Y_0^j$ are the inner fixed point components of the action.

In order to conclude that $\psi$ is an Atiyah flip, we first show that it is locally a toric Atiyah flip. Since, by Lemma \ref{lem:equalizedSS},  the $H$-action is equalized, then, for every point $y$ of an inner fixed point component $Y_0^j$, we may take an analytic neighborhood $U$ of $y\in Y_0^j$, biholomorphic to $\C^q$, and an $H$-invariant analytic open set in $X_\flat$ which is $H$-biholomorphic to $(N_{Y_0^j/X})_{|U}$ (see \cite[Theorem~2.5]{BB}), 
which is then $H$-biholomorphic to a toric variety $X_{\Sigma(\Delta)}$, as in Section \ref{ssect:cobordism}. In the notation of that section, the $H$-action defines two good $H$-quotients $X_{\Sigma_\pm}$, of two open subsets in $X_{\Sigma(\Delta)}$, and a birational transformation among them that is a toric Atiyah flip. Since these two varieties are naturally embedded in $Y_\pm^\flat$, we conclude that the map $\psi$ is locally a toric Atiyah flip. 

Condition (3) of  Definition \ref{def:Atiyah.assumptions} is now easy to check: we have seen that $\varphi_{\pm}$ are elementary, and $Y_-^\flat$ (resp. $Y_+^\flat$) restricts to $\cO(1)$ on the fibers of $\varphi_-$ (resp. $\varphi_+$).   Since, by Remark \ref{rem:Btypeintnumber} $(Y_-^\flat-Y_+^\flat)\cdot C_{\gen} =0$ we have, by Lemma \ref{lem:Cgenort}, $Y_-^\flat-Y_+^\flat \in i_\#^\pm(\Pic(Y_\pm^\flat))$, hence $Y_-^\flat-Y_+^\flat=i^-_\#Y_-^\flat|_{Y_-^\flat}$. By Lemma \ref{lem:psistar} we then  have $\psi_*(Y_-^\flat)=-Y_+^\flat$.
\end{proof}

We may finally apply Theorem \ref{bordism=>unique}, to get the following:

\begin{corollary}\label{cor:uniqueBW2}
Let $(X,L)$ be a polarized pair supporting an  $H$-action of bandwidth two, equalized at the sink $Y_-$ and source $Y_+$, which are both positive dimensional. Assume moreover that $\rho_X=1$, that there exists an inner fixed point component, and that the vector bundles $N_{Y_{\pm}/X}^\vee \otimes L$ are semiample. Then $X$ is uniquely determined by $(Y_{\pm}, N_{Y_{\pm}/X})$.
\end{corollary}

\begin{proof}
Let $\alpha:X_\flat \to X$ be the blowup of $X$ along $Y_-\sqcup Y_+$; then, by Lemma \ref{lem:BW2bordism}, the induced $H$-action is a bordism $Y_-^\flat\hooklongrightarrow X_\flat \hooklongleftarrow Y_+^\flat$, which, by Proposition \ref{prop:Atiyah_adjoint}, satisfies the assumptions of Theorem \ref{bordism=>unique}. It follows that $X_\flat$ is uniquely determined by the Atiyah flip 
$\begin{tikzcd}[cramped](Y_\pm^\flat,(Y_\pm^\flat)_{|Y_\pm^\flat})\arrow["\varphi_\pm",shift right = 1]{r}&Y'\supset Y_0 \end{tikzcd}$; since $\rho_{Y_\pm^\flat}=2$, the flip is uniquely determined by  $(Y_{\pm}, N_{Y_{\pm}/X})$.
\end{proof}

\subsection{Examples: bandwidth two varieties and short gradings}\label{ssec:HomoBW2}

In this section we present some examples of polarized varieties supporting an action of $H=\C^*$ of bandwidth two, that is equalized at the sink and source. Our examples will be adjoint varieties of semisimple algebraic groups, and we will show, using the machinery developed previously, that those of Picard number one  are determined by the sink and the source of the action, together with their normal bundles. We refer the interested reader to \cite[Section~3]{CARRELL} for general results on torus actions on rational homogeneous spaces.
 
Given a simple algebraic group $G$, its associated adjoint variety is, by definition, the minimal orbit of the (Grothendieck) projectivization of the dual of the adjoint representation $\P(\fg^\vee)$. Let us denote it by $X_{{\aj}} =G/P$, where $P$ is the parabolic subgroup associated to the choice of a set of nodes $I\subset D$ of the Dynkin diagram $\cD$ of $G$; more precisely, $I=\{1,{n}\}$ in the case in which $G$ is of type $\DA_{n}$, and  $I$ consists of a unique root in the rest of the cases:
\setlength{\tabcolsep}{3pt}
\begin{table}[!h!]
\caption{Adjoint varieties of simple Lie algebras.}
\begin{tabular}{|r|c|c|c|c|c|c|c|c|c|}
\hline
type& $\DA_n$ & $\DB_n$ & $\DC_n$ & $\DD_n$ & $\DE_6$ & $\DE_7$ & $\DE_8$ & $\DF_4$ & $\DG_2$\\ [2pt]\hline
adjoint& $\DA_n(1,n)$ & $\DB_n(2)$ & $\DC_n(1)$ & $\DD_n(2)$ & $\DE_6(2)$ & $\DE_7(1)$ & $\DE_8(8)$ & $\DF_4(1)$ & $\DG_2(2)$\\[2pt]\hline
\end{tabular}
\label{tab:adjoint}
\end{table}

We will denote by $L_{\aj}$ the restriction to $X_{\aj}\subset \P(\fg^\vee)$ of the tautological bundle $\cO_{\P(\fg^\vee)}(1)$. This is the ample generator of the Picard group of $X_{\aj}$ in all the cases of Picard number one, with the exception of $\DC_{m}(1)\simeq\P^{2m-1}$, for which $L_{\aj}\simeq \cO_{\P^{2m-1}}(2)$, and it is equal to $\cO_{\P(T_{\P^m})}(1)$ in the case of $\DA_m(1,m)\simeq \P(T_{\P^m})$.

In this section we will show how to define $H$-actions  with prefixed bandwidth on adjoint varieties, by studying at the polytope of roots of the algebra, and describe completely the $H$-actions of bandwidth two. In particular, we will give geometric descriptions of these actions in the cases of classical type. We will not discuss the well known case of $\P^n=\P(V)$, on which bandwidth two actions are determined by the choice of a decomposition $V=V_{-}\oplus V_0 \oplus V_{+}$, and the choice of three characters $m_-,m_0,m_+$ whose sum is zero. In particular, we may assume that the identity component of the automorphism group of the adjoint variety $X_{\aj}$  is equal to the adjoint group $G_{\ad}$, associated to the Lie algebra $\fg$ (cf. \cite[Theorem~2, p. 75]{Akh}).

Given a nontrivial action of $H$ on $X_{\aj}$, we may always substitute $H$ by the image of the homomorphism $H\to\Aut(X_{\aj})$, and assume that the action of $H$ is faithful, given by a monomorphism $H\to G_{\ad}$, and set $G:=G_{\ad}$.

%
%

First of all, we complete $H$ to a maximal torus $\To\subset  G$, consider the corresponding Cartan decomposition of $\fg$ with respect to $\To$, and choose a base of positive simple roots of $\fg$. Following \cite[Lemma~3.12]{CARRELL}, the set of points of $X_{\aj}=G/P$ fixed by $\To$ is equal to $\{wP,\,\,w\in W\}$, where $W$ denotes the Weyl group of $G$. On the other hand we may locate these points among the fixed points of the projectivization of the adjoint representation (which can be read out of the adjoint representation of $\fg$). Since $G/P$ is the orbit in $\P(\fg^\vee)$ of the class of a highest weight vector of the adjoint representation (cf. \cite[Claim~23.52]{FH}), it follows that the fixed points of the action of $\To$ on $X_{\aj}$ are in one to one correspondence with the long roots of $\fg$. Moreover the weight of the action on the restriction of $L_{\aj}$ to the fixed point corresponding to a long root $\beta$ is $-\beta$. In particular the weights of the action of $H$ on $X_{\aj}$ can be then computed by looking at the image of these roots by the induced map on the lattices of characters:
$$
\Mo(\To)\to\Mo(H).
$$
To give such a map is equivalent to give a $\Z$-grading on the Lie algebra $\fg$. See \cite[Section~2.3]{Tev05} for an account on these gradings. They are determined by the choice of a simple root $\alpha_j$, i.e., of an index $j\in D$; moreover, since we want the $H$-action to be nontrivial on the adjoint variety $X_{\aj}$, $j$ must be chosen outside of the set $I$ defining the parabolic subgroup $P$. Then the corresponding $\Z$-grading is given by:
$$
\fg=\fh\oplus \bigoplus_{m\in \Z}\fg_m,
$$
where $\fg_m$ is the direct sum of the eigenspaces $\fg_\beta$, where $\beta$ is a root containing $\alpha_j$ as a summand with multiplicity $m$, and $\fh$ is the Lie algebra of $\To$. It is then known that $\fg':=\fh\oplus \fg_0$ is a Lie subalgebra of $\fg$, and every $\fg_m$ is a $\fg'$-module.

The associated $H$-action on $\fg$ is defined by associating the weight $-m$ to the subspace $\fg_m$, $m\in \Z\setminus\{0\}$, and $0$ to the subspace $\fg'$.

\begin{remark}\label{rem:SSHomoBW2}
By construction, the induced action on $\P(\fg^\vee)$ leaves $X_{\aj}$ invariant, and the restriction of this action to $X_{\aj}$ has fixed point components: 
$$
Y_{m}:=X_{\aj}\cap \P(\fg_m^\vee),\,\,m\neq 0, \quad Y_0=X_{\aj}\cap \P(\fg'^\vee).
$$
Moreover, the subgroup $G'\subset G$ defined by having Lie algebra $\fg'\subset\fg$, acts via the adjoint representation of $G$ transitively on every fixed component $Y_{m}$, $m\neq 0$, and $Y_0$.  In other words, these varieties are rational homogeneous quotients of $G'$, obtained as closed $G'$-orbits in the representations $\P(\fg_{m}^\vee)$, $m\neq 0$, and $\P(\fg')$, respectively. In particular $Y_{-m}$ and $Y_m$ will be isomorphic rational homogeneous manifolds for every $m\neq 0$. 
Furthermore, let us denote by $M$ the maximal value of $m$, so that the bandwidth of the $H$-action on $(X_{\aj},L_{\aj})$ is $2M$. Then $Y_{\pm M}$ will be the extremal fixed point components of the action (so we will write $Y_{\pm}=Y_{\pm M}$), and they can be described as the rational homogeneous varieties associated to the marking of the Dynkin diagram of $\fg'$ in the nodes of $I$.  On the other hand, whenever $\fg'$ contains a long root of $\fg$, the subvariety $Y_0$ will be the adjoint variety of the Lie algebra $\fg'$. Otherwise, it is empty (this will be the case in which $\fg$ is of type $\DC_m$). Note that the Dynkin diagram of $\fg'$, which is obtained by deleting the node $j$ on the Dynkin diagram of $\fg$, is not necessarily connected; if it is not, then $Y_0$ will consist of the union of more than one irreducible component (this will be the case in which $\fg$ is of type $\DA_m$). \end{remark}

We note now the following property of the fixed point components of the  $H$-actions on adjoint varieties, that we will use later on.

\begin{lemma}\label{lem:conormalnef}
Let $Y_m \subset X_{\aj} \subset \P(\fg^\vee)$ be a fixed point component as above, and denote by $L_{\aj}$ the restriction of the hyperplane line bundle on $\P(\fg^\vee)$. Then $N^\vee_{Y_m/X_{\aj}}\otimes L_{\aj}$ is globally generated.
\end{lemma}

\begin{proof}
It is enough to note that the inclusions:
$$
\xymatrix{Y_m\ar@{^{(}->}+<2ex,0ex>;[r]\ar@{_{(}->}-<0ex,2ex>;[d]&X_{\aj}\ar@{_{(}->}-<0ex,2ex>;[d]\\\P(\fg_m^\vee)\ar@{^{(}->}+<4ex,0ex>;[r]&\P(\fg^\vee)}
$$
provide an inclusion: $N_{Y_m/X_{\aj}}\hookrightarrow (N_{\P(\fg_m^\vee)/\P(\fg^\vee)})_{|Y_m}\simeq\cO(1)^{\oplus (\dim\fg-\dim\fg_m)}$. 
\end{proof}

We will now focus on the case in which $M=1$, so that the bandwidth of the $H$-action on the polarized pair $(X_{\aj},L_{\aj})$ will be equal to two. 
Let us now describe concretely the examples of bandwidth two actions on adjoint varieties:

\begin{proposition}\label{prop:HomoBW2}
Let $X_{\aj}$ be the adjoint variety associated to a simple Lie algebra $\fg$, and assume that $X_{\aj}$ is not a projective space. Then the faithful $H$-actions of bandwidth two on $X_{\aj}$ are in $1$--$1$ correspondence with the short $\Z$-gradings on $\fg$, that is, the gradings whose only weights are $-1,0,1$. \end{proposition}

The complete list of the adjoint varieties admitting a faithful $H$-action of bandwidth two, that can be read out of the list of short gradings in simple Lie algebras (cf. \cite[p.~42]{Tev05}), is provided in Tables \ref{tab:exceptbw2}, and  \ref{tab:classicbw2}.

\begin{table}[!h!]
\caption{Bandwidth two $H$-actions associated to short gradings for the exceptional Lie algebras.}
\begin{tabular}{|l|c|c|c|c|}
\hline
type&$X_{\aj}$&short gradings&$Y_\pm$&$Y_{0}$\\
\hline\hline
$\DG_2$&$\DG_2(2)$&none&&\\\hline
$\DF_4$&$\DF_4(1)$&none&&\\\hline
$\DE_6$&$\DE_6(2)$&$1,6$&$\DD_5(5)$&$\DD_5(2)$\\\hline
$\DE_7$&$\DE_7(1)$&$7$&$\DE_6(1)$&$\DE_6(2)$\\\hline
$\DE_8$&$\DE_8(8)$&none&&\\\hline
\end{tabular}
\label{tab:exceptbw2}
\end{table}

\begin{table}[!h!]
\caption{Bandwidth two $H$-actions associated to short gradings for the Lie algebras of classical type.}
\begin{tabular}{|l|c|c|c|c|}
\hline
type&$X_{\aj}$&short gradings&$Y_\pm$&$Y_{0}$\\
\hline\hline
	\multirow{2}{*}{$\DA_{m}$}&\multirow{2}{*}{$\DA_{m}(1,m)$}&\multirow{2}{*}{$r\in\{2,m-1\}$}&$\DA_{r-1}(1)\,\,\,\times$&$\DA_{r-1}(1,r-1)\,\,\,\sqcup$\\
	&&&$\DA_{m-r}(m-r)$&$\DA_{m-r}(1,m-r)$\\\hline
$\DB_{m}$&$\DB_{m}(2)$&$1$&$\DB_{m-1}(1)$&$\DB_{m-1}(2)$\\\hline
$\DC_{m}$&$\DC_{m}(1)$&$m$&$\DA_{m-1}(1)$&$\emptyset$\\\hline
$\DD_{m}$&$\DD_{m}(2)$&$1$&$\DD_{m-1}(1)$&$\DD_{m-1}(2)$\\\hline
$\DD_m$&$\DD_m(2)$&$m-1,m$&$\DA_{m-1}(2)$&$\DA_{m-1}(1,m-1)$\\\hline
\end{tabular}
\label{tab:classicbw2}
\end{table}

\begin{remark}\label{rem:downgrading}
In the language of \cite[Section~2.5]{BWW}, the choice of a short grading corresponds to a particular case of {\em downgrading}, namely to a projection of the root polytope of the group $G$ onto a line which sends all the vertices to three points, associated to sink, source and central components.
\end{remark}

A remarkable feature of these bandwidth two actions is that they are completely determined by their sink and source. 
\begin{corollary}\label{cor:adjointBW2}
Let $(X,L)$ be a polarized pair supporting an action of $H=\C^*$ of bandwidth two, equalized at the sink $Y_-$ and the source $Y_+$. Assume moreover that $\rho_X=1$, that $X$ has at least one inner fixed point component,  
and that $Y_\pm$ are isomorphic to one of the following rational homogeneous varieties:
$$
\DB_{m-1}(1),\quad \DD_{m-1}(1),\quad \DA_{m-1}(2),\quad \DD_5(5), \quad \DE_6(1).
$$ 
Then $X$ is, respectively, isomorphic to the  adjoint variety $X_{\aj}$: 
 $$\DB_{m}(2),\quad \DD_{m}(2),\quad \DD_{m}(2),\quad \DE_6(2), \quad \DE_7(1),$$ provided that  $N_{Y_\pm/X}$ is isomorphic to $N_{Y_{\pm} /X_{\aj}}$. Moreover the isomorphism is an $H$-isomorphism of the pairs $(X,L)$ and $(X_{\aj},L_{\aj})$.
\end{corollary}

\begin{proof}
Since $\rho_X=1$ and, by Lemma \ref{lem:equalizedSS}, $L$ has degree $1$ on $H$-invariant curves with source or sink on an inner component, then $L$ is the ample generator of $\Pic(X)$. Moreover, its restriction to $Y_\pm$ generates $\Pic(Y_\pm)$, by Lemma \ref{lem:Pic_extremal}.  In particular $L_{|Y_{\pm}}\simeq L_{\aj|Y_{\pm}}$ and using Lemma \ref{lem:conormalnef} we know that $N^\vee_{Y_\pm/X}\otimes L$ is globally generated, hence applying Corollary \ref{cor:uniqueBW2} the variety $X$ is uniquely determined by $(Y_\pm,N_{Y_\pm/X})$.
\end{proof}

In the cases in which $G$ is of classical type, the examples of short gradings and of the corresponding fixed point components admit some projective descriptions, that we include below.

\begin{example}\label{ex:segre}
Let $V$ be a complex vector space of dimension $n+2$, and let $\P(V)$ be its Grothendieck projectivization. We will consider the manifold $X_{\ad}:=\P(T_{\P(V)})$, that we will identify with the closed subset:
$$
\P(T_{\P(V)})=\{(P,H)\in\P(V)\times\P(V^\vee)|\,\,\mbox{and }P\in H\}\subset\P(V)\times\P(V^\vee).
$$
In order to define a bandwidth two $H$ action on $X_{\ad}$ 
we consider a decomposition:
$$
V=V_{-}\oplus V_{+}, \quad \dim V_{-}\in\{1,\dots,n+1\},
$$
and an $H$-action on $V$ defined as:
$$
t(v_-+v_+):=v_-+tv_+, \quad\mbox{where }t\in H,\,\,v_-\in V_{-}, \,\,v_+\in V_+.
$$
The corresponding actions on $\P(V)$ and $\P(V^\vee)$ have two fixed point components, $\P(V_{-}),\P(V_{+})\subset \P(V)$, $\P(V_{-}^\vee),\P(V_{+}^\vee)\subset \P(V^\vee)$. Now we consider the induced action on $X_{\ad}$, that  has precisely four fixed point components:
$$
\begin{array}{l}\vspace{0.2cm}
Y_-:=\P(V_{-})\times \P(V_{+}^\vee),\\\vspace{0.2cm} Y_{0,-}:=\P(T_{\P(V_{-})}),\,\,\, Y_{0,+}:=\P(T_{\P(V_{+})}),\\\vspace{0.2cm} Y_+:=\P(V_{+})\times \P(V_{-}^\vee).
\end{array}
$$
It is a straightforward computation to check that the $H$-action extends to the tautological line bundle $\cO(1)$, and
one may compute the weights of the action on the restriction of $\cO(1)$ to $Y_-,Y_{0,-},Y_{0,+},Y_+$, obtaining the values $-1,0,0,1$, respectively. For instance: the restriction of $\cO(1)$ to $Y_-$ is equal to the tensor product of the tautological line bundles $\cO_{\P(V_{-})}(1)\otimes\cO_{\P(V_{+}^\vee)}(1)$; since the weight of the $H$-action on $\cO_{\P(V_{-})}(1)$ is equal to zero, and the weight of the $H$-action on $\cO_{\P(V_{+}^\vee)}(1)$ is equal to $-1$, it follows that the weight of the $H$-action on $\cO_{Y_-}(1)$ is equal to $-1$. \qed
\end{example}

\begin{example}\label{ex:quadric}
We consider now the cases of type $\DB$ and $\DD$, whose corresponding adjoint varieties $X_{\ad}$ are quadric Grassmannians, parametrizing lines contained in smooth quadrics.

We start with of an $(n+2)$-dimensional smooth quadric $\Q^{n+2}\subset \P^{n+3}=\P(V)$, and define an action of $H$ on $\Q^{n+2}$ as follows: choose two points $P_{\pm}\in \Q^{n+2}$ satisfying that the line joining them is not contained in $\Q^{n+2}$; denote by $T_{\pm}\subset\P^{n+3}$ the projective tangent spaces of $\Q^{n+2}$ at $P_{\pm}$. Their common intersection with $\Q^{n+2}$ is a smooth quadric $\Q^{n}$ of dimension $n$. Choose a set of homogeneous coordinates $(x_0:x_1:x_2:\dots:x_{n+3})$ satisfying that $P_{-}=(1:0:0:\dots:0)$, $P_{+}=(0:1:0:\dots:0)$, and $T_{+}$ $T_-$ are given by the equations $x_1=0$, and $x_0=0$, respectively. We then consider the action on $\P^{n+3}$ given by:
$$
t(x_0:x_1:\dots:x_{n+3}):=(t^{-1}x_0:tx_1:x_2:\dots:x_{n+3}), \quad t\in H,
$$
which leaves invariant the quadric $\Q^{n+2}$, by construction. It has  bandwidth two, and fixed point sets $\{P_-\}$, $\Q^n$, and $\{P_+\}$. Its $1$-dimensional orbits are open sets either in the lines joining $P_\pm$ with $\Q^n$, or in smooth conics passing by $P_-$ and $P_+$.

Now we consider the induced action of $H$ on the adjoint variety $X_{\ad}$, which is the $(2n+1)$-dimensional quadric Grassmannian of lines in $\Q^{n+2}$. This action is of course the restriction of the induced action on the projectivization of $\bigwedge^2V$, which has three fixed subspaces, corresponding to weights $-1$, $0$, $1$. The intersection of these spaces with $X_{\ad}$ are the fixed point components of the $H$-action on $X_{\ad}$: 
$$
Y_\pm:=\{\mbox{lines in $\Q^{n+2}$ passing by $P_{\pm}$}\},\quad Y_{0}= \{\mbox{lines in $\Q^{n}$}\}.
$$
Note that, identifying every line of $Y_\pm$ with its intersection with $\Q^n$, we get isomorphisms $Y_\pm\simeq\Q^n$, while the central component $Y_0$ is the quadric Grasmmannian of lines in the quadric $\Q^n$. \qed
\end{example}

\begin{example}\label{ex:quadriceven} Quadric Grassmannians of lines on a quadric of even dimension admit yet another $H$-action of bandwidth two, that we will describe here.

Let $n+2=2m$ be an even positive integer, and $\Q^{n+2}\subset\P^{n+3}=\P(V)$ be a smooth $(n+2)$-dimensional quadric. We take two disjoint projective spaces of maximal  dimension $\P^m_{\pm}\subset\P^{m+2}$, and homogeneous coordinates $(x_0:\dots:x_{n+3})$ in $\P^{n+3}$ such that $\Q^{n+2}$, $\P^m_-$ and $\P^m_+$ are given, respectively, by the equations:
$$
\left(\sum_{i=0}^mx_ix_{m+i+1}=0\right),\quad (x_{m+1}=\dots=x_{n+3}=0),\quad  (x_{0}=\dots=x_{m}=0).
$$
We consider the $H$-action on $\P^{n+3}$ given by:
$$
t(x_{0}:\dots:x_{m}:x_{m+1}:\dots:x_{n+3})=(x_{0}:\dots:x_{m}:tx_{m+1}:\dots:tx_{n+3}),$$
which clearly leaves $\Q^{n+2}$ invariant. Its fixed point components are obviously $\P^{m}_{\pm}$. As in the previous example, we consider the induced action on  the Grassmannian of lines in $\Q^{n+2}$, $X_{\ad}\subset\P(\bigwedge^2V)$. By twisting the action with a character, we may assume that the weights are $-1,0,1$, corresponding to fixed point components:
$$
Y_\pm=\{\mbox{lines in $\P^m_{\pm}$}\},\quad Y_{0}=\{\mbox{lines in $\Q^{n+2}$ meeting both $\P^m_{\pm}$}\}.
$$
In order to describe $Y_0$, we note first that it is naturally embedded in $\P^{m}_-\times\P^m_+$. Moreover, given a point $P_-\in \P^m_-$, the intersection of the tangent hyperplane to $\Q^{n+2}$ at $P_-$ meets $\P^m_+$ in a hyperplane $\rho(P_-)\in \P^{m\vee}_+$; this correspondence defines an isomorphism $\rho:\P^m_-\to\P^{m\vee}_+$, so that
$$
Q\in\rho(P_-) \iff P_-+Q\subset\Q^{n+2}.
$$
By means of this correspondence one may then easily show that $Y_0$ is isomorphic to $\P(T_{\P^m_-})$, and to $\P(T_{\P^m_+})$. \qed
\end{example}



\section{Bandwidth three actions and Cremona transformations}\label{sec:BW3}
In this section we consider the case of a smooth variety $X$ together with an ample line bundle $L$, and an equalized action of $H=\C^*$ on the pair $(X,L)$, such that the bandwidth of the action is equal to three. As usual, we will denote the sink and the source of the action by $Y_-$, $Y_+$, and we will focus on the case in which $Y_-,Y_+$ are isolated points.  
In order to understand its importance, let us recall the following result from \cite{RW}: 
\begin{theorem}\label{thm:bw3}\cite[Theorem~3.6]{RW}
Let $(X,L)$ be a polarized pair with an equalized $\C^*$-action of bandwidth three, such that its sink and source are isolated points, and assume   
 that $\dim X=n\geq 3$. Then one of the following holds:
\begin{itemize}
\item [(1)] $X=\P(\cV)$ is a projective bundle over $\P^1$, with $\cV=\cO(1)^{n-1}\oplus\cO(3)$, or $\cO(1)^{n-2}\oplus\cO(2)^2$, and $L$ is the corresponding tautological bundle $\cO_{\P(\cV)}(1)$.
  \item [(2)] $X=\P^1\times\Q^{n-1}$, and $L=\cO(1,1)$.
  \item [(3)] $n\geq 6$ is divisible by 3, $X$ is Fano of Picard number one, and $-K_X=\frac{2n}{3}L$. The inner fixed points components are two
    smooth subvarieties of dimension $\frac{2n}{3}-2$.
  \end{itemize}
\end{theorem}

This statement was there used to classify contact Fano manifolds of Picard number one and dimensions $11$ and $13$ satisfying certain assumptions (cf. \cite[Theorem~5.3]{RW}). In order to extend this way of approaching LeBrun--Salamon conjecture to higher dimensions, one needs to classify the varieties appearing in case (3) of the above Theorem. We will fulfill here this task by using the birational machinery developed in the preceding sections. The applications of this result to LeBrun--Salamon conjecture will be the goal of a forthcoming paper (\cite{WORS2}).

Before getting into details we briefly present here an example (see \cite[Example~3.9]{RW}) of an action of type (3) in the list of Theorem \ref{thm:bw3}.

\begin{example}\label{ex:bw3_C33}
Let $V$ be a $6$-dimensional complex vector space endowed with a nondegenerate skew-symmetric form $\omega$, $X:=\DC_3(3)$ the Lagrangian Grassmannian of $3$-dimensional subspaces in $V$ isotropic with respect to $\omega$, and $G=\SP(V)$ the group of automorphisms of $V$ that preserve $\omega$. It is known that $X$ is homogeneous with respect to $G$, and that it can be seen as the unique closed orbit of the projectivization $\P(W)$ of a representation $W$ of $G$ of dimension $14$. Let $H_{\max}$ be a Cartan subgroup of $G$. The set $\P(W)^{H_{\max}}$ consists of $14$ points, corresponding to the $14$ weights of the representation $W$. Among them there are exactly $8$ points that belong to $X^{H_{\max}}$, that correspond to the weights of maximal length of $W$, and one may check that these weights are the vertices of a cube in the lattice of weights $\Mo(H)$. One may then consider the projection of these weights to one of the diagonals of the cube, corresponding to a map of lattices $\Mo(H)\to \Z$ (sending the vertices of the cube to three consecutive elements in the lattice $\Z$) and to the choice of a $1$-dimensional torus $H\subset H_{\max}$. Then it follows that $X$ has four $H$-fixed point components and that the extremal ones are isolated points.
\begin{figure}[H]
\includegraphics[width=6cm]{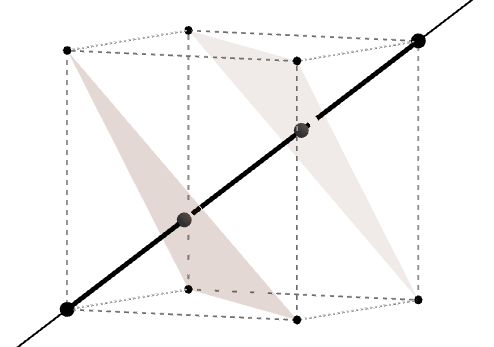}
\caption{The weights of the action of a maximal torus on $\DC_3(3)$ and its reduction to an $H$-action of bandwidth $3$.}
\label{fig:cube}
\end{figure}
Furthermore, one may show that there are exactly two inner fixed point components $Y_1$, $Y_2$ of the action, isomorphic to two Veronese surfaces $v_2(\P^2)$, lying in two $5$-dimensional projective subspaces of $\P(W)$ fixed by $H$. Each line in $X$ passing through the sink $Y_-$ (resp. the source $Y_+$) of the action meets $Y_1$ (resp. $Y_2$) in exactly one point, so that $Y_1$ (resp. $Y_2$) can be interpreted as a parameter space for the family of lines in $X$ passing through $Y_-$ (resp. $Y_+$). 
\end{example}

One may in fact check that we have a similar situation in other three rational homogeneous varieties, namely $\DA_5(3)$, $\DD_6(6)$ and $\DE_7(7)$, which are those for which the family of lines passing through a point is a {\it Severi variety} (cf. \cite[III~Definition~2.3]{Z}). The goal of this Section is to prove that these varieties are completely determined by admitting a $\C^*$-action of bandwidth $3$ with the properties required above, and to relate these properties to birational transformations of $X$.

\subsection{Cremona and Severi}\label{ssec:Cremona_Severi}

Let us assume, from now on, that $(X,L)$ is a polarized pair as in Theorem  \ref{thm:bw3}, case $(3)$. By adding a character to the linearization of the action, we may assume that $\mu_L(Y_-)=0$. Then $\mu_L(Y_+)=3$, and the action has weights $1$ and $2$ on the two inner fixed point components, that we denote by $Y_1$ and $Y_2$.
We will set:
$$
n=:3m,\,\,\,m\geq 1,
$$ 
so that $-K_X=2mL$, and $\dim Y_1=\dim Y_2=2m-2$. 

We denote an orbit and its closure in $X$ by using the same letter. The action is equalized hence, by Corollary \ref{cor:allequalized} (i),  the closures of all orbits are smooth rational curves.
Moreover, by Lemma \ref{lem:AMvsFM},  the intersection number of the line bundle $L$ with the closure of an orbit (i.e., the degree of the restriction of $L$ to the closure of an orbit) is obtained as the difference between the value of the linearization at the source and at the sink of the orbit. In the following graph the intersection with $L$ is
presented as the subscript at each of the letters.
\begin{equation}\label{eq:gr1}
\xygraph{
!{<0cm,0cm>;<3cm,0cm>:<0cm,1cm>::}
!{(0,0) }*+{\bullet_{Y_-}}="0"
!{(1,0) }*+{\bullet_{Y_1}}="1"
!{(2,0) }*+{\bullet_{Y_2}}="2"
!{(3,0)}*+{\bullet_{Y_+}}="3"
"0"-@/^/@[green]"1" _{A_1}
"2"-@/^/@[green]"3" _{B_1}
"0"-@/_1cm/@[blue]"2" _{A_2}
"1"-@/_1cm/@[blue]"3" _{B_2}
"1"-@/^/@[red]"2" _{C_1}
"0"-@/^1cm/@[brown]"3" ^{C_3}
} 
\end{equation}
The closure of a general orbit $C_3$, which is of degree $3$ with
respect to $L$, can degenerate, as a curve in $X$, to the $1$-cycles $A_1+B_2$, $A_2+B_1$ and
$A_1+C_1+B_1$.

Let us denote by $\varphi:X_\flat\to X$ the blowup of $X$ along $Y_-,Y_+$, by $Y_-^\flat$, $Y_+^\flat$ the corresponding exceptional divisors, which are isomorphic to $\P^{n-1}$, and set 
$$Z_-^j:=\ol{X_{\flat}^+(Y_{j})}\cap Y_-^\flat\subset Y_-^\flat  \quad Z_+^j:=\ol{X_{\flat}^-(Y_{j})}\cap Y_+^\flat\subset Y_+^\flat,$$ 
(see Figure \ref{fig:orbits}). 
As in Lemma \ref{lem:setup_bir}, the $\CC^*$-action on $X$ lifts up to a B-type $H$-action on $X_\flat$, whose sink and source are $Y_-^\flat$, $Y_+^\flat$ and we have a birational map 
$$\begin{tikzcd}\psi:Y_-^\flat\arrow[rightarrow,dashed]{r}&Y_+^\flat\end{tikzcd}$$
which, by Lemma \ref{lem:birational_map_b}, is defined on $Y_-^\flat \setminus (Z_-^1 \cup Z_-^2)$. 
The main result of this Section describes the map $\psi$. We first recall the following:

\begin{definition}\label{def:Cremona}
A {\em Cremona transformation} is a birational map $\psi:\P^n\dashrightarrow\P^n$. If $V\subset \HH^0(\P^n,\cO(d))$, $d\geq 1$, denotes the linear system (without fixed components) defining $\psi$, the {\em base scheme} of $\psi$ is the subscheme of $\P^n$ defined by the ideal sheaf in $\cO_{\P^n}$ generated by $V$. If the base scheme of $\psi$ is smooth and connected, we say that $\psi$ is a {\em special Cremona transformation}  (cf. \cite{ESB}). 
\end{definition}

\begin{theorem}\label{thm:bandwidth3->Cremona}
Let $(X,L)$ be a polarized pair which satisfies Theorem \ref{thm:bw3} $(3)$. 
Then the birational maps $\psi$ and $\psi^{-1}$ are special Cremona
transformations defined by linear systems of quadrics. In particular, $m=2,3,5,9$ and the centers of the transformations are projectively equivalent to a Severi variety in $\PP^{3m-1}$.
\end{theorem}

Let us start by showing that the birational map $\psi$ can be extended to $Y_-^\flat \setminus Z_-^1$.

\begin{lemma}\label{lem:BW3_exc_loc}
We have $\nu^+(Y_1)=\nu^{-}(Y_2)=1$, and there exist isomorphisms
$$Z_-^1\simeq Y_1,\,\,\, \ol{X_\flat^+(Y_1)}\simeq\P(\cO_{Y_1}\oplus \cO_{Y_1}(L)),\,\,\,Z_+^2\simeq Y_2,\,\,\,  \ol{X_\flat^-(Y_2)}\simeq\P(\cO_{Y_2}\oplus \cO_{Y_2}(L)).$$
Moreover, the exceptional locus of the map $\psi:Y_-^\flat\dashrightarrow Y_+^\flat$ is equal to $Z_-^1$ and the exceptional locus of $\psi^{-1}$ is equal to $Z_+^2$.
\end{lemma}

\begin{proof}
The first part of the statement may be obtained from \cite[Lemma 2.9 (3)]{RW} or, alternatively in our setting, as follows. Note first that $\nu^+(Y_1)=1$; otherwise, the family of closures of orbits $A_1$ passing through $Y_-$ and a given point of $Y_1$ would be positive dimensional, and we would get a contradiction with the fact that these curves have degree one with respect to $L$ by means of the Bend and Break lemma (see \cite[Proposition~3.2]{De}).

Since $\nu^+(Y_1)=1$, by Theorem \ref{thm:BB_decomposition} (2)  $X^+_\flat(Y_1)=X^+(Y_1)$ is a line bundle over $Y_1$; similarly, $X_\flat^-(Z_-^1)$ is a line bundle over $Z_-^1$. Their intersection is then an $H$-principal bundle over $Y_1$ and $Z_-^1$, and so, quotienting by the action of $H$, the isomorphism $Y_1\simeq Z_-^1$ follows. Moreover
$$\ol{X^+_\flat(Y_1)}=X^+_\flat(Y_1)\cup X_\flat^-(Z_-^1)$$ is a $\P^1$-bundle, that can be seen as the Grothendieck projectivization of the pushforward to $Y_1$ of the unisecant divisor $L$. The fact that $Z_-^1$ and $Y_1$ are two disjoint sections of this $\P^1$-bundle provides the isomorphism $\ol{X_\flat^+(Y_1)}\simeq\P(\cO_{Y_1}\oplus \cO_{Y_1}(L))$ by standard arguments. The other two isomorphisms regarding $Y_2$ are analogous.

For the second part, we need to prove that the birational map $\psi$  can be extended on $(Z_-^2\setminus Z_-^1)\subset Y_-^\flat$ . Let $x_0\in Z_-^2\setminus Z_-^1$ be a point. There exists a unique invariant $1$-cycle $\Gamma$ of $L$-degree $3$ passing through $x_0$, constructed as follows: let $O$ be  unique $1$-dimensional orbit converging to $x_0$ at $\infty$, and let $x_2$ be its limit at zero; since $x_0\in Z_-^2\setminus Z_-^1$, it follows that $x_2\in Y_2$; then, since $\nu^-(Y_2)=1$, there exists a unique orbit $O'$ converging to $x_2$ at $\infty$, and we have $\Gamma=\overline{O}+\overline{O'}$.
\begin{figure}[H]
\includegraphics[width=7cm]{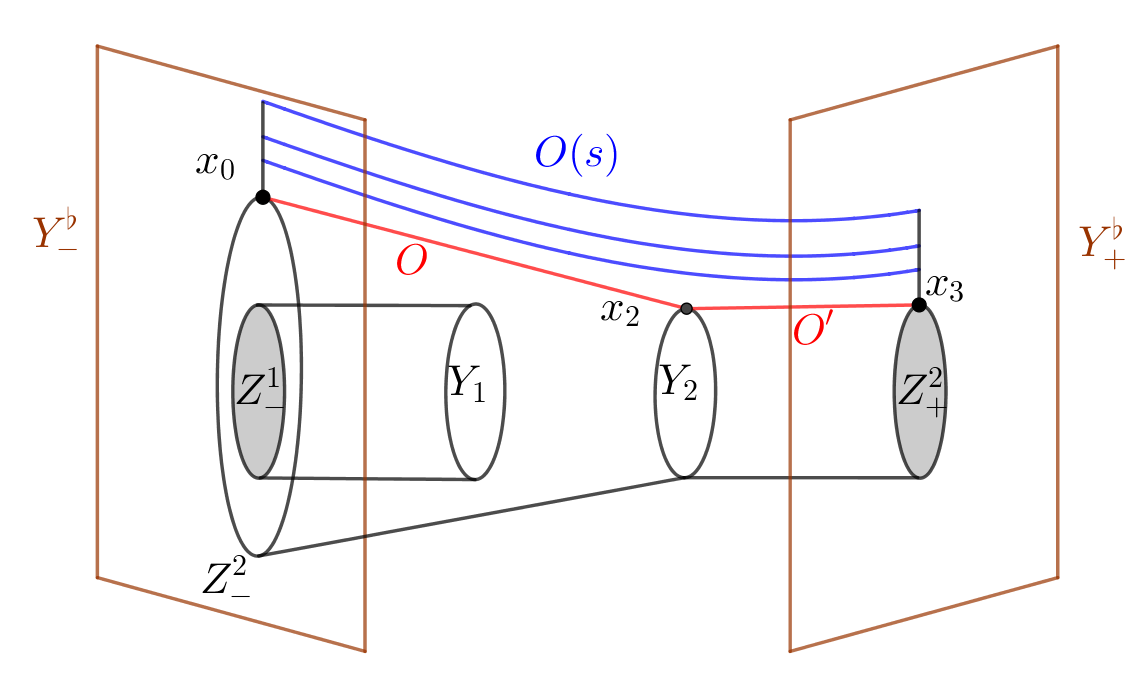}
\caption{Proof of Lemma \ref{lem:BW3_exc_loc}.}
\label{fig:orbits}
\end{figure}

Let us denote by $x_3$ the limit of $O'$ at zero. Taking any holomorphic curve $\gamma(s)$ converging to $x_0$ when $s$ goes to $0$, with $\gamma(s)\in Y_-^\flat\setminus Z_-^2$ for $s\neq 0$, the images $\psi(\gamma(s))$ are obtained by considering, for every $s$, the unique orbit $O(s)$ having limit $\gamma(s)$ when $t$ goes to $\infty$; the closures $\overline{O(s)}$ must then converge, when $s$ goes to zero, to the unique $H$-invariant $1$-cycle passing through $x_0$, which is $\Gamma$, and it follows that
$$\lim_{s\to 0}\psi(\gamma(s))=\Gamma\cap Y_+^\flat=\overline{O'}\cap Y_+^\flat=x_3.$$ 
Since this limit does not depend on the choice of the curve $\gamma$, it follows that the map $\psi$ extends to $x_0$. This finishes the proof.
\end{proof}

Set $\cU_1:=\ol{X^+_\flat(Y_1)}$, $\cU_2:=\ol{X^-_\flat(Y_2)}$, and recall that, by Lemma \ref{lem:BW3_exc_loc}, we have isomorphisms $\cU_i\simeq\P(\cO_{Y_i} \oplus \cO_{Y_i}(L))$, for $i=1,2$. 

\begin{proposition}\label{prop:birBW3}
The variety $X_\flat$ is birationally equivalent to a $\P^1$-bundle. More concretely, the blowup of $X_\flat$ along $\cU_1$ and $\cU_2$ admits a birational contraction onto a $\P^1$-bundle over the blowup $Y'_-$ of $Y_-^\flat$ along $Z_-^1$. Furthermore, denoting by $Y'_+$  the blowup of $Y_+^\flat$ along $Z_+^2$, we have an isomorphism $Y'_-\simeq Y'_+$ such that the birational map $\psi:Y_-^\flat\dashrightarrow Y_+^\flat$ factors as:
$$
\xymatrix@R=15pt{&Y'_-\simeq Y'_+\ar[ld]_{\mbox{\tiny blowup}}\ar[rd]^{\mbox{\tiny blowup}}&\\Y_-^\flat\ar@{-->}[rr]^\psi &&Y_+^\flat}
$$
\end{proposition}

\begin{proof}
Denote by $\alpha:X'_\flat \to  X_\flat$ the blowup of $X_\flat$ along $\cU_1$ and $\cU_2$, with exceptional divisors $\cU^\flat_1$ and $\cU^\flat_2$. Denote by $Y_-'$ and $Y_+'$ the strict transforms of the divisors $Y_-^\flat$ and $Y_+^\flat$ (see Figure \ref{fig:bw3}). 

\begin{figure}[H]
$
\xymatrix@R=20pt{\includegraphics[width=5cm]{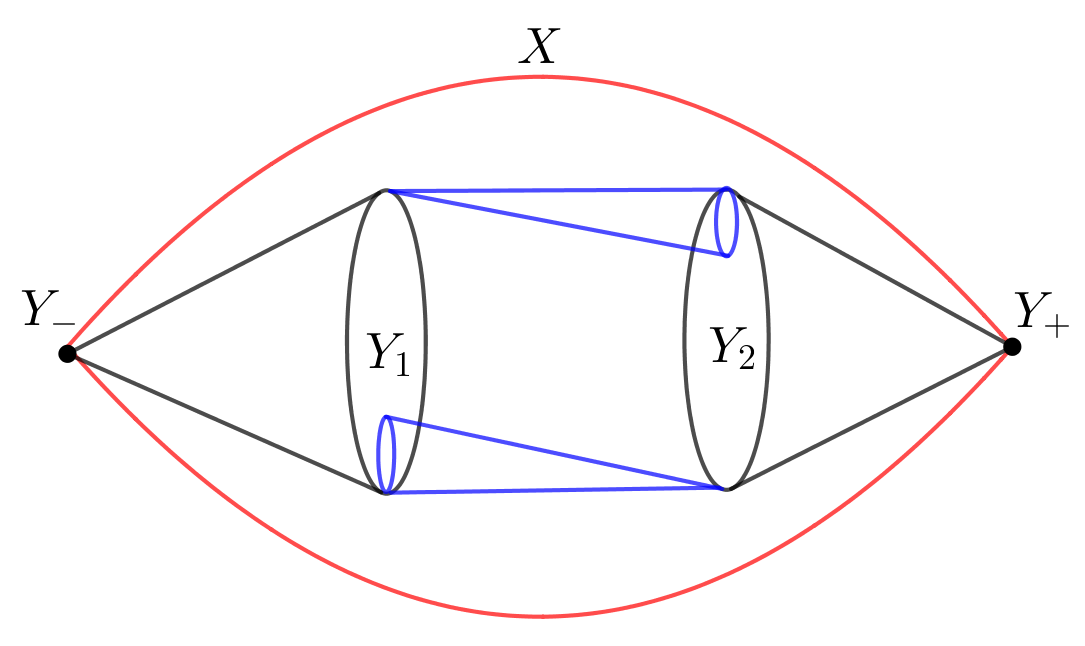}\ar@{-->}[d]&\includegraphics[width=4.5 cm]{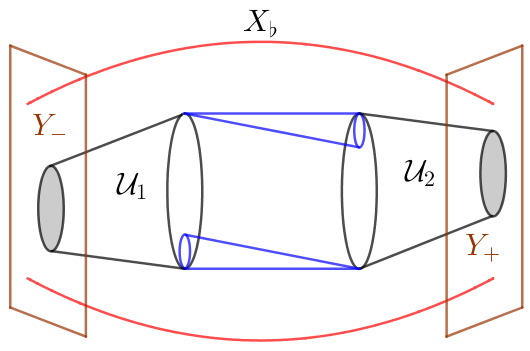}\ar[]!<0ex,10ex>;[l]!<0ex,10ex>_(.47){\varphi}\\
\includegraphics[width=5.4cm]{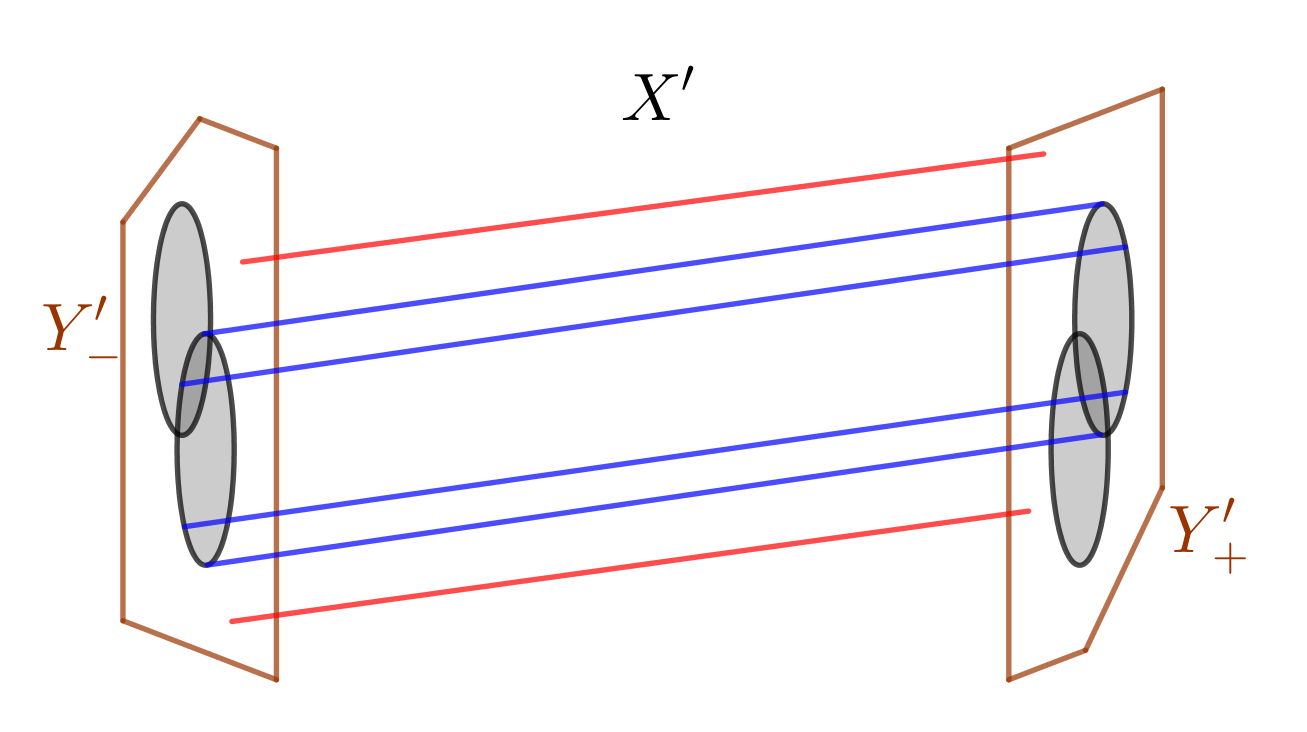}&\includegraphics[width=5cm]{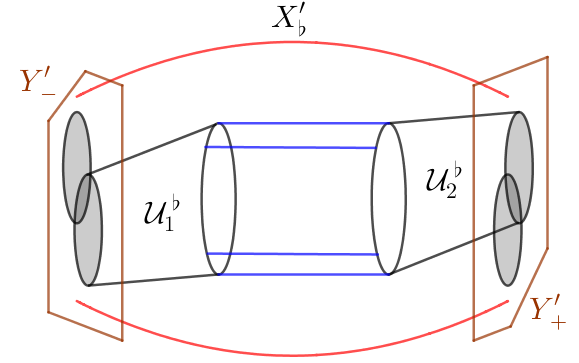}\ar[u]^(.85){\alpha}\ar[]!<0ex,10ex>;[l]!<0ex,10ex>_(.45){\varphi'}}
$
\caption{Birational transformation of a bandwidth three variety with isolated extremal fixed points into a $\P^1$-bundle.
}
\label{fig:bw3}
\end{figure}

Note that $\cU_1$ intersect $Y_-^\flat$ transversally along a variety isomorphic to $Y_1$ (see Lemma \ref{lem:BW3_exc_loc}). 
Hence the divisor $\cU^\flat_1$ is a $\P^1$-bundle over $\cU^\flat_1 \cap Y_-'$, which is the exceptional divisor of the blowup of $Y_-^\flat$ along the variety $Z_-^1 \simeq Y_1$. An analogous statement holds for $\cU_2^\flat$.

Now we will consider the closures $A_i,B_i,C_i$ of the orbits of the $H$-action in $X$, introduced at the beginning of the section, and denote their strict transforms into $X_\flat$ by the same symbols.  By Corollary \ref{cor:splittingTangent}, we have $N_{A_1/X} \simeq \cO(1)^{2m-2} \oplus \cO^{m+1}$, hence, by looking at the differential of $\varphi\colon X_\flat\to X$, we get  $N_{A_1/X_\flat} \simeq \cO^{2m-2} \oplus \cO(-1)^{m+1}$. Moreover from the sequence
$$ 0 \to N_{A_1/\cU_1} \to N_{A_1/X_\flat} \to N_{\cU_1/X_\flat}|_{A_1}\to 0$$
we get $N_{\cU_1/X_\flat}|_{A_1} \simeq \cO_{\P^1}(-1)^{m+1}$. 
Denote by $A'_{1}$ and $B'_{1}$ the minimal sections of $\cU_1^\flat$ and $\cU_2^\flat$ over curves of type $A_{1}$ and $B_{1}$; one can easily compute that $\cU^\flat_1 \cdot A'_{1}=\cU^\flat_2 \cdot B'_{1}=K_{X'_\flat}\cdot A'_{1} =K_{X'_\flat}\cdot B'_{1}=-1$.  By Nakano contractibility criterion (see \cite[Theorem 3.2.8]{BS})
there exists a smooth blowup $\varphi':X'_\flat \to X'$ contracting the curves of type $A'_{1}$ and $B'_{1}$ whose exceptional divisors are $\cU^\flat_1$ and $\cU^\flat_2$. 

Let us denote the strict transforms in $X'_\flat$ of the closure of orbits different from $A_1$ and $B_1$ by adding a prime. The orbit graph is the same as the graph (\ref{eq:gr1}) and, since the action is equalized, by Corollary \ref{cor:sum_of_orbits}, we have the following numerical equivalences of cycles: 
\begin{equation} C'_{3} \equiv A'_{1}+B'_{2} \equiv A'_{1}+C'_{1}+A'_{2} \equiv A'_{2}+B'_{1}. \label{eq:numeq}\end{equation}

The closure of the orbits in $X'$ are the images of the $1$-cycles $C'_{3}$, $B'_{2}$, $A'_{2}$ and
$C'_{1}$ which, by formula (\ref{eq:numeq}), are all numerically equivalent, since $A'_{1}$ and $B'_{1}$ are contracted by $\varphi'$. In particular, the induced $H$-action on $X'$ is a bordism of rank zero. Denoting again by $Y_-'$ and $Y_+'$ the images via $\varphi'$  of  $Y_-'$ and $Y_+'$, 
we can use Lemma \ref{lem:rank0bordism}
to get that $X'$ is a $H$-equivariant $\PP^1$-bundle over ${Y_-'}$. 

Finally, since in a similar way we get that $X'$ is  a $\PP^1$-bundle over ${Y_+'}$, we conclude that  ${Y_-'}$ is isomorphic to ${Y_+'}$, by means of an isomorphism that identifies the sink and the source of the $1$-dimensional orbits of the induced $H$-action. Hence, by the way in which $\psi$ has been defined, we get the commutativity of the diagram in the statement.
\end{proof}

Identifying $Y_-^\flat$ and $Y_+^\flat$ with $\P^{n-1}$, the birational map $\psi$ is given by a linear system whose base scheme is supported on $Z_-^1$. Since we have shown that $\psi$ is resolved by a single blowup along $Z_-^1$, it follows that the base scheme of the system is $Z_-^1$ with the reduced structure. A similar statement holds for $\psi^{-1}$, so  we may conclude that:

\begin{corollary}\label{cor:specialCremona}
The birational map $\psi:Y_-^\flat\dashrightarrow Y_+^\flat$ and its inverse are special Cremona transformations.
\end{corollary}

Let us recall the following Theorem, due to Ein and Shepherd-Barron (cf. \cite[Theorem 2.6]{ESB}), and the classification of Severi varieties proved by Zak (cf. \cite[Theorem 4.7]{Z}):

\begin{theorem}\label{thm:ESB}
Let $F:\P^{N}\dashrightarrow\P^N$ be a birational transformation, satisfying that its base locus scheme $Z$ is a connected nonempty smooth subvariety. Assume that $F$ and its inverse are both defined by linear systems of quadrics. Then $Z\subset \P^N$ is a projectively equivalent to a Severi variety:
\begin{enumerate}
\item $Z 
\simeq v_2(\P^{2})\subset \P^{5}$
 ($\dim(Z)=2$, Veronese surface);
\item $Z 
\simeq \P^{2}\times \P^{2}\subset \P^{8}$ ($\dim(Z)=4$, Segre variety);
\item $Z 
\simeq\G(1,5)\subset \P^{14}$ ($\dim(Z)=8$, Grassmann variety); 
\item $Z \simeq \DE_6(1)\subset \P^{26}$ ($\dim(Z)=16$, Cartan variety).
\end{enumerate}
\end{theorem}

We may now conclude the proof of Theorem \ref{thm:bandwidth3->Cremona}. Let us identify $Y_-^\flat$ and $Y_+^\flat$ with $\P^{n-1}$, and denote by $Y:=Y'_-\simeq Y'_+$ the common resolution of  $\psi$ and $\psi^{-1}$, obtained by blowing up the exceptional loci $Z_-^1, Z_+^2$:
$$\xymatrix{&Y\ar[ld]_{g}\ar[rd]^{f}&\\\P^{n-1}&&\P^{n-1}}
$$
Denote by $H_1$ (resp. $H_2$) the line bundle $g^*\cO_{\P^{n-1}}(1)$ (resp. $f^*\cO_{\P^{n-1}}(1)$) and by $E_1$ the exceptional divisor of $g$.

\begin{proof}[Proof of Theorem \ref{thm:bandwidth3->Cremona}]
By Corollary  \ref{cor:specialCremona} the map $\psi$ is a special Cremona transformation.
Denote by $(m_1,m_2)$ its type, i.e., $m_1$ and $m_2$ are the degrees of the linear systems defining $\psi$ and $\psi^{-1}$, respectively. 
Then (see \cite[Equation~(2.0.3)]{ESB}) $H_2 \sim m_1H_1-E_1$. In particular the strict transform $\tilde \ell$ of a $m_1$-secant line to $Z_-^1$ is contracted by $f$. Since $f$ is an elementary contraction and $H_1 \cdot \tilde \ell=1$
the numerical class of $\tilde \ell$ generates the extremal ray contracted by $f$, which has length equal to the codimension of $Z_+^2\simeq Y_2$ (cf. Lemma \ref{lem:BW3_exc_loc}) in $\P^{n-1}$ minus one, i.e., to $n-1-(2m-2)-1
=m$.
On the other hand, computing the anticanonical bundle of $Y$ we get 
$$-K_Y \cdot \tilde \ell = nH_1\cdot \tilde \ell -mE_1 \cdot \tilde \ell = m(3-m_1),$$
which implies that $m_1=2$. Repeating the argument with $\psi^{-1}$ we get also that $m_2=2$, and we conclude by Theorem \ref{thm:ESB}.
\end{proof}

\subsection{Bandwidth three varieties with Severi inner components}

In this section we will complete the classification of the polarized pairs admitting an equalized $H$-action of bandwidth three with isolated extremal fixed points (see Theorem  \ref{thm:bw3}). We will start by proving the following:

\begin{proposition}\label{prop:bw3class}
Let $(X,L)$ be a polarized pair as in Theorem \ref{thm:bw3} $(3)$. Set $n:=3m$, $m\geq 1$. Then $m=2,3,5,9$ and, for each $m$ the variety $X$ is unique.
\end{proposition}

We will use here the notation introduced in subsection \ref{ssec:Cremona_Severi}. In particular, $Y_-$ and $Y_+$ will be the extremal fixed points of the $H$-action, and $Y_1$, $Y_2$ the correspondent inner fixed point components, which are isomorphic Severi varieties by Theorem \ref{thm:bandwidth3->Cremona}. From the same theorem, we know that $\dim X=3m$, $m=2,3,5,9$, and the exceptional divisors $Y_-^\flat,Y_+^\flat$ in the blowup $ X_\flat$ of $X$ along $Y_-,Y_+$ are isomorphic to $\P^{3m-1}$. The birational maps $\psi:Y_-^\flat\dashrightarrow Y_+^\flat$ and $\psi^{-1}$ can be resolved by blowing up the base schemes $Z_-^1\subset Y_-^\flat$, and $Z_+^2 \subset Y_+^\flat$, which by Lemma \ref{lem:BW3_exc_loc} are isomorphic to  $Y_1,Y_2$, that, as we said above, are Severi varieties of dimension $2m-2$. The common resolution is denoted by $Y$. 

Let us denote by $H_1$ and $H_2$ the pullbacks to $Y$ of the hyperplane line bundles via the two blowups. Since the two contractions of $Y$ are blowups of smooth varieties of dimension $2m-2$ we have 
\begin{equation}\label{eq:canY}
-K_Y = m(H_1+H_2).
\end{equation}

\begin{proof}[Proof of Proposition \ref{prop:bw3class}]
In the proof of Proposition \ref{prop:birBW3}  we have shown that $X$ is birationally equivalent to a $\P^1$-bundle $X'$, by means of a precise sequence of smooth blowups and blowdowns (see Figure \ref{fig:bw3}):
$$
X\stackrel{\varphi}{\longleftarrow}X_\flat \stackrel{\alpha}{\longleftarrow}X'_\flat\stackrel{\varphi'}{\longrightarrow} X'.
$$
The $\PP^1$-bundle $X'$ constructed in Proposition \ref{prop:birBW3} can be described as the projectivization of the rank two vector bundle:
$$\cE:= \cO\oplus N_{Y_-'/X'}\simeq \cO\oplus \cO_{Y_-'}(Y_-') \simeq  \cO \oplus \cO(K_{Y_-'}-K_{X'}|_{Y_-'}).$$
The extremal rays of $Y_-' \simeq Y$ are generated by the class $[C_\alpha]$ of a minimal curve contracted by $\alpha$ and by the class $[C_s]$ of a strict transform of a (bi)secant line of $\cU_1 \cap Y_-^\flat$ in $Y_-^\flat$. 
By formula (\ref{eq:canY}) we have $-K_{Y'_0} \cdot [C_\alpha]=-K_{Y'_0} \cdot [C_s]=m$. On the other hand, since
$$-K_{X'_\flat} = -\alpha^*K_{X_\flat}-m(\cU_1^\flat+\cU^\flat_2)=-\varphi'^*K_{X'} -(\cU_1^\flat+\cU^\flat_2)$$
we can compute that   $-K_{X'} \cdot [C_\alpha]=m-1$, $-K_{X'} \cdot [C_s]=m+1$. It follows that   $K_{Y_-'}-K_{X'}|_{Y_-'}\simeq H_1-H_2$.

This shows that $X'$, and consequently $X$, is uniquely determined by the resolution of the Cremona transformation $\psi$, for a choice of the fixed components
$Y_1$ and $Y_2$ among the varieties listed in Theorem \ref{thm:ESB}.
\end{proof}

\begin{theorem}\label{thm:BW3end}
Let $(X,L)$ be a polarized pair as in Theorem \ref{thm:bw3} $(3)$. Then $X$ is one of the following rational homogeneous varieties:
$$\LG(2,5)\simeq\DC_3(3),\,\,\, \G(2,5)\simeq \DA_5(3),\,\,\, S^{15}\simeq \DD_6(6),\,\,\, \DE_7(7),$$
and $L$ is the ample generator of the Picard group.
\end{theorem}
\begin{proof}
In view of Theorem \ref{thm:bw3} and Proposition \ref{prop:bw3class}, it is enough to show that the varieties listed in the statement admit an equalized $H$-action of bandwidth three with isolated extremal fixed points. 

Given one of the varieties above, we write it as $G/P$, where $G$ is a semisimple group of type $\DC_3,\DA_5,\DD_6$ or $\DE_7$, and use the notation introduced in subsection \ref{ssec:notRH}. We consider the action of a maximal torus $T\subset P$ on $G/P$, whose fixed points are known to be the elements of the form $wP$, $w\in W$ (see \cite[Lemma 3.12]{CARRELL}). We can see $G/P$ as the minimal orbit of a projective representation $\P(H^0(G/P,L_i))$, where $L_i$ is the homogeneous ample line bundle associated with the fundamental weight $\lambda_i$ ($i=3,3,6,7$, respectively), cf. \cite[Claim 23.52]{FH}, and then the fixed points correspond to the $T$-invariant subspaces $H^0(G/P,L_i)_{w\lambda_i}$ associated with the weights $w\lambda_i\in\Mo(T)$. On the other hand, we have a weight decomposition of the action of $T$ on the tangent bundle at each of these points: $$T_{G/P,wP}=\bigoplus_{\beta\in\Phi^+(D\setminus \{i\})}\fg_{w(-\beta)}.$$

Now we consider, in each case, the $1$-parameter subgroup $\mu_i:H=\C^*\to T$ given by $\mu_i(\lambda_j)=\delta_{ij}$, for every $j$. It is then a straightforward computation (that one can  perform with SageMath, for instance) to check that the induced $H$-action satisfies the required properties. 
\end{proof}


\bibliographystyle{plain}
\bibliography{bibliomin}

\begin{thebibliography}{10}

\bibitem{Akh}
Dmitri~N. Akhiezer.
\newblock {\em Lie group actions in complex analysis}.
\newblock Aspects of Mathematics, E27. Friedr. Vieweg \& Sohn, Braunschweig,
  1995.

\bibitem{AO2}
Marco Andreatta and Gianluca Occhetta.
\newblock Special rays in the {M}ori cone of a projective variety.
\newblock {\em Nagoya Math. J.}, 168:127--137, 2002.

\bibitem{CoxRings}
Ivan Arzhantsev, Ulrich Derenthal, J\"{u}rgen Hausen, and Antonio Laface.
\newblock {\em Cox rings}, volume 144 of {\em Cambridge Studies in Advanced
  Mathematics}.
\newblock Cambridge University Press, Cambridge, 2015.

\bibitem{BaRo}
Lorenzo Barban and Eleonora~A. Romano.
\newblock Toric non-equalized flips associated to $\mathbb{C}^*$-actions.
\newblock {\em ArXiv preprint:2104.14442}, 2021.

\bibitem{BS}
Mauro~C. Beltrametti and Andrew~J. Sommese.
\newblock {\em The adjunction theory of complex projective varieties},
  volume~16 of {\em De Gruyter Expositions in Mathematics}.
\newblock Walter de Gruyter \& Co., Berlin, 1995.

\bibitem{BB}
Andrzej Bia{\l}ynicki-Birula.
\newblock Some theorems on actions of algebraic groups.
\newblock {\em Ann. of Math. (2)}, 98:480--497, 1973.

\bibitem{BBS}
Andrzej Bia{\l}ynicki-Birula and Joanna \'Swi{\c e}cicka.
\newblock Complete quotients by algebraic torus actions.
\newblock In {\em Group actions and vector fields ({V}ancouver, {B}.{C}.,
  1981)}, volume 956 of {\em Lecture Notes in Math.}, pages 10--22. Springer,
  Berlin, 1982.

\bibitem{BWW}
Jaros{\l}aw Buczy{\'n}ski, Jaros{\l}aw~A. Wi{\'s}niewski, and Andrzej Weber.
\newblock Algebraic torus actions on contact manifolds.
\newblock {\em {\rm To appear in} Journal of Differential Geometry; arXiv
  preprint:1802.05002}, 2018.

\bibitem{CARRELL}
James~B. Carrell.
\newblock Torus actions and cohomology.
\newblock In {\em Algebraic quotients. Torus actions and cohomology. The
  adjoint representation and the adjoint action}, volume 131, pages 83--158.
  Springer, 2002.

\bibitem{CaSo}
James~B. Carrell and Andrew~John Sommese.
\newblock Some topological aspects of {${\bf C}^{\ast} $} actions on compact
  {K}aehler manifolds.
\newblock {\em Comment. Math. Helv.}, 54(4):567--582, 1979.

\bibitem{CasciniLazic}
Paolo Cascini and Vladimir Lazi\'{c}.
\newblock New outlook on the minimal model program, {I}.
\newblock {\em Duke Math. J.}, 161(12):2415--2467, 2012.

\bibitem{CortiLazic}
Alessio Corti and Vladimir Lazi\'{c}.
\newblock New outlook on the minimal model program, {II}.
\newblock {\em Math. Ann.}, 356(2):617--633, 2013.

\bibitem{CLS}
David~A. Cox, John~B. Little, and Henry~K. Schenck.
\newblock {\em Toric varieties}, volume 124 of {\em Graduate Studies in
  Mathematics}.
\newblock American Mathematical Society, Providence, RI, 2011.

\bibitem{De}
Olivier Debarre.
\newblock {\em Higher-dimensional algebraic geometry}.
\newblock Universitext. Springer-Verlag, New York, 2001.

\bibitem{ESB}
Lawrence Ein and Nicholas~I. Shepherd-Barron.
\newblock Some special {C}remona transformations.
\newblock {\em Amer. J. Math.}, 111(5):783--800, 1989.

\bibitem{Fult}
William Fulton.
\newblock {\em Intersection theory}, volume~2 of {\em Ergebnisse der Mathematik
  und ihrer Grenzgebiete. 3. Folge. A Series of Modern Surveys in Mathematics
  [Results in Mathematics and Related Areas. 3rd Series. A Series of Modern
  Surveys in Mathematics]}.
\newblock Springer-Verlag, Berlin, second edition, 1998.

\bibitem{FH}
William Fulton and Joe Harris.
\newblock {\em Representation theory}, volume 129 of {\em Graduate Texts in
  Mathematics}.
\newblock Springer-Verlag, New York, 1991.

\bibitem{HuKeel}
Yi~Hu and Sean Keel.
\newblock Mori dream spaces and {GIT}.
\newblock volume~48, pages 331--348. 2000.
\newblock Dedicated to William Fulton on the occasion of his 60th birthday.

\bibitem{Ito}
Atsushi Ito.
\newblock Algebro-geometric characterization of {C}ayley polytopes.
\newblock {\em Advances in Mathematics}, 270:598 -- 608, 2015.

\bibitem{IVERSEN}
Birger Iversen.
\newblock A fixed point formula for action of tori on algebraic varieties.
\newblock {\em Inventiones mathematicae}, 16(3):229--236, 1972.

\bibitem{Kan}
Akihiro Kanemitsu.
\newblock Extremal {R}ays and {N}efness of {T}angent {B}undles.
\newblock {\em Michigan Math. J.}, 68(2):301--322, 2019.

\bibitem{KKLV}
Friedrich Knop, Hanspeter Kraft, Domingo Luna, and Thierry Vust.
\newblock {\em Local Properties of Algebraic Group Actions}, pages 63--75.
\newblock Birkh{\"a}user Basel, Basel, 1989.

\bibitem{Morelli}
Robert {Morelli}.
\newblock {The birational geometry of toric varieties.}
\newblock {\em {J. Algebr. Geom.}}, 5(4):751--782, 1996.

\bibitem{WORS2}
Gianluca Occhetta, Luis~E. Sol{\'a}~Conde, Eleonora~A. Romano, and Jaros\l
  aw~A. Wi\'sniewski.
\newblock High rank torus actions on contact manifolds.
\newblock {\em Selecta Math.}, 27(10), 2021.

\bibitem{pas}
Boris Pasquier.
\newblock On some smooth projective two-orbit varieties with {P}icard number 1.
\newblock {\em Math. Ann.}, 344(4):963--987, 2009.

\bibitem{CPS}
Victor Przyjalkowski, Ivan Cheltsov, and Constantin Shramov.
\newblock Fano threefolds with infinite automorphism groups.
\newblock {\em Izv. Ross. Akad. Nauk Ser. Mat.}, 83(4):226--280, 2019.

\bibitem{ReidToric}
Miles {Reid}.
\newblock {Decomposition of toric morphisms.}
\newblock {Arithmetic and geometry, Pap. dedic. I. R. Shafarevich, Vol. II:
  Geometry, Prog. Math. 36, 395-418 (1983).}, 1983.

\bibitem{ReidFlip}
Miles Reid.
\newblock What is a flip? 

\bibitem{RW}
Eleonora~A. Romano and Jaros{\l}aw~A. Wi{\'s}niewski.
\newblock Adjunction for varieties with a $\mathbb{C}^*$ action.
\newblock {\em Transf. Groups}, 2020.
\newblock https://doi.org/10.1007/s00031-020-09627-8.

\bibitem{Tev05}
Evgueni~A. Tevelev.
\newblock {\em Projective duality and homogeneous spaces}, volume 133 of {\em
  Encyclopaedia of Mathematical Sciences}.
\newblock Springer-Verlag, Berlin, 2005.
\newblock Invariant Theory and Algebraic Transformation Groups, IV.

\bibitem{Thaddeus1994}
Michael Thaddeus.
\newblock Toric quotients and flips.
\newblock In {\em Topology, geometry and field theory}, pages 193--213. World
  Sci. Publ., River Edge, NJ, 1994.

\bibitem{Thaddeus1996}
Michael Thaddeus.
\newblock Geometric invariant theory and flips.
\newblock {\em J. Amer. Math. Soc.}, 9(3):691--723, 1996.

\bibitem{JW-Toric}
Jaros{\l}aw~A. {Wi\'sniewski}.
\newblock {Toric Mori theory and Fano manifolds.}
\newblock In {\em {Geometry of toric varieties. Lectures of the summer school,
  Grenoble, France, June 19--July 7, 2000}}, pages 249--272. Paris: Soci\'et\'e
  Math\'ematique de France, 2002.

\bibitem{Wlodarczyk}
Jaros{\l}aw {W{\l}odarczyk}.
\newblock {Birational cobordisms and factorization of birational maps.}
\newblock {\em {J. Algebr. Geom.}}, 9(3):425--449, 2000.

\bibitem{Z}
Fyodor~L. Zak.
\newblock {\em Tangents and secants of algebraic varieties}, volume 127 of {\em
  Translations of Mathematical Monographs}.
\newblock American Mathematical Society, Providence, RI, 1993.

\end{thebibliography}
\end{document}